\documentclass{article}
\usepackage{graphicx} 
\usepackage{framed}
\usepackage{fancybox}
\usepackage{amsfonts}
\usepackage{amsmath,amssymb,amsthm}
\usepackage{hyperref}
\usepackage{cleveref}
\usepackage{todonotes}
\usepackage{yfonts}
\usepackage{tikz-cd}
\usepackage{comment}
\usepackage[ruled,vlined]{algorithm2e}

\theoremstyle{plain}
\newtheorem{theoremintro}{Theorem}

\newtheorem{corintro}[theoremintro]{Corollary}
\newtheorem{theorem}{Theorem}[section]
\newtheorem{lemma}[theorem]{Lemma}
\newtheorem{proposition}[theorem]{Proposition}
\newtheorem{corollary}[theorem]{Corollary}
\newtheorem{remark}[theorem]{Remark}

\theoremstyle{definition}
\newtheorem{definition}[theorem]{Definition}
\newtheorem{example}[theorem]{Example}

\newcommand{\Id}{\mathrm{Id}}

\renewcommand{\epsilon}{\ensuremath{\varepsilon}}
\newcommand{\CC}{\ensuremath{\mathbb{C}}} 
 
\newcommand{\kk}{\ensuremath{K}} 
\newcommand{\KK}{\ensuremath{\mathrm{K}}}

\newcommand{\QQ}{\ensuremath{\mathbb{Q}}}  
\newcommand{\NN}{\ensuremath{\mathbb{N}}}  
\newcommand{\RR}{\ensuremath{\mathbb{R}}} 
\newcommand{\ZZ}{\ensuremath{\mathbb{Z}}} 

\newcommand{\vspan}[1]{\langle #1 \rangle}
\newcommand{\cl}[1]{\mathcal{#1}}
\newcommand{\vb}[1]{\mathbf{#1}}
\newcommand{\bm}[1]{\mathbf{#1}}
\newcommand{\bms}[1]{\boldsymbol{#1}}

\newcommand{\hh}{{\rm h}}
\def\h{\mathfrak{h}}

\def\rng{\RR[\bm x]}
\def\kkrng{\kk[\bm x]}
\def\QQrng{\QQ[\bm x]}
\def\ZZrng{\ZZ[\bm x]}
\def\CCrng{\CC[\bm x]}
\def\mod{\,\textup{mod}\,}
\def\dist{\,\textup{dist}\,}

\def\ib{\bm i\,}

\def\diag{\textup{diag}}
\def\Chow{\textup{Ch}}

\newcommand{\cst}{{\cl C}}
\def\Cn{{\cl C}(n)\,}
\def\C1{{\cl C}(1)\,}
\def\Cnd{{\cl C}(n; d)\,}

\newcommand{\wtau}{\tau_p}
\newcommand{\otau}{\eta_p}
\newcommand{\wf}{p}

\makeatletter
\newif\if@borderstar
\def\bordermatrix{\@ifnextchar*{%
  \@borderstartrue\@bordermatrix@i}{\@borderstarfalse\@bordermatrix@i*}%
}
\def\@bordermatrix@i*{\@ifnextchar[{%
  \@bordermatrix@ii}{\@bordermatrix@ii[()]}
}
\def\@bordermatrix@ii[#1]#2{%
  \begingroup
    \m@th\@tempdima8.75\p@\setbox\z@\vbox{%
      \def\cr{\crcr\noalign{\kern 2\p@\global\let\cr\endline }}%
      \ialign {$##$\hfil\kern 2\p@\kern\@tempdima & \thinspace %
      \hfil $##$\hfil && \quad\hfil $##$\hfil\crcr\omit\strut %
      \hfil\crcr\noalign{\kern -\baselineskip}#2\crcr\omit %
      \strut\cr}}%
    \setbox\tw@\vbox{\unvcopy\z@\global\setbox\@ne\lastbox}%
    \setbox\tw@\hbox{\unhbox\@ne\unskip\global\setbox\@ne\lastbox}%
    \setbox\tw@\hbox{%
      $\kern\wd\@ne\kern -\@tempdima\left\@firstoftwo#1%
        \if@borderstar\kern2pt\else\kern -\wd\@ne\fi%
      \global\setbox\@ne\vbox{\box\@ne\if@borderstar\else\kern 2\p@\fi}%
      \vcenter{\if@borderstar\else\kern -\ht\@ne\fi%
        \unvbox\z@\kern-\if@borderstar2\fi\baselineskip}%
        \if@borderstar\kern-2\@tempdima\kern2\p@\else\,\fi\right\@secondoftwo#1 $%
    }\null \;\vbox{\kern\ht\@ne\box\tw@}%
  \endgroup
}
\makeatother

\title{An Effective Positivstellensatz over the Rational Numbers \\ for Finite Semialgebraic Sets}
\author{ Lorenzo Baldi$^{1}$ \& Teresa Krick$^2$ \& Bernard Mourrain$^3$ }

\date{%
    {\small
    $^1$ Universit\"at Leipzig \& Max Planck Institute for Mathematics in the Sciences, Leipzig, Germany\\%
    $^2$ Departamento de Matemática \& IMAS-Conicet, Universidad de Buenos Aires, Argentina\\%
    $^3$ Centre Inria d’Universit\'e Côte d’Azur, Sophia Antipolis, France\\[2ex]%
    }
    \today}

\begin{document}

\maketitle
\vspace{-0.7cm}
\begin{center}
    \itshape{To the memory of our dear friend Agnes Szanto}
\end{center}

\begin{abstract}

We study the problem of representing multivariate polynomials with rational coefficients, which are nonnegative and strictly positive on finite semialgebraic sets, using rational sums of squares.
We focus on the case of finite semialgebraic sets $S$ defined by equality constraints, generating a zero-dimensional ideal $I$, and by nonnegative sign constraints.

First, we obtain existential results. We prove that a strictly positive polynomial $f$ with coefficients in a subfield $\kk$ of $\RR$ has a representation in terms of  weighted Sums-of-Squares with coefficients in this field, even if the ideal $I$ is not radical. We generalize this result to the case where $f$ is nonnegative on $S$ and $(f) + (I:f)=1$. 
We deduce that nonnegative polynomials with  coefficients in $\kk$ can be represented in terms of Sum-of-Squares of polynomials with coefficients in $\kk$, when the ideal is radical.

Second, we obtain degree bounds for such Sums-of-Squares representations, which depend linearly on the regularity of the ideal and the degree of the defining equations, when they form a graded basis.

Finally, we analyze the bit complexity of the Sums-of-Squares representations for polynomials with  coefficients in $\QQ$, in the case the ideal is radical.
The bitsize bounds are linear in the bitsize of the input polynomials, quadratic or cubic in the Bezout bound, and linear in the regularity, generalizing and improving previous results obtained for special zero dimensional ideals.

As an application in the context of polynomial optimization, we retrieve and improve results  on the finite convergence and exactness of the moment/Sums-of-Squares hierarchy.
\end{abstract} 

\tableofcontents
\section{Introdution}\label{sec:intro}
The certification of the nonnegativity of a (multivariate) polynomial with real coefficients is a central problem in real algebra. Many different theorems, known as {\em Positivstellens\"atze}, provide existential results on the representation of nonnegative polynomials in terms of Sums of Squares (SoS).  While not all globally nonnegative polynomial are sums of squares of polynomials, Artin \cite{artinUberZerlegungDefiniter1927} answered positively to Hilbert's 17th problem by showing that a representation as sums of squares of rational functions is always possible. This result was generalized by the Positivstellensätze of Krivine \cite{krivineAnneauxPreordonnes1964} and Stengle \cite{stengleNullstellensatzPositivstellensatzSemialgebraic1974} to polynomials that are nonnegative on arbitrary semialgebraic sets. Later on, the breakthrough results of Schm\"udgen \cite{schmudgenTheKmomentProblemCompact1991} and Putinar \cite{putinarPositivePolynomialsCompact1993} showed that general denominator-free SoS representations exist for strictly positive polynomials on compact basic semialgebraic sets. 

SoS certificates for nonnegative polynomials have many applications: for instance, 
they can be used to develop constructive mathematics, see e.g. \cite{LombardiQuitte2015}, or in theorem checker as positivity witnesses \cite{Harrison2007, Magron2014, DBLP:conf/cpp/Martin-DorelR17}. 

Another area of application, where SoS representations play an important role, is Polynomial Optimization (POP). Indeed, after the discovery of the connection between SoS and positive semidefinite matrices, due to Shor \cite{shorClassGlobalMinimum1987} and Choi, Lam and Reznick \cite{choiSumsSquaresReal1995}, SoS representations acquired an important role in optimization and numerical computation. This connection was exploited by Lasserre \cite{lasserreGlobalOptimizationPolynomials2001a} and Parrilo \cite{parriloStructuredSemidefinitePrograms2000} to construct the celebrated moment/SoS hierarchies, which provide arbitrarily tight semidefinite relaxations to polynomial optimization problems.
The spectrum of applications of polynomial optimization itself is huge (see e.g. \cite{Lasserre2009} and references therein).

In all these applications, controlling the degrees and sizes of the SoS representation is a bottleneck for the performance of the related algorithms: this is an active and challenging area of research, which is witnessing recent interesting developments. General degree bounds for Positivstellens\"atze are challenging to obtain and not proved to be tight: for instance the best known  bound for the Krivine-Stengle Positivstellensatz, presented in \cite{Lombardi2020}, consists of five towers of exponentials in the degree of the input polynomials and the number of variables, and similar bounds apply for Putinar's Positivstellensatz (see e.g. \cite{Baldi2025, baldiDegreeBoundsPutinar2024} and references therein). 

It is therefore natural to analyze simpler special cases, still useful for applications, where we can expect much precise degree estimates. For instance, Lasserre studied in \cite{lasserrePolynomialsNonnegativeGrid2002}  the case of  strictly positive polynomials over a finite grid of real points, and then Parrilo considered in \cite{Parrilo2002} the more general case of nonnegative polynomials over a finite basic semialgebraic set, assuming a radicality assumption in the description of the set.

It is also natural to ask in which field we can take the coefficients of the SoS representation, if the positive polynomial has coefficients in a subfield $\kk$ of $\RR$, in particular when $K = \QQ$. 
Although the existence of rational SoS representations is known for globally nonnegative, univariate polynomials with rational coefficients \cite{landauUberDarstellungDefiniter1906,pourchet} and for strictly positive polynomials on compact basic semialgebraic sets in the multivariate case (mostly under some additional assumptions \cite{Schweighofer2002, Powers2011, Magron2021, Davis2022}), this is not always true for globally nonnegative multivariate polynomials: Scheiderer \cite{scheidererSumsSquaresPolynomials2016}  showed that in the multivariate case there exist rational polynomials that are nonnegative on $\RR^n$ and admit a real SoS representation, but not a rational one.

\paragraph{The zero-dimensional setting.}
In this paper, we will consider {finite  semialgebraic sets}, focusing in particular on  the existence of {exact effective}  SoS representations. 
Instead of working with the specific field $\QQ$ of rational numbers, we study polynomials with coefficients in a subfield $K$ of $\RR$, that are strictly positive (or nonnegative) on a finite basic closed semialgebraic set $S\subset \RR^{n}$ defined by polynomials with coefficients in $K$. 
We aim at analyzing the smallest field in which we can compute {an SoS} representation and at bounding the degree and the height (when $K = \QQ$) of the coefficients of this SoS representation. 
{In the case $K  = \RR$, these representations are well-studied, see \cite{Parrilo2002} and \Cref{sec:applications}. The main technical difficulty in generalizing these results to subfields of $\RR$ is to ensure a certain \emph{strict feasibility} condition for the matrices associated with the SoS representation: we refer to \Cref{subs:radical} for more details.}

The precise problems we consider here are the following.
Given a field $K\subset \RR$, let \begin{equation}\label{eq:gh} \bm g = \{ \, g_1, \dots g_r \, \}, \ \bm h = \{ \, h_1, \dots h_s \, \} \  \subset \ \kkrng := \kk[x_1, \dots x_n],\end{equation} be two  sets of polynomials in $n$ variables with coefficients in $\kk$, where the ideal $I=(\bm h)\subset \kkrng$ generated by $\bm h$ is zero-dimensional, {i.e. with finitely many complex solutions}. The polynomials $\bm g$ and $\bm h$ then define the {\em finite} (basic, closed) semialgebraic set  
\begin{equation} \label{eq:s}
S = S(\vb g,\vb h) = \{ \, \xi \in \RR^n
\colon g_i(\xi) \ge 0   \textup{ for } 1\le i\le r   \textup{ and }  h_j(\xi)=0 \textup{ for } 1\le j\le s\, \} \ \subset \RR^n.
\end{equation}
Given a polynomial $f\in \kkrng$, we study the following problems.

\begin{enumerate} 
    \item If $f(\xi) > 0$ for all $ \xi\in S$, how can we represent $f$ as
       \begin{equation*}
        f = \sum_{k} \omega_{0,k}\, q_{0,k}^2 + \sum_{i} \Big( \sum_{k} \omega_{i,k}\, q_{i,k}^2 \Big) \, g_i + \sum_{j} p_j \, h_j    
     \end{equation*}
    with $\omega_{i,k}\in \kk_{\ge 0} {:= \{ \, \omega \in \kk \colon \omega \ge 0 \, \}}$ and $q_{i,k}, p_{j}\in \kkrng$, 
   and  what are the degrees of $q_{i,k}, p_{j}$?

\item If $f(\xi) \geq 0$ for all $ \xi\in S$, which  condition  guarantees that $f$ can be represented as an SoS as above, and  what are the degrees of $q_{i,k}$ and $p_{j}$? 

\item  When $K= \QQ$, what is the bitsize of the $\omega_{i,k}$ and the  coefficients of $q_{i,k}, p_{j}$?  What is the cost to compute such a representation?
\end{enumerate}

Before stating our main results, we recall that an ideal $I\subset \kkrng$ generated by $\bm h=\{h_1,\dots,h_s\}$ is a zero-dimensional ideal  when 
 $$V_{\CC}:=\{\zeta \in \CC^n\ : \ h_j(\zeta)=0  \mbox{ for }  1\le j\le s\}$$ is a finite set, 
or equivalently, when the quotient ring $\kkrng/I$ is a finite dimensional $\kk$-vector space.  In the sequel we will denote by $B$ a  monomial basis of $\kkrng/I$, and by $\langle B\rangle_\kk$, the vector space spanned by $B$ in $\kkrng$. 
We also recall that 
$\bm h$ is a \emph{graded basis} of the ideal $I$ when for all $p\in I$, there exist $p_j\in \kkrng$ with $\deg(p_j)\le \deg(p)-\deg(h_j)$, $1\le j\le s$, such that 
$ 
p= \displaystyle \sum_{j=1}^{s}  p_j \, h_{j}
$. {Finally, for $f \in \kkrng$ we denote $(I \colon f) = \{ \, h \in \kkrng \colon fh \in I \,\}$ the colon or quotient ideal.}\\

We can now state our main theorems. 
\begin{theoremintro}[{See Theorems \ref{thm:epss:finite} and \ref{thm:nonneg}}]
    \label{thm:A}
    Let $\kk\subset \RR$ be a field, 
   $\bm g,  \bm h  \subset  \kkrng $ be as in \eqref{eq:gh}, $I = (\vb h) \subset \kkrng$  be a {\em zero-dimensional} ideal, $B$  be a monomial basis of $\kkrng \big / I$ with $|B|=D$, and
 $S\subset \RR^n$ be the finite basic closed semialgebraic set defined in \eqref{eq:s}.
 
If $f\in \kkrng$ is such that $f>0$ on $S$, or  $f \ge 0$ on $S$ and $(I \colon f) + (f) = (1)$, then
 there exist 
\begin{itemize}
    \item $\omega_{i,k}\in \kk_{\ge 0}$ and $q_{i,k} \in \vspan{B}_K$ for $0\le i\le r$ and $1\le k\le D$,
    \item $p_j \in \kkrng$ for $1\le j\le s$ 
\end{itemize}
such that 
     \begin{equation}\label{eq:main_result}
        f = \sum_{k=1}^D \omega_{0,k}q_{0,k}^2 + \sum_{i=1}^r \left( \sum_{k=1}^D \omega_{i,k}q_{i,k}^2 \right) \, g_i + \sum_{j=1}^s p_j \, h_j    
     \end{equation}
Furthermore, if $\vb h$ is a graded basis of $I$ and $\deg(B)$ is an upper bound for the maximum degree of a monomial in $B$, then for all $j$
$$\displaystyle \deg (p_j h_j)  \le \max\{ \deg(f), \deg(g_i)+2\deg(B) \,\colon \,1\le i\le r\}.$$
\end{theoremintro}

Note that the smallest possible degree of a basis $B$ of $\kkrng/I$ is  the \emph{Castelnuovo-Mumford regularity} $\textup{reg}(\kkrng/I) $ of $\kkrng \big / I$ (see e.g. \cite{BayerStillman1987}), which  can always be upper bounded by  $D =\dim (\kkrng/I)$,  by choosing a suitable monomial basis $B$ of  $\kkrng/I$, which itself satisfies $D\le\displaystyle{\max_j}\{\deg(h_j)\}^n$ by B\'ezout 
bound. Moreover, when $\bm h$ is a graded basis of degree $\le d$ generating the zero-dimensional ideal $I$, we have $ \textup{reg}(\kkrng/I) \le n (d-1)$ (see \cite[Theorem 9.4]{Chardin2007} and \cite[Appendix, Prop.~A.9 and  Lem.~A.10]{Szanto2008}). 

Since the condition $(I \colon f) + (f) = (1)$ is always satisfied when the ideal $I$ is radical,  we can deduce  a complete characterization of nonnegativity of polynomials on $S$ defined by a radical zero-dimensional ideal $I$ (see \Cref{rem:iff_nonnegative_radical}).

\begin{corintro}
    \label{thm:B}
    Let $\kk$, $I$ and $S$ be as in \Cref{thm:A}, and assume that $I$ is a {\em radical} ideal. Then for $f \in \kkrng$,  $f \ge 0$ on $S$ if and only if $f$ admits {an SoS} representation as in \eqref{eq:main_result}.
\end{corintro}

We note that since every Archimedean field $\kk$ (i.e. a field where for every $x\in \kk$, there exists $n\in \NN$ such that $x<n$) can be embedded as a subfield of $\RR$ by H\"older's embedding theorem (see e.g. \cite[Th.~2.1.10]{knebuschRealAlgebraFirst2022}), \Cref{thm:A} and \Cref{thm:B} hold for any Archimedean field $K$. 

When $K=\QQ$, we can assume furthermore that  the polynomials $q_{i,k}$ in  \Cref{eq:main_result} belong to the free $\ZZ$-module $\vspan{B}_{\ZZ}$ generated by the monomial basis $B$, by taking a common denominator of their  coefficients and dividing the weights $\omega_{i,k}$ by the square of this common denominator. Our next result deals with the size of these coefficients in the particular case of a radical zero-dimensional ideal $I$.

\begin{theoremintro}[{See Theorems \ref{thm:f2 strict pos} and \ref{thm:Thethm2}}]
    \label{thm:C}
Let  $f, g_1, \dots , g_r,  h_1, \dots , h_s$ be polynomials in $ \ZZrng $ of degrees bounded by $d$ and heights  bounded by $\tau$.
Assume that  $I = (\vb h) \subset \QQrng$ is a radical zero-dimensional ideal, and let $S$ and $B$ be as in \Cref{thm:A}, with 
 $\delta:= \max\{d,\deg(B)\}$.
 
If $f>0$ on $S$ (resp. $f\ge 0$ on $S$), 
then the following holds for {an SoS} representation as in \eqref{eq:main_result} of $f$:
  \begin{itemize}\item
  $\omega_{0,k}\in \QQ_{\ge 0}$
    and $q_{0,k}\in \vspan{B}_{\ZZ}$ (resp. $q_{0,k}\in f\cdot \vspan{B}_{\ZZ}$) with  heights bounded by $\widetilde {\cl O}(d^{3n})\,\delta\, \tau$; \item  $\omega_{i,k}\in \QQ_{\ge 0}$ and $q_{i,k}\in \vspan{B}_{\ZZ}$ (resp. $q_{i,k}\in f\cdot \vspan{B}_{\ZZ}$) with heights bounded by $\widetilde{\cl O}(d^{2n})\, \delta \, \tau$ for $1\le i\le r$;
    \end{itemize}
If in addition, $\bm h=\{h_1,\dots,h_s\}$ is a graded basis of $I$, then $\deg(q_{i,j}) \le n d$ (resp. $\deg(q_{i,j})\le (n+1) d$), and $p_j\in \QQrng$ with degrees bounded by $\widehat d-\deg(h_j)$ and heights bounded by $\widetilde{\cl O} (d^{2n})\, \delta\, \tau + \widetilde{\cl O}({\widehat d}^{\,n}) \tau$, for $1\le j\le s$, where  $\widehat d: = 2(d + \deg(B))+1$. 
\end{theoremintro}

In \Cref{thm:C} the (logarithmic) {\em height} of a polynomial $p\in \ZZrng$ is the maximum base 2-logarithm of the absolute values  of its coefficients, and the height  of a  rational a polynomial $p\in \QQrng$ is  the maximum height of the numerator and a common denominator in a primitive representation $p=\frac 1 \nu \widehat p$ with $\widehat p\in \ZZrng$ and $\nu\in \NN$. 
The notation $\widetilde{\cl O}(M)$ means to be bounded by a linear function of $M$ up to logarithmic factors of $M$. 

The bounds in \Cref{thm:C} involve B\'ezout type bounds  derived from \cite{KPS01} for the height of the Chow form of $V_{\CC}$.
These bounds could be refined using mixed-volume bounds as in  \cite{Martinez2018} or \cite{Emiris2020}, at the cost of a genericity assumption: we won't follow this path hereafter.

\subsection{Related works and Putinar-type representations}
\label{sec:applications}
\label{sec:2}

\paragraph{Putinar-type representations over Archimedean fields.}
In order to better compare our work with the literature, we fix some notation about quadratic modules.
For a field $\kk\subset \RR$, we denote by  $\Sigma_{\kk}^2$ the convex cone of weighted \emph{Sums of Squares} (SoS) of polynomials over $\kk$:
$$
\Sigma_{\kk}^2 := \big\{ \, f \in \kkrng\ : \  \exists \, r
\in \NN, \omega_i \in \kk_{\geq 0}, \ q_i \in \kkrng \mbox{\ s.t\ } f = \omega_1\, q_1^2 + \dots + \omega_r q_r^2
\,\big\}
$$
Here we choose to take positive \emph{weighted} SoS in $\Sigma_K^2$, i.e. multiplying by the nonnegative weights $\omega_i \in \kk_{\geq 0}$, and not classical SoS.  When $K = \QQ$ or $K = \RR$, weighted SoS and standard SoS coincide, since in these two fields every nonnegative element is a sum of squares. However, this is not true in general, e.g. 
$\sqrt{2}$ is not a Sum of Squares in $K=\QQ(\sqrt{2})$; see also \cite[Cor.~1.1.12]{knebuschRealAlgebraFirst2022}.
The choice of weighted SoS is the standard choice when working with ordered fields which are not real closed, see e.g. \cite{Prestel2001, Schweighofer2002}. 

The \emph{quadratic module} generated by the polynomials 
$\vb g= \{g_{1}, \ldots,
g_{r}\}$ and $\pm \vb h= \{\pm h_1, \ldots, \pm h_s\}$ is 
\begin{equation}\label{eq:quadratic module}
Q_{\kk} = Q_{\kk}(\bm g, \pm \bm h) := \Sigma_{\kk}^2+ \sum_{i=1}^{r}\Sigma_{\kk}^2\cdot g_i + \sum_{j=1}^{s}  \kkrng \cdot  h_{j} \subset \kkrng
\end{equation}

 Putinar's Positivstellensatz \cite{putinarPositivePolynomialsCompact1993} states that any real polynomial that is  strictly positive on a compact  semialgebraic set $S$ can be represented in the quadratic module $Q_\RR$, under the assumption that $Q_\RR$ is {Archimedean}, i.e. that there exists $g\in Q_\RR$ such that $\{\xi\in \RR: g(\xi)\ge 0\}$ is compact.
 While Schm\"udgen's Positivstellensatz (closely related to Putinar's one) is known to hold over any Archimedean field $K \subset \RR$, see e.g. \cite{Schweighofer2002}, it is not {known} if the same is true for Putinar's Positivstellensatz, because it is uncertain if $Q_\RR$ Archimedean implies that $Q_K$ is Archimedean.
 
 But special results are known over the rational numbers: if $S$ is defined by polynomial with rational coefficients, $Q_\RR$ is Archimedean and $f \in \QQrng$ is strictly positive on $S$, then $f \in Q_\QQ$ if:
\begin{itemize}
    \item $R-x_1^2-\dots - x_n^2 \in Q_\QQ$ for some $R \in \QQ_{\ge 0}$ (see \cite{Powers2011});
    \item $S$ satisfies the unisolvent property\footnote{i.e. the polynomials appearing in the decomposition can be uniquely recovered by their values on the semialgebraic set $S$.}, implied by the non-emptyness of the interior of $S \subset \RR^n$ in the Euclidean topology (see \cite{Davis2022}). 
\end{itemize}
{In our setting, $Q_K$ satisfies the Archimedean condition by taking $g=-\sum h_j^2 \in Q_K$, but none of the conditions above are obviously satisfied: it is an open question whether the Archimedean condition over $\QQ$ implies the condition in \cite{Powers2011}, and the unisolvent property for all polynomials of bounded degree cannot occur if the semialgebraic set is zero-dimensional.}

The paper \cite{Magron2023} analyzes the special case of a nonnegative polynomial $f\in \QQrng$ over $\RR^n$ and proves that when  its gradient ideal  $I_{\mathrm{grad}} = (\partial_{x_1}f, \dots , \partial_{x_n}f)$ is a radical zero-dimensional ideal and the infimum $f^*=\inf\{f(\xi):\xi\in \RR^n \} $ is attained, then $f$ has a representation as a rational SoS modulo $I_{\mathrm{grad}}$.

For strictly feasible Gram matrix representations, the existence of rational SoS has been proved in \cite{peyrlComputingSumSquares2008} and an algorithm based on rounding techniques has been proposed.
See also \cite{Davis2022}.

The univariate case over the rational numbers has been investigated for instance in \cite{Magron2019} and  \cite{KMS2023}. The statement of \Cref{thm:A} and also its assumption for the nonnegative case are  a natural generalization of the latter to finite semialgebraic sets in $\RR^n$. 

To the authors' best knowledge, \Cref{thm:A} is the first result establishing the existence of a general Putinar-type representation for strictly positive polynomials over the rational numbers, for our very particular setting of finite semialgebraic sets.

\paragraph{Minimality of our assumption in the nonnegative case.}
As shown in \Cref{exp:x3}, not all nonnegative polynomials on $S$ have a representation in $Q_\kk$.
In the case of $K=\RR$, Theorem \ref{thm:A}  can  be viewed as  an extension of \cite{Parrilo2002}, which holds for nonnegative polynomials on finite semialgebraic sets $S$ defined by a radical zero-dimensional ideal $I$, to  an arbitrary zero-dimensional ideal $I$ for $f$ satisfying $(I:f)+(f)=(1)$. Moreover, \Cref{thm:B} can be seen as a full extension  of the same paper's result to other fields $K\subset \RR$, and in particular to the field $\QQ$.

{Another related work over $\RR$ is \cite{Hua2025}. In this work, it is assumed that  the first $n$ equations $h_1, \ldots, h_n$ have no root at infinity so that they form a graded basis (see \cite[Th.~3]{Hua2025}).
It is also assumed that the zeros of $I$ where $f-f_*$ vanishes are simple, which implies that $(I \colon (f-f_*)) + (f-f_*) = (1)$. 
The degree bounds on the SoS representation (see \cite[Th. 1]{Hua2025}) depend on $\mathfrak{n}=\sum_{i=1}^n \deg(h_i)-n$, which is the minimal degree of a basis of the quotient by $(h_1, \ldots, h_n)$. \Cref{thm:A} generalizes the degree bounds of \cite{Hua2025} by providing degree bounds for any subfield $K\subset \RR$ in terms of the minimal degree of a basis $B$ of the quotient by $(h_1, \ldots h_s)$, which can be substantially smaller than $\mathfrak{n}$, and by allowing a weaker regularity condition. See e.g. \Cref{ex:singular} for an example where \Cref{thm:A} applies, but \cite[Th. 1]{Hua2025} does not.%
}

We now discuss the minimality of our assumption $(I:f)+(f)=(1)$ in general.
For zero-dimensional ideals, as already mentioned,  $Q_\RR$ is an Archimedean quadratic module (see also \cite[Cor.~7.4.4]{Marshall2008}), and therefore, by \cite[9.1.2]{Marshall2008}  \[f\in Q_\RR \ \Longleftrightarrow \ f \in Q_\RR + (f^2) =\Sigma_{\RR}^2+ \Sigma_{\RR}^2 \cdot g_1 + \dots + \Sigma_{\RR}^2 \cdot g_r + I_{\RR}+ (f^2). \]
Since we can readily verify that $(I:f)+(f)=(1)$ is equivalent to $f\in I +(f^2)$, we then see that our necessary assumption is stronger, but in the same spirit, than the necessary and sufficient one.

We also notice that all the results of this paper can be extended to any finitely generated quadratic module $Q_\KK$ where the ideal $I = Q_\kk\cap - Q_\kk$, called \emph{support} of $Q_\kk$, is zero-dimensional. In this case the associated semialgebraic set is finite, and the result can be obtained by replacing the $h_j$ by generators of the support of $Q_\kk$. This might induce a shift of the degree 
in the SoS representations, due to the degree of the SoS representations of the generators of the support in $Q_\kk$.

\paragraph{Degrees bounds for Putinar-type representations.}
The degrees of Putinar-type representations for strictly positive polynomials on general $S$ have been investigated in \cite{Nie2007, Magron2021} and more recently in \cite{Baldi2022, Baldi2025}, respectively with exponential bounds and polynomial bounds depending on the minimum $f^*$ of $f$ on $S$: the closer $f$ is having a zero on $S$, the bigger is the degree needed for the representation. {These bounds have also been generalized to the matrix case, see \cite{Huang2025}.}

In our results, the degree bounds depend on the regularity of $\rng/I$, of order $\cl O(n\,d)$, but they are independent of $f^*$.
This can be explained using the concept of \emph{stability} (see for instance \cite[Sec.~6]{baldiDegreeBoundsPutinar2024}) that occurs 
in the case of zero-dimensional semialgebriac sets.
We discuss the degree bounds in more details in \Cref{sec:app}, in connection with Polynomial Optimization. 

Regarding the optimality of these bounds, 
it is shown in \cite{grigorievLinearLowerBound2001}, that the degree of {an SoS} certificate for deciding the emptyness of a basic semialgebraic set is at least $c\, d^n$ for some $c\in \RR_{>0}$.
As discussed in \Cref{sec:app}, the bounds on the degree in \Cref{thm:C} are tight in the case of binary polynomials $h_i=x_i^2-x_i$ and linear constraints $g_i$ (see also \cite{kurpiszHardestProblemFormulations2015}).

\paragraph{Heights for Putinar-type representations over $\QQ$.}
Obtaining good bitsize bounds in real algebra is a (mostly open) research challenge. For instance, for the (more general) problem of certifying the insatisfiability of sign conditions using the Krivine-Stengle Positvstellensatz, the only known bitsize bound, provided in the recent seminal work \cite{Lombardi2020}, reads as:
\[
            2^{2^{\left(2^{\max \{2, d\}^{4^{n}}}+m^{2^d} \max \{2, d\}^{16^{n}\mathrm{bit}(d)}\right)}}
\]
where $n$ denotes the number of variables, $m$ the number of input polynomials and $d$ a bound on their degrees. 

For a general Putinar-type representation over $\QQ$, the  following height bound of similar order, which fixes some estimates in \cite{Magron2021}, was presented in \cite{https://doi.org/10.48550/arxiv.1811.10062}:
\[  
    (r+1)\left(2^{n \tau d^{2 n+2}}\right)^{\mathcal O (2^{n \tau d^{2 n+2}})},
\]
where $r$ is the number of inequality constraints and $\tau$ is a height bound for the input.
A height analysis has also been developed in \cite{Davis2024},  involving the condition number of a certain Hankel matrix or the distance to the boundary of a cone, while the (worst-case) bounds in \Cref{thm:C} only depend on the number of variables, the height and the degree of the input. 

In \cite{Magron2023}, bitsize bounds for the coefficients in the {SoS} representation modulo the gradient ideal, wich is radical zero-dimensional, are of order  
$\cl O((n+d+\tau) d^{5n+2})$ for the coefficients of $\omega_{0,k}$ and $ q_{0,k}$.  The polynomial coefficients $p_j$ are given in terms of the lexicographic Gr\"obner basis of $I$, under a Shape Lemma assumption, and have heights bounded by 
$\widetilde{\mathcal O}\left(n (\tau+n+d) d^{3 n+1}\right)$. 

 The height bounds obtained in \Cref{thm:C}, which are polynomial in terms of the maximum degree $d$ of the input integer polynomials, linear in the  maximum height $\tau$ of the coefficients and the Castelnuovo-Mumford regularity, and simply exponential in terms of the number of variables $n$, are to our knowledge the first simply exponential bounds for the heights in such SoS representations for arbitrary finite basic closed semialgebraic sets.

\paragraph{Bitsize lower bounds for rational SoS representations.}
Determining lower bounds for the bitsize (and also the degrees) of SoS representations is a mostly open challenge with important applications. For instance, in Theoretical Computer Science, there is interest in studying the question of certifying emptiness of a given (finite) semialgebraic set $S$, by providing {an SoS} representation for $-1$, and the bitsize of such SoS representation governs the complexity of the problem.
It has also been observed in \cite{odonnellSOSNotObviously2016,raghavendraBitComplexitySumofSquares2017} that the bitsize of  Positivstellens\"atze representations can grow exponentially in the number of variables.
Lower bounds on the bitsize for low degree representations are also provided in these papers, showing  in a particular example where $d=2$  that any degree $2$ SoS representation of a nonnegative polynomial must have bitsize in $\Theta(d^n)$. 
Unfortunately, this lower bound does not apply in our case, since the given generators of the ideal of this example do not form a graded basis.
We also notice that in \cite[Th.~1]{raghavendraBitComplexitySumofSquares2017}, the existence of {an SoS} representation over $\QQ$ is assumed, whereas Theorems \ref{thm:A} and \ref{thm:B} provide existential results for such representations. 
Finally, we mention that the tradeoff between bitsizes and degrees of representations in the Positivstellens\"atze has been investigated in \cite{atseriasSizeDegreeTradeOffsSumsofSquares2019, hakoniemiMonomialsizeVsBitcomplexity2021}.

\subsection{Applications to discrete polynomial optimization}
\label{sec:app}
An important area of applications of {SoS representations} is Polynomial Optimization, which relies on the celebrated  moment/SoS hierarchy \cite{lasserreGlobalOptimizationPolynomials2001a, parriloStructuredSemidefinitePrograms2000}, that we recall briefly. 

 To define this hierarchy, we need the \emph{truncated quadratic module} generated, in degree $\le d$, by the polynomials 
$\vb g$ and $\pm \vb h$: 
$$
Q_{\kk,d} = Q_{\kk,d}(\bm g, \pm \bm h)  := \Sigma_{\kk,d}^2+ \sum_{i=1}^{r}\Sigma_{\kk,d-\deg(g_{i})}^2\cdot g_i+ \sum_{j=1}^{s}  \kkrng_{d-\deg(h_{j})}\cdot  h_{j} \subset \kkrng_{d}
$$
where $$\Sigma^{2}_{\kk,k} := \Sigma^{2}_{\kk}\cap \kkrng_{k}= \{ \, \omega_1\, q_1^2+\dots +\omega_r\, q_r^2 \colon r \in \NN, \omega_i\in \kk_{\ge 0}, \ q_i \in \kkrng_{\lfloor k/2 \rfloor} \}\subset \kkrng_{k}$$ 
denotes SoS of degree $\le k$.

Consider the problem of the minimization of a polynomial map $f$ on a finite basic closed semialgebraic set  $S=S(\bm g,\bm h) \subset \RR^n$ 
as in \eqref{eq:s}, and denote $f^*= \inf_{x\in S} f(x)$, the minimum of $f$ on $S$.

Given any integer $\ell \in \NN$ we can define a lower bound for $f^*$ as:
\begin{equation*}
    f_{(\ell)} := \sup \Big\{ \, m \in \RR \,:\, f - m \in Q_{\RR, 2\ell} \,\} \le f^* 
\end{equation*}
which is the order $\ell$ of the so-called \emph{SoS hierarchy}. 
By definition we have $f_{(\ell)} \le f_{(\ell+1)} \le f^*$ for all $\ell \in \NN$.
The hierarchy of dual convex optimization problems
\begin{equation*}
    f^{(\ell)} := \inf \Big\{ \, \lambda(f) \,:\,\lambda \in {\rng_{2\ell}^{*}}, \lambda(1)=1,\, \lambda(q)\ge 0\ \mbox{for all} \ q \in Q_{\RR,2\ell}\} \le f^* 
\end{equation*}
is called the \emph{moment hierarchy}. We have $f^{(\ell)}\le f^{(\ell+1)} \le f^*$ and $f_{(\ell)}\le f^{(\ell)}$.

A natural question is about the \emph{finite convergence} of such hierarchy, namely: does there exist $\ell \in \NN$ such that $f_{(\ell)} = f^{(\ell)}= f^*$? In this case, can we certify this convergence, providing an explicit representation $f - f^* = \sigma_0 + \sum_{i=1}^{r}\sigma_i  g_i+ \sum_{j=1}^{s}  p_j   h_{j} \in Q_{\RR}$ or $Q_{\QQ}$? 

\Cref{thm:A} provides an answer to these questions, when applied to the polynomial $f-f^*$.

\begin{corollary}\label{cor:pop_finite}
Let 
   $\bm g,  \bm h  \subset  \QQrng $ be as in \eqref{eq:gh}, where $\bm h$ is a  graded basis of the zero-dimensional ideal  $I \subset \QQrng$,   $B$  be a monomial basis of $\QQrng \big / I$, and
 $S\subset \RR^n$ be the finite basic closed semialgebraic set defined in \eqref{eq:s}.
 
   Let $f\in \QQrng$ and  $f^*$ be the minimum of $f$ on $S$. If  $(I \colon f-f^*) + (f-f^*) = (1)$, then $f_{(r)} = f^*$ for $$r \ge \frac{1}{2} \max \{ \deg (f), 2\deg(B) + \deg (g_1), \dots , 2\deg(B) +\deg(g_r) \}. $$
    Furthermore, $f-f^*$  has a certificate of nonnegativity as in \Cref{eq:main_result}   using polynomials of degree $\le 2r$, which belong to $\QQrng$ when  $f^*\in \QQ$.
\end{corollary}

The hypotheses of \Cref{cor:pop_finite} might at first seem difficult to verify, but they are satisfied for most of the concrete instances of optimization problems on finite sets.
The degree bounds are similar to those that can be deduced implicitly in \cite{Parrilo2002} for SoS representations over $\RR$, in the case of a zero-dimensional ideal $I$.

The most prominent example of this is 0/1 optimization. Indeed, $$\{ 0, 1 \}^n=V(x_1^2-x_1,\dots,x_n^2-x_n),$$ where for $h_i=x_i^2-x_i$,  $\bm h=\{h_1,\dots,h_n\}$ is a graded basis of the  zero-dimensional {\em radical} ideal  $I = (\bm h)$, and $B=\{x_1^{e_1}\cdots x_n^{e_n}: e_i\in\{0,1\} \}$ has  degree $n$. Moreover, since $\{ 0, 1 \}^n\subset \QQ^n$ we have   $f^*\in \QQ$ and therefore \Cref{cor:pop_finite} fully applies. 

The finite convergence of different families of convex relaxations for $0/1$ programs has been studied since a long time, see e.g. \cite{Sherali1990, Lovsz1991}. The comparison of these techniques with the more recent moment/SoS hierarchy has been performed in \cite{Laurent2003}, demonstrating the strength of this new approach.
In \cite{Laurent2006}, the same degree bound of \Cref{cor:pop_finite} is provided for the finite convergence of the Moment hierarchy (i.e. $f^{(\ell)}$) and the SoS hierarchy (i.e. $f_{(\ell)}$), without deducing degree bounds for the SoS representation (when it exists). Similarly, sparse-type bounds  $2\ell \le n + \deg(f) - 1$ for the order $\ell$ of moment-like relaxations in the quotient algebra where $h_i=x_i^2-1$ and $V_{\RR}= \{-1, 1\}°^n$ have been obtained in \cite{Sakaue2017} (see also \cite{Fawzi2015}), without bounding the degree of SoS representations.

The bound of \Cref{cor:pop_finite} is tight in general, see for instance  \cite{kurpiszHardestProblemFormulations2015} and references therein: for binary optimization problems with $h_i=x_i^2-x_i$, finite convergence always happens at order $r = n$ when there are no inequalities, and at order $r = n+1$ when the $g_i$'s are linear. Applying \Cref{cor:pop_finite} (and assuming that $\deg(f) \le n$, which can be done by reducing $f$ modulo $\bm h$) we get the same bounds, and we furthermore show that a rational certificate of nonnegativity for $f-f^*$ exists. The cases when $f_{(n-1)} < f^*$ (resp. $f_{(n)} < f^*$) are identified and characterized in \cite{kurpiszHardestProblemFormulations2015}, showing that SoS representations may not exists in degree $< 2\delta$ (resp. $< 2\delta+ \max_i\deg(g_i)$) and that {the degree bounds of \Cref{thm:A} (and \Cref{cor:pop_finite}) are tight}. 

More generally, the results of \Cref{cor:pop_finite} can be directly applied when optimizing over any grid of points, see for instance \cite{lasserrePolynomialsNonnegativeGrid2002}.
We also notice that \Cref{cor:pop_finite} can be used to provide degree bounds and rational certificates for the exactness of \emph{$\theta$-body} approximations of polynomial ideals, see e.g. \cite{gouveiaThetaBodiesPolynomial2010}. These approximations are, under the assumption that $\bm h$ is a graded basis, essentially obtained using the moment/SoS hierarchy, see for instance \cite[Rem.~3.25]{netzerGeometryLinearMatrix2023}.

Finally, we remark that the zero-dimensional setting is important also for continuous polynomial optimization problems, as the latter can be reduced to the former by adding gradient \cite{nieMinimizingPolynomialsSum2006, Magron2023} or KKT \cite{demmelRepresentationsPositivePolynomials2007} constraints.

\subsection{Proof strategy and structure of the manuscript}

The proof of \Cref{thm:A} is developed in Sections \ref{sec:fpositive} and \ref{sec:fnonnegative}. We first consider in \Cref{subs:radical} the case 
when the ideal $I$ is {radical} and $f$ is strictly positive {on the whole $V_\RR(I)$}, obtaining   \Cref{thm:radical}. This is a refined version of \Cref{thm:A} in the radical case, with 
   an additional non-vanishing condition that will allow us to  lift our construction to a non-radical ideal $I$ in \Cref{subs:arbitrary}.   
    Then, we consider in  \Cref{subs:extension} the case when $f> 0$ on an {arbitrary finite basic closed semialgebraic set $S\subset \RR^n$} as in \eqref{eq:s}, by perturbing the original polynomial $f$, which is positive on $S$, to an appropriate polynomial positive on all $V_\RR(I)$. 
    Finally, we study in  \Cref{sec:fnonnegative} the case where {$f\ge 0$ on $S$ with $1\in (I:f)+(f)$}. In this case we show that we can replace our original polynomial $f$, which is nonnegative on $S$, by another related polynomial which is strictly positive on $S$, and conclude by applying the previous results. 

Sections \ref{sec:heightradical} to \ref{sec:7} analyze the  height bounds for the case $\kk=\QQ$ and $I$ radical: the results are summarized in \Cref{thm:C}.
In  \Cref{sec:heightradical}  we derive, from an arithmetic Bézout theorem, height bounds for different polynomials in the quotient ring $\QQrng/I$ that we need in the sequel. 
The cases of $f$ positive on $S$ and $f$ nonnegative on $S$ are then developed in \Cref{sec:6} and \ref{sec:7} respectively.

Finally \Cref{sec:8} presents an algorithm to obtain SoS representations of the kind previously studied, and some examples that illustrate our constructions. 

\section{SoS representation for $f$ strictly positive on $S$}\label{sec:fpositive}

In this section we consider Problem~1 stated in \Cref{sec:intro}, whose solution is summarized in \Cref{thm:A}, under the assumption that $f$ is strictly positive on $S$. Without loss of generality, we give all our proofs for $\kk=\QQ$ (the smallest  field included in $\RR$) since they mutatis-mutandis hold for any field $\kk\subset \RR$, using the fact that $\kk$ is dense in $\RR$.

\subsection{Notation and preliminaries}
\label{sec:preliminaries}

We fix the following notations that we will use in all our text.

\begin{itemize}
\item $\bm g=\{g_1,\dots,g_r\},\bm h=\{h_1,\dots,h_s\}\subset \QQrng$ are as in \eqref{eq:gh} and $S=S(\bm g, \bm h)\subset \RR^n$ is the basic closed semialgebraic set defined in \eqref{eq:s}. 
\item 
$I= (\bm h) \subset \QQ[\bm x]$ is the zero-dimensional ideal generated by 
$\bm h$, and for 
 $d\in \NN$, $\QQrng_d=\{ p\in \QQrng \colon \deg(p) \le d\}$ denotes the polynomials of degree $\le d$, and $I_{d}=I\cap \QQrng_{d}$. 
\item  $I_{\CC} $ denotes  the ideal  generated by $(\bm h)$  in $\CC[\bm x]$ and   $I_{\RR} $ denotes the ideal it generates in $\RR[\bm x]$. Analogously, 
$V_{\CC}:= V_{\CC}(I)$  is the affine variety defined by $\bm h$  in $\CC^n$ and $V_{\RR}:= V_{\RR}(I)$  is the affine variety defined by $\bm h$  in $\RR^n$. 

\item  $D:=\dim(\QQrng/I)=\#(V_\CC(I))$, where the roots are counted with multiplicities, which satisfies $D\le \displaystyle \max_j\{\deg(h_j)\}^n$ by  B\'ezout's theorem.

\item $S^D(\RR)$ (resp. $S^D(\QQ)$) denotes the set of symmetric matrices in $\RR^{D\times D}$ (resp. $\QQ^{D\times D}$).

\item  $B$ denotes a monomial basis of $\QQrng/I$ associated to the basis $\bm h$, with first element $1$ and then constructed by adding in the basis monomials obtained by multiplying with the variables $x_1, \dots , x_n$. We denote $\deg(B):= \max_{b\in B} \deg(b)$, which implies that  $\deg(B)<D$. We also denote by $\langle B\rangle_\QQ$ (resp. $\langle B\rangle_\RR$,  $\langle B\rangle_\CC$) the $\QQ$-vector space spanned by $B$ in $\QQrng$ (resp. $\RR$-vector and $\CC$-vector space  space in $\RR[\bm x]$ and $\CCrng$). We also observe that 
any polynomial $p\in \QQrng$ can  be decomposed as
$$ 
p = \sum_{j=1}^{s} p_j h_{j}  + \cl N(p)
$$
for a unique  $\cl N(p)\in \vspan{B}_\QQ$ (the {\em normal form} of $p$) and some $p_j\in \QQrng$ for $1\le j\le s$. Moreover, when  $\bm h$ is a \emph{graded basis} of $I$, i.e. when for any $d\in \NN$ we have
$$ 
I_{d}= \sum_{j=1}^{s}  \QQrng_{d-\deg(h_{j})} \cdot h_{j}, 
$$
we can take  $p_j\in \QQrng_{\deg(p)-\deg(h_{j})}$ for $1\le j\le s$. 

\item We will make use of the 
idempotents $u_{\zeta}\in \vspan{B}_{\CC}$ for $\zeta\in V_\CC(I) = V_\CC$,  which satisfy for $\zeta,\xi\in V_\CC$:
\begin{itemize}
 \item $ u_{\zeta}(\zeta)=1$ and $ u_{\zeta}(\xi)=0$ for $\xi\neq \zeta$;
 \item $u_{\overline \zeta}=\overline{ u_{\zeta}}$, and in particular $u_{\xi}\in \rng$ if $\xi\in V_{\RR}$;
 \item when  $I$ is a radical ideal (which implies that if $p\in \CC[\bm x]$ satisfies that $p(\zeta)=0$ for all $\zeta\in V_\CC$, then $p\in I$) we have $ u_{\zeta}^{2} \equiv  u_{\zeta} \mod I_\CC$ and 
 $u_{\zeta} \, u_{\xi}\equiv 0 \mod I_\CC$ if $\zeta\neq \xi$.
\end{itemize}
\end{itemize}

\subsection{The case when $S=V_\RR(I)$ for a radical  zero-dimensional  ideal $I$ }\label{subs:radical}
We assume in this section  that $I$ is a radical zero-dimensional ideal, i.e. in addition to be zero-dimensional,  we assume that $\sqrt{I}=I$, so that all the points $\zeta \in V_{\CC}$ have multiplicity $1$. We also assume that $\bm g =\emptyset $, so that $S$ coincides with the finite real algebraic variety $V_\RR = V_\RR(I)$.

Our goal is to extend the univariate case treated in \cite[Proposition 2.2]{KMS2023} to the multivariate setting. The final result of the section is the following theorem.

\begin{theorem}\label{thm:radical}
Let $I \subset \QQrng$ be a zero-dimensional radical ideal and let $f\in \QQrng$ be such that  $f(\xi)>0$ for all $\xi\in V_{\RR}(I)$. Let $B$ be a basis of $\QQrng\big /I$.  
 Then, for all $\zeta\in V_{\CC}(I)$,  there exist   $\omega_{\zeta}\in \QQ_{> 0}$  and $\theta_{\zeta}\in \vspan{B}_{\QQ}$   such that  
\begin{equation*}
{f} \equiv \sum_{\zeta\in V_{\CC}} \omega_{\zeta}\, \theta_{\zeta}^{2}\quad   \mod I.
\end{equation*}
Moreover, when $V_{\CC}\neq \emptyset$ there exists $\zeta_{0}\in V_{\CC}$  such that $\theta_{\zeta_{0}}(\zeta)\neq 0$  for all $\zeta\in V_{\CC}$.
\end{theorem}

We begin by considering in the next lemma the case where $V_\RR(I)\neq \emptyset$.

\begin{lemma}\label{lem:sos:nempty} Let $I \subset \rng$ be a zero-dimensional radical ideal such that that $V_\RR(I)\ne \emptyset$, and let 
 $f\in \rng$ be such that $f(\xi)>0$ for all $\xi\in V_\RR$. Let  $B$ be a basis of $\rng\big /I$.   Then  there exist $\theta_{\zeta}\in \vspan{B}_\RR$,  $\zeta\in V_{\CC}$, which form a basis of $\rng/I$,  such that
\begin{equation*}
  f \equiv  \sum_{\zeta \in V_{\CC}}  \theta_{\zeta}^2 
\quad \mod I_\RR .
\end{equation*} 
Moreover,  there exists  $\xi_0\in V_\RR$ such that $\theta_{\xi_0}(\zeta)\ne 0$ for all $\zeta \in V_{\CC}$.
\end{lemma}
\begin{proof} The condition $\theta_{\xi_0}(\zeta)\ne 0$ for all $\zeta \in V_{\CC}$  in the statement requires a change in the proof of \cite[Proposition~2.2]{KMS2023}: we will take a suitable linear combination of the $u_i$'s to guarantee it.  We choose $\xi_{0}\in V_\RR$ and given  $\varepsilon\in \RR$  for $\xi\in V_\RR\setminus \{\xi_0 \}$, we define 
\begin{equation}\label{eq:theta0}
v_{\xi_0}= u_{\xi_{}} +\varepsilon \sum_{\xi\in V_{\RR},\xi\neq \xi_{0}}   u_{\xi} +\sum_{\zeta\in V_{\CC}\setminus V_\RR}  u_{\zeta} \ \in \ \vspan{ B}_{\RR} 
\end{equation}
which is a real polynomial since for $\zeta\in V_\CC\setminus V_\RR$, $u_{\overline\zeta}=\overline u_\zeta$. This polynomial satisfies $v_{\xi_0}(\xi_0)=1$, $v_{\xi_0}(\xi)=\varepsilon$ for all $\xi\in V_\RR$, $\xi\ne \xi_0$, and $v_{\xi_0}(\zeta)=1$ for all $\zeta\in V_\CC\setminus V_\RR$.  Set $Z$ to be a minimal set of complex non-real roots in $V_\CC$
such that $V_\CC\setminus V_\RR=Z\cup \overline Z$.
Then, given $\lambda_\zeta\in \RR$ for $\zeta \in V_\CC \setminus V_\RR$ we can verify that
\begin{align*} f &\equiv 
f(\xi_{0}) \,v_{\xi_{0}}^{2}
+\sum_{\xi \in V_\RR,\xi\neq \xi_{0}} (f(\xi)-\varepsilon^{2}f(\xi_{0}))\,  u_{\xi}^2\\
& \ \  + \sum_{\zeta\in Z} \big((f(\zeta)-f(\xi_{0}))\,  u_{\zeta}^2 + (f(\overline{\zeta})-f(\xi_{0}))\, \overline{ u}_{\zeta}^{2} + 2 \lambda\,  u_{\zeta}\, \overline{ u}_{\zeta}\big) \quad   
\mod I_{\RR}
\end{align*} 
since  the expression in the right-hand side is a real polynomial that coincides with $f$ on all $\zeta \in V_\CC$, and therefore, by the Nullstellensatz, their difference belongs to $\sqrt I=I$ by assumption.

We choose $\varepsilon\in \RR_{>0}$ small  enough such that  $f(\xi)-\varepsilon^{2}f(\xi_{0})>0$ for all $\xi\in V_\RR \setminus \{ \xi_0 \}$, and for $\zeta\in Z$, we will apply the following idendity to \begin{equation}\label{eq:nonreal} 
(f(\zeta)-f(\xi_{0}))\,  u_{\zeta}^2 + (f(\overline{\zeta})-f(\xi_{0}))\, \overline{ u}_{\zeta}^{2} + 2 \lambda\,  u_{\zeta}\, \overline{ u}_{\zeta}\end{equation} for a suitable $\lambda$:
Given  $a,b,\lambda \in \RR$ with $a+\lambda\ne 0$ and $u,v\in \RR[\bm x]$:
\begin{eqnarray}
\lefteqn{ (a + \ib b) (u+ \ib v)^2
+ (a - \ib b) (u - \ib v)^2 + 2\,\lambda |u+\ib v|^2 } \nonumber\\
&=& 2\,\big( (\lambda+a)\,u^2 - 2\,b\, u\,v + (\lambda-a)\, v^2 \big) \nonumber\\
&=& {2 (\lambda + a)} ( u - \frac{b}{ \lambda + a} \, v)^2 +  2\frac{\lambda^2-|a+\ib b|^2}{\lambda + a} v^2.\label{eq:identity}
\end{eqnarray}
We observe that when $\lambda>|a+\ib b|$, then $\lambda + a>0$ and $\lambda^2-|a+\ib b|^2>0$.
Identity~\eqref{eq:identity} thus shows that 
 when $\lambda>|f(\zeta)-f(\xi_0)|$, \eqref{eq:nonreal} is a sum of two positive-weighted squares of real polynomials.
 
Therefore, by taking  $\lambda\in \RR_{>0}$  such that $\lambda>|f(\zeta)-f(\xi_0)|$ for all $\zeta\in V_\CC\setminus V_\RR$, we obtain that 
\begin{eqnarray}\label{eq:nsos}
  f &\equiv& \sum_{\xi \in V_\CC}  \theta_\xi^2 
\quad \mod I_\RR
\end{eqnarray}
where, denoting respectively $\Re$ and $\Im$ the real and imaginary part:
\begin{itemize}
\item  $\theta_{\xi_{0}}=\sqrt{f(\xi_0}\,v_{\xi_0}\in \vspan{B}_{\RR}$ where $\widetilde u_{\xi_0}$ is defined in \eqref{eq:theta0};
\item  $\theta_\xi=  \sqrt{f(\xi)-\varepsilon^{2}f(\xi_{0})}\,u_\xi\ \in \vspan{B}_{\RR}$, \ for all $\xi\in V_\RR\setminus \{\xi_{0}\}$;
\item  
     $\theta_{\zeta}= \sqrt{2\,(\lambda + \Re(f(\zeta))-f(\xi_{0}))}\,\big(\Re( u_{\zeta}) - \dfrac{\Im(f(\zeta))}{ \lambda_\zeta +\Re(f(\zeta))-f(\xi_{0})}\, \Im( u_{\zeta})\big)\in \vspan{B}_{\RR}$, \  for all $ \zeta\in Z$;
  \item 
 $\theta_{\overline \zeta}=\sqrt{2\dfrac{{ \lambda_\zeta^2-|f(\zeta)-f(\xi_{0})|^{2}}}{\lambda_\zeta + \Re(f(\zeta))-f(\xi_{0})}}\, \Im( u_{\zeta}) \in \vspan{B}_{\RR}$, \ 
 for all $\zeta\in Z$.
 \end{itemize}
 Since the polynomials $ \{u_{\zeta}\,:\,\zeta\in V_\CC\}$ are linearly independent over $\CC$, the polynomials \[\{u_\xi,\Re(u_{\zeta}),\Im(u_\zeta)\,:\,\xi\in V_\RR,\zeta\in Z\}\] are also linearly independent over $\RR$, and therefore a basis of $\RR[\bm x]/I$ because of its cardinality. This implies that the polynomials $\{\theta_\zeta\,:\,\zeta\in V_\CC\}$, are a basis of $\rng/I$ and by construction  $\theta_{\xi_0}(\zeta)\ne 0$, $\forall \,\zeta\in V_\CC$.
This concludes the proof.
\end{proof}

We now consider the other case, when  $V_\RR=\emptyset$.
\begin{lemma} \label{lem:sos:empty}  Let $I \subset \rng$ be a zero-dimensional radical ideal such that  $V_{\RR}(I)= \emptyset$.
Let $B$ be a basis of $\rng\big /I$ and $f \in \rng$.  Then  there exist $\theta_{\zeta}\in \vspan{B}_\RR$,  $\zeta\in V_{\CC}$, which form a basis of $\rng/I$,  such that
\begin{equation*}
  f  \equiv  \sum_{\zeta \in V_{\CC}}  \theta_{\zeta}^2 
\quad \mod I_\RR .
\end{equation*} 
Moreover,  there exists  $\zeta_0\in V_\CC$ such that $\theta_{\zeta_0}(\zeta)\neq 0$ for all $\zeta \in V_{\CC}$.
\end{lemma}
\begin{proof}
 We choose $\zeta_0\in  V_{\CC}$ and  for $\varepsilon\in \RR$ define
$$
{ v}_{\zeta_0} =  u_{\zeta_{0}} + \varepsilon u_{\overline \zeta_{0}} + \sum_{\substack{\zeta\in V_\CC\\\zeta \neq \zeta_{0},\zeta\ne \overline \zeta_0} } u_{\zeta} \quad  \mbox{ and } \quad {v}_{\overline \zeta_0} =\overline{ v}_{\zeta_0}
$$
which satisfy $ v_{\zeta_0}(\zeta_0)=1$, $ v_{\zeta_0}(\overline \zeta_0)=\varepsilon$ and $ v_{\zeta_0}(\zeta)=1$, $\forall \, \zeta\in V_\CC$ with $ \zeta\ne \zeta_0,\zeta\ne \overline \zeta_0$.
We note that the set $\{v_{\zeta_0},v_{\overline \zeta_0}, u_\zeta \,:\, \zeta \in V_\CC,\zeta\ne \zeta_0, \zeta\ne \overline\zeta_0\}$ is linearly independent over $\CC$ for $\varepsilon\ne 1$, and therefore a basis of $\CC[\bm x]/I$.

Let $Z\subset V_\CC$ be a minimal set such that $V_\CC=Z\cup \overline Z$, and let $f \in \rng$.
Given $\lambda_0,\lambda\in \RR$, we observe that  
\begin{align}\label{eq:equivalencef}
f &\equiv  \big(\frac{f(\zeta_{0})-\varepsilon^{2} f(\overline \zeta_{0})} {1- \varepsilon^{4}} -\varepsilon\lambda_{{0}} \big) \,{ v}_{\zeta_{0}}^{2}
+
\big(\frac{f(\overline \zeta_{0})- \varepsilon^{2} f(\zeta_{0})} {1 -\varepsilon^{4}}-\varepsilon\lambda_{{0}} \big)\,{ v}_{\overline \zeta_{0}}^{2}\\ \nonumber
  & \quad  + (1+\varepsilon^2)\lambda_{{0}} { v}_{\zeta_{0}}{ v}_{\overline \zeta_{0}} + \sum_{\zeta\in Z, \zeta\neq \zeta_{0}} \big(\rho_{\zeta} \, { u}_{\zeta}^2 + \overline{\rho}_{\zeta}\, {u}_{\overline \zeta}^{2} + 2 \lambda\, { u}_{\zeta}\,{ u}_{ \overline \zeta}\big)\quad \mod I_\RR
\end{align}
where $$\rho_\zeta:=f(\zeta)-\Big( \frac{f(\zeta_0)+f(\overline \zeta_0)}{1+\varepsilon^2}+(1-\varepsilon)^2\lambda_{0} \Big)=f(\zeta)-\Big( \frac{2\Re(f(\zeta_0))}{1+\varepsilon^2}+(1-\varepsilon)^2\lambda_{0} \Big)$$
since $f(\zeta)$ and the expression in the right-hand side coincide on all $\zeta \in V_\CC$.
In order to apply Identity~\eqref{eq:identity} to the first terms in the right-hand side of \Cref{eq:equivalencef}, we first choose  $\varepsilon,\lambda_{0} \in \RR_{>0}$ 
such that 
\begin{equation*}\label{eq:relation0} 
(1+\varepsilon^2)\lambda_{{0}}> 2\,\Big| \frac{f(\zeta_{0})- \varepsilon^{2} f(\overline \zeta_{0})} {1 -\varepsilon^{4}}-\varepsilon\lambda_{0}\Big|.
\end{equation*}
(we can first choose $\lambda_{0}>2 |f(\zeta_{0})|$ and then $\epsilon>0$ so that the above inequality, which is valid at $\epsilon=0$, still holds for $\epsilon>0$ small enough by continuity).
Identity~\eqref{eq:identity} then gives expressions  
\begin{equation*}
\theta_{{\zeta_0}}= \sqrt{2(\lambda +a)}\Big(\Re(v_{\zeta_0})-\dfrac{b}{\lambda + a}\Im(v_{\xi_0})\Big)\,  , \  \theta_{\overline\zeta_0}=\sqrt{2\dfrac{\lambda^2-|a+\ib b|^2}{\lambda + a} }\,\Im (v_{\zeta_0}) \quad \in \RR[\bm x].\end{equation*}
Now since  \begin{equation*} \Re(v_{\zeta_0})(\zeta_{0})=\Re(v_{\zeta_0})(\overline\zeta_0)=\dfrac{1+\varepsilon}{2}\ , \ 
  \Re(v_{\zeta_0})(\zeta)=\Re(v_{\zeta_0})(\overline\zeta)=1 \  \mbox{ for }  \zeta\in Z, \ \zeta\ne \zeta_0,
\end{equation*}
 and 
 \begin{equation*} \Im(v_{\zeta_0})(\zeta_{0})=\dfrac{(1-\varepsilon)\ib}{2},\ \Im(v_{\zeta_0})(\overline\zeta_0)=-\Im(v_{\zeta_0})(\zeta_{0}), \ 
  \Im(v_{\zeta_0})(\zeta)=\Im(v_{\zeta_0})(\overline\zeta)=0  \  \mbox{ for }  \zeta\in Z, \ \zeta\ne \zeta_0 ,
\end{equation*}
we have $\theta_{\zeta_0}(\zeta)\ne 0$  for all $\zeta\in Z$.

For the remaining $\zeta$, we  take  $\lambda \in \RR$ such that 
$ 
\lambda > |\rho_\zeta|
$, for all $\zeta\in V_\CC$, $\zeta\ne \zeta_0$.  
This allows to deduce from  Identity \eqref{eq:identity}  a decomposition of the form
$$ 
 f \equiv \sum_{\zeta\in Z} \big( \theta_{\zeta}^{2} +  \theta_{\overline \zeta}^{2}\big)
\quad \mod I,
$$
with  $\bm \theta_{\zeta_{0}}(\zeta)\neq 0$ for all $\zeta\in V_{\CC}$, which concludes the proof.
\end{proof}

With these two lemmas, which provide real representations, we can now prove Theorem~\ref{thm:radical}.

\begin{proof}[Proof of Theorem~\ref{thm:radical}]
 If $V_{\CC}= \emptyset$, then $I_\QQ=(1)$ and we have the decomposition  $f\equiv 0\ \mod (1)$. We therefore  assume that $V_{\CC}\neq\emptyset$.
 
From \Cref{lem:sos:nempty} and \Cref{lem:sos:empty},  we have that 
$$
    f\equiv \sum_{\zeta \in V_{\CC}}  \widetilde\theta_\zeta^2 \ \mod I_\RR \quad \mbox{with} \quad \sum_{\zeta \in V_{\CC}}  \widetilde\theta_\zeta^2=
 B(\widetilde \Theta \, \widetilde \Theta^t) B^t, $$ for  a basis 
$\{\widetilde\theta_\zeta, \zeta\in V_{\CC}\}=(\widetilde \theta_1,\dots,\widetilde \theta_D)$  in $\langle B\rangle_\RR$  of   $\RR[\bm x]/I$ where $ \widetilde\theta_D:=\widetilde\theta_{\zeta_{0}}$ satisfies $\widetilde\theta_D(\zeta)\ne 0$  for all $\zeta\in V_{\CC}$, and, by setting 
$B=(b_1,\dots,b_D)$,  $\widetilde \Theta\in \RR^{D\times D}$ is  the invertible matrix defined by  $(\widetilde \theta_1,\dots,\widetilde \theta_D)=B \,\widetilde \Theta$, which satisfies $\big((B\widetilde \Theta)(\zeta)\big)_D=\sum_{i=1}^D \widetilde\Theta_{iD}b_i(\zeta)\ne 0$ for all $\zeta\in V_\CC$. 

As $\QQ$ is dense in $\RR$, we can choose a close  invertible approximation  $ \Theta \, \in \, \QQ^{D \times D}$  of $\widetilde \Theta$, so that $\widetilde \Theta \, \Theta^{-1}=\Id+E\sim \Id$, i.e. $\widetilde \Theta=  \Theta (\mathrm{Id}+E)$, which still satisfies that 
 $\big((B  \Theta)(\zeta)\big)_D\neq 0$ for all $\zeta \in V_{\CC}$, and 
 $\widetilde Q:=(\Id+E)(\Id+E^t)\in S^D(\RR)$, close to $\Id$,  is still definite positive.
 
 Let now $\widehat Q\in S^D(\QQ)$ be a close enough positive definite rational  approximation of $\widetilde  Q$, so that  the projection $Q\in S^D(\QQ)$ of $\widehat Q$ on the linear variety 
 $$ 
\cl L_{f} = \{Y \in S^{D}(\RR) \,:\, f\equiv  B ( \Theta \,Y \,\Theta^t) B^t \ \mod I_\RR\},
$$
which satisfies $f\equiv B(\Theta\, Q \,\Theta^t) B^t\ \mod I$,
is still positive definite. 

Applying the square-root-free Cholesky decomposition (see \Cref{prop:cholesky}) to $Q$, we obtain $$Q=L\Delta L^t$$ where $L\in \QQ^{D\times D}$ is invertible and lower triangular and $\Delta$ is a diagonal matrix with rational positive entries. Moreover, we observe that  the last column of $\Theta \,L$ satisfies that  $( \Theta\,L)_{i,D}= L_{D,D} \Theta_{i,D}$ for $1\le i\le D$, and therefore,  $\big((B\, \Theta\, L)(\zeta)\big)_D = L_{D,D} \big((B\,\Theta )(\zeta)\big)_D\ne 0$ for all $\zeta\in V(\CC)$.

We conclude by observing that 
$$ B( \Theta \,Q \, \Theta^t) B^t=(B\,\Theta \,L) \Delta (B\, \Theta\, L)^t =\sum_{\zeta\in V_\CC} \omega_\zeta \theta_\zeta^2$$
with $\omega_\zeta\in \QQ_{>0}$ (the  coefficients of $\Delta$) and $\theta_\zeta=B\,\Theta \, L\in \langle B\rangle_\QQ$ which satisfies $\theta_{\zeta_0}(\zeta)= \big((B\, \Theta \,L)(\zeta)\big)_D \ne 0$, for all $\zeta \in V_\CC$. 
\end{proof}
 
\subsection{The case $S=V_\RR(I)$ for an  arbitrary zero-dimensional ideal $I$} \label{subs:arbitrary}

Here we consider the case where $I=(\bm h)\in \QQ[\bm x]$ is an arbitrary zero-dimensional ideal defining the roots $\zeta\in V_{\CC}$ with some multiplicities. We will extend Theorem~\ref{thm:radical} to this setting by Hensel lifting. We need the following elementary lemma.

\begin{lemma}\label{lem:invert} Let $J\in \QQ[\bm x]$ be a zero-dimensional  ideal and $\theta\in \QQ[\bm x]$ such that 
$\theta(\zeta)\neq 0$, $\forall\, \zeta\in V_{\CC}(J)$. Then $\theta$ has an inverse $\sigma$ modulo $J$ in $\QQ[\bm x]$. 
\end{lemma}
\begin{proof}
Since $\theta(\zeta)\ne 0, \ \forall\, \zeta \in V_\CC$, $V_\CC(J+(\theta))=\emptyset$ and therefore, by the Nullstellensatz, $1\in J+(\theta)\subset \QQ[\bm x]$. Therefore, there exists
$\sigma\in \QQ[\bm x]$ 
such that $\theta\,\sigma\equiv 1 \ \mod J$.
\end{proof}

We now  lift a sum of squares representation modulo a radical ideal $J$ to its powers by Hensel lifting, applying Newton iterations for computing square-roots modulo an ideal (see e.g. \cite[Chap.~9]{vonzurGathen2013}, \cite[Lemma~2.10]{KMS2023}).

\begin{proposition}\label{prop:sqrtNewton} Let $J\subset \QQrng$ be a zero-dimensional   ideal. Let $\theta\in \QQrng$ and 
assume that there exists $\theta_{0}\in \QQrng$
with $\theta_0(\zeta)\neq 0$,  $\forall\, \zeta\in V_{\CC}(J)$, such that $\theta_{0}^{2} \equiv \theta \mod J$.
Then for all $k \in \NN$, there exists $\theta_{k}\in \QQrng$ such that
\begin{equation*}
\theta_{k}^{2} \equiv \theta \ \mod J^{2^{k}}.
\end{equation*}
\end{proposition}
\begin{proof}
We prove the proposition by induction, applying Newton iterations to the equation $x^{2}-\theta=0$.
The property is true for $k=0$ by hypothesis.  
Assume that it is true for $k\in \NN$. There exists $\theta_{k}\in \QQrng$ such that $\theta_{k}^{2} \equiv \theta \mod J^{2^{k}}$ or equivalently $\theta_{k}^{2}-\theta  \in J^{2^{k}}$.
Then, since  $\theta_k^2\equiv \theta_0^2 \, \mod J$, one has that  for all $\zeta\in V_\CC$, $\theta_{k}^{2}(\zeta) \neq 0$, and by \Cref{lem:invert}, $\theta_{k}$ is invertible modulo the zero-dimensional ideal $J^{2^{k+1}}$.

Let $\sigma_{k}\in \QQrng$ be such that $\theta_{k} \sigma_{k} \equiv 1 \mod J^{2^{k+1}}$. We define $\theta_{k+1}= \frac 1 2 ( \theta_{k} + \theta \sigma_{k})$.
Then, for  $p_k:=\theta_{k}^{2}-\theta $, we have
\begin{align*}
  \theta_{k+1} -\theta_k & =  \frac 1 2 ( \theta_{k} + \theta \sigma_{k}) - \theta_{k}= \frac 1 2 ( - \theta_{k} + \theta \sigma_{k})\\
&\equiv \frac 1 2 \sigma_{k}\, (- \theta_{k}^{2} + \theta ) \quad  \mod J^{2^{k+1}} \\
& 
\equiv - \frac 1 2 \sigma_{k}\, p_{k} \quad \mod J^{2^{k+1}} .
\end{align*}
We deduce that
\begin{align*}
\theta_{k+1}^{2} & =   (\theta_{k} + \theta_{k+1}-\theta_{k})^{2} \  \equiv \ (\theta_{k} -\frac 1 2 \sigma_{k}\, p_{k})^{2} \quad \mod J^{2^{k+1}}\\
& \equiv  \ \theta_{k}^{2} - \theta_{k} \sigma_{k} p_{k} + \frac 1 4 \sigma_{k}^{2} p_{k}^{2} \quad\mod J^{2^{k+1}} \ \equiv  \  \theta_{k}^{2} - \theta_{k} \sigma_{k} p_{k} \quad\mod J^{2^{k+1}}\\
& \equiv \ p_k+\theta  - p_{k} \quad\mod J^{2^{k+1}} \    \ \equiv \ \theta \quad\mod J^{2^{k+1}},
\end{align*}
since $p_{k}^{2}\in (J^{2^{k}})^{2}=J^{2^{k+1}}$ and  $\theta_{k} \sigma_{k} \equiv 1 \mod J^{2^{k+1}}$.

This completes the proof by induction.
\end{proof}

\Cref{prop:sqrtNewton} allows us to extend Theorem~\ref{thm:radical} to the case when $I$ is an arbitrary zero-dimensional ideal, not necessarily radical.

\begin{theorem}\label{th:nonradical}
Let $I\subset \QQrng$ be a zero-dimensional  ideal and $f\in \QQrng$ be such that  $f(\xi)>0$ for all $\xi\in V_{\RR}$.
Let  $B$ be a monomial basis of $\QQrng/I$ with  $|B|=D$.  Then,   there exist   $\omega_k\in \QQ_{\ge 0}$  and $q_k\in \vspan{B}_{\QQ}$, $1\le k\le D$, such that  
\begin{equation*}
{f} \equiv \sum_{k=1}^D \omega_{k}\,q_{k}^{2}\quad   \mod I.
\end{equation*}
\end{theorem}

\begin{proof}
  Define $J:=\sqrt I$. Since $\QQrng/J\simeq (\QQrng/I)\big/(J/I)$,  we can assume that a monomial basis  of $\QQrng/J$ is included in the monomial basis $B$ of $\QQrng/I$. 
  Then, by \Cref{thm:radical}, since $\#(V_\CC(J))=D$, there exists
$\omega_k\in \QQ_{>0}$, $\theta_k\in \vspan{B}_{\QQ}$, $1\le k\le D$, where we can assume   $ \theta_{1}(\zeta)\neq 0$ for all $\zeta\in V_{\CC}$, such that
$$
{f} \equiv  \sum_{k=1}^{D} \omega_{k}\, {\theta}_{k}^{2} \quad \mod J.
$$ 
Let $$\theta:= \frac 1 {\omega_{1}}({f} - \sum_{k=2}^{D} \omega_{k} \, \theta_{k}^{2})\in \QQrng,$$ so that 
$ \theta_{1}^{2}\equiv \theta \ \mod J$ with $ \theta_{1}(\zeta)\neq 0$ for all $\zeta\in V_{\CC}$.
By \Cref{prop:sqrtNewton} applied to $k\in \NN$ such that $J^{2^{k}}\subset I$, there exists $q_{1}\in \QQrng$ such that $q_{1}^{2} \equiv \theta \mod I$. Moreover, by reducing $q_1$ modulo $I$ we can assume that $q_1\in \vspan{B}_{\QQ}$.
Setting $q_k:=\theta_k$ for $2\le k\le D$, we deduce that
$$
f \equiv \sum_{k=1}^{D} \omega_{k}\, {q}_{k}^{2} \quad \mod I
$$
with $q_k\in \vspan{B}_{\QQ}$ for $1\le k\le D$.
\end{proof}

\subsection{Extension to a basic closed semialgebraic set $S$}\label{subs:extension}
We show in this section how to extend Theorem~\ref{th:nonradical} to the case when $S$ is an arbitrary basic closed semialgebraic set defined as in \eqref{eq:s} by a set of polynomials $\bm g=\{g_1,\dots,g_r\}$ and $\bm h=\{h_1,\dots,h_s\}$ with the ideal $I=( \bm h) $  non-necessarily radical. 
More precisely we obtain the following result.

\begin{theorem}\label{thm:epss:finite} Let $\bm g,  \bm h  \subset  \QQrng $ be as in \Cref{eq:gh}, $I = (\vb h) \subset \QQrng$   be a {\em zero-dimensional} ideal,  $B$  be a monomial basis of  $\QQrng \big / I$ with $|B|=D$, and $S\subset \RR^n$ be the finite basic closed semialgebraic set defined as in \eqref{eq:s}.

If $f\in \QQrng$ is such that $f>0$ on $S$,
then  there exist $\omega_{i,k}\in \QQ_{\ge 0}$, $q_{i,k}\in \vspan{B}_\QQ$ and $p_j \in \QQrng$ for $0\le i\le r$, $1\le k\le D$ and $1\le j \le s$ such that  
\begin{equation*}
{f} = \sum_{k=1}^D \omega_{0,k}q_{0,k}^2 + \sum_{i=1}^r \left( \sum_{k=1}^D \omega_{i,k}q_{i,k}^2 \right) g_i + \sum_{j=1}^s p_j \, h_j    
     .
\end{equation*}
Furthermore, if $\vb h$ is a graded basis of $I$ and $\deg(B)$ is an upper bound for the maximum degree of a monomial in $B$, then for all $j$
$$\displaystyle \deg (p_j h_j)  \le \max\{ \deg(f), \deg(g_i)+2\deg(B) \,\colon \,1\le i\le r\}.$$
\end{theorem}
First, in the next lemma we proeed as in \cite[Algorithm~2]{Parrilo2002}, and reduce from the condition  $f(\xi)>0$ for all $\xi\in S$ (where $S$ is the closed basic semialgebraic set described in \eqref{eq:s}) to the condition  $\widetilde f(\xi)>0$ for all $\xi\in V_\RR(I)$, where $I=(\bm h)$ and $\widetilde f$ is an appropriate perturbation of $f$.

\begin{lemma}\label{lem:perturb} Let $S$ be a  basic closed semialgebraic set as in \eqref{eq:s}, and assume that $I = (\vb h)$ is zero-dimensional. Let 
 $f\in \QQrng$ be such that  $f(\xi)>0$   for all $\xi\in S$. Then there exist $\omega_{i,k}\in \QQ_{\ge 0}$  and $q_{i,k}\in \langle B \rangle$, $1\le i\le r, 1\le k\le D$,
 such that 
$$\varphi=\sum_{i=1}^r \big(\sum_{k=1}^{D} \omega_{i,k}q_{i,k}^2 \big)g_i$$  satisfies
$$(f-\varphi)(\xi)>0 \ \mbox{ for all } \   \xi\in V_{\RR}.$$
\end{lemma}
\begin{proof}
For every $\xi \in  V_{\RR}\setminus S$, there exists $i_{\xi}$ with $1\le i_\xi \le r$ such that $g_{i_{\xi}}(\xi)<0$. Let $\rho_{\xi}\in \RR_{+}$ be such that $f(\xi)-\rho_{\xi} g_{i_{\xi}}(\xi)>0$.
We define
$$ 
\widetilde{\varphi} := \sum_{\xi\in V_\RR\setminus S} \rho_{\xi}    u_{\xi}^{2} \, g_{i_{\xi}}  
$$
where $ u_{\xi}\in \vspan{ B}_{\RR}$ are the idempotents associated to $\xi \in V_\RR\setminus S$.
Then
\begin{itemize}
\item $(f-\widetilde{\varphi})(\xi)= f(\xi)>0$ for $\xi\in S$,
\item $(f-\widetilde{\varphi})(\xi)= f(\xi)- \rho_{\xi} g_{i_{\xi}}(\xi)>0$ for $\xi\in V_\RR\setminus S$.
\end{itemize}
Now, since $\QQ $ is dense in $\RR$ we can approximate $\rho_\xi$ by some rational value and $u_\xi$ by some rational polynomial of the same degree, for all $\xi\in V_\RR\setminus S$, to obtain a polynomial $\varphi\in \QQ[\bm x]$  as stated such that 
$(f-\varphi)(\xi)>0$ for all $\xi\in V_\RR$ still hold.
\end{proof}

\begin{proof}[Proof of Theorem~\ref{thm:epss:finite}]
Let $\varphi$ be the polynomial defined in \Cref{lem:perturb} and let
 $\widetilde{f} = f - \varphi$.
As $\widetilde{f}(\xi)>0$ for all $\xi\in V_{\RR}$, by \Cref{th:nonradical} there exists
$\omega_{0,k}\in \QQ_{\ge 0}$, ${q}_{0,k}\in \langle B\rangle_\QQ $, $1\le k\le D$,  such that
$$
\widetilde{f} \equiv  \sum_{k=1}^{D} \omega_{0,k}\, q_{0,k}^{2} \quad \mod I.
$$ 
Thus  $f\equiv  \sum_{k=1}^{D} \omega_{0,k}\, q_{0,k}^{2}+\varphi \ \mod I $, i.e.  $f- \sum_{k=1}^{D} \omega_{0,k}\, q_{0,k}^{2}-\varphi \ \in \ I $, and since 
$$\deg( f -  \sum_{k=1}^{D} \omega_{0,k}\, q_{0,k}^{2}-\varphi)\le  \max\big\{\deg(f), 2\deg(B)+\max_i\{\deg(g_i)\}\big\}, $$
 when $\bm h$ is a graded basis of $I$, we can take $$f -  \sum_{k=1}^{D} \omega_{0,k}\, q_{0,k}^{2}-\varphi = \sum_{j=1}^s p_j h_j$$ with \[\deg (p_j h_j) \le \max\{\deg(f), 2\deg(B)+\max_i\{\deg(g_i)\}\}.\qedhere\]
\end{proof}

\section{SoS representation for $f$ nonnegative  on $S$} 
\label{sec:fnonnegative}
Here we consider Problem 2 stated in \Cref{sec:intro}, whose answer is summarized in \Cref{thm:A}, when $f$ is nonnegative on $S$ and $(I:f)+(f)=(1)$. 

We first observe that if $\kk$ is a subfield of $\RR$ and   $f\in \kkrng$ is a polynomial such that $f(\xi)\ge 0$ for all $\xi\in S$, it might happen that  
$f$ does not belong to the associate quadratic module $Q_{\kk}$, as shown by the following example (see also \cite[Remark 1]{Parrilo2002}).

\begin{example}\label{exp:x3} Let $I=(x^{2})\subset \kk[x]$ and $f=x$. We have $f(\xi)\ge 0$ for $\xi\in S=\{0\}$. But $f \not\in Q_{\kk} = Q_{\kk}(\pm \, x^2)$ (see definition \eqref{eq:quadratic module}). Otherwise, there would exists $\sigma_{0}\in \Sigma^{2}_\kk$ and $p\in \kkrng$ such that
\begin{equation*}\label{eq:dec1}
  x = \sigma_{0} + p\cdot x^{2} .
\end{equation*}
This implies that  $\sigma_{0}$, which is a sum of squares of polynomials which  can't have a  constant term,  is   divisible by $x^{2}$.
This contradicts the possible SoS decomposition  and shows that $x \not\in  Q_{\kk}$. We observe that the same argument can be used to show that
$x\not\in  Q_{\kk}(\pm x^k)$ for $k\ge 2$.
\end{example}

However we can give a positive answer under a condition of separability of the zeros of $f$, which is discussed in detail in \Cref{sec:2}.

\begin{theorem}\label{thm:nonneg} Let $\bm g,  \bm h  \subset  \QQrng $ be as in \eqref{eq:gh}, $I = (\vb h) \subset \QQrng$   be a {\em zero-dimensional} ideal,  $B$  be a  monomial basis of  $\QQrng \big / I$ with $|B|=D$, and $S\subset \RR^n$ be the finite basic closed semialgebraic set defined in \eqref{eq:s}.

If $f\in \QQrng$ is such that  $f\ge 0$ on $S$ and $(I:f) + (f) = (1)$,
then  there exist $\omega_{i,k}\in \QQ_{\ge 0}$, $q_{i,k}\in \vspan{B}_\QQ$ and $p_j \in \QQrng$ for $0\le i\le r$, $1\le k\le D$ and $1\le j \le s$ such that  
\begin{equation*}
{f} = \sum_{k=1}^D \omega_{0,k}q_{0,k}^2 + \sum_{i=1}^r \left( \sum_{k=1}^D \omega_{i,k}q_{i,k}^2 \right) g_i + \sum_{j=1}^s p_j \, h_j    
      .
\end{equation*}
Furthermore, if $\vb h$ is a graded basis of $I$ and $\deg(B)$ is an upper bound for the maximum degree of a monomial in $B$, then for all $j$
$$\displaystyle \deg (p_j h_j)  \le \max\{ \deg(f), \deg(g_i)+2\deg(B) \,\colon \,1\le i\le r\}.$$
\end{theorem}
\begin{proof} 
Since $(I:f) + (f) = (1)$, we can find $a\in \QQrng$ and $b\in (I:f)$ such that $b+a\,f=1$.

We reduce $a$ and $b$ modulo $I$ and obtain $\widetilde a, \widetilde b\in \vspan{ B}$ that satisfy $\widetilde b\in(I:f)$
and $\widetilde b + \widetilde a\,f \equiv 1 \ \mod I$.
Set $S^{*}:=\{\xi \in S\,:\,f(\xi)=0\}\subset   V_{\RR}(f)$, 
let $\rho\in \QQ_+$ be such that $\rho + \widetilde a (\xi)>0$ for all $\xi\in S^*$, and 
define $\overline a:=\widetilde a +\rho \,\widetilde b\in \QQ[\bm x]$. 
 Then, since $\widetilde b \,f\in I$ one has 
\begin{equation}\label{eq:b}\widetilde b+\overline a \, f \equiv 1 \ \mod I.\end{equation}
We observe that  $\overline a (\xi)>0$ for all $\xi \in S$. Indeed, for $\xi\in S^*$  Identity~\eqref{eq:b} implies that $\widetilde b(\xi)=1$ when $f(\xi)=0$, and then  $$\overline a(\xi)=\widetilde a(\xi)+\rho\,\widetilde b(\xi)=\widetilde a(\xi)+\rho>0,$$ while for $\xi\in S\setminus S^*$, $\widetilde b\, f\in I$ implies that $\widetilde b(\xi)=0$ since $f(\xi)\ne 0$, and then  $$\overline a (\xi)=1/f(\xi)>0.$$ 
We deduce from \Cref{thm:epss:finite} that $$\overline a \equiv \sum_{k=1}^D \omega_{0,k}\, \overline q_{0,k}^{2}+ \sum_{i=1}^r \big(\sum_{k=1}^{D} \omega_{i,k}\overline q_{i,k}^2 \big)g_i\quad   \mod I$$ for some $\omega_{i,k}\in \QQ_{+}$ and $\overline q_{i,k}\in \QQ[\bm x]$. Moreover, Identity \eqref{eq:b} and $\widetilde b \, f\in I$ imply that 
$$f\equiv \overline a\, f^2 = \sum_{k=1}^D \omega_{0,k}\, (\overline q_{0,k}f)^{2}+ \sum_{i=1}^r \big(\sum_{k=1}^{D} \omega_{i,k}(\overline q_{i,k} f)^2 \big)g_i\ \mod I.$$  Now let $\overline q_{i,k}f\equiv q_{i,k} \ \mod I$ with $q_{i,k}\in \langle B\rangle $.  Then
$$f\equiv \sum_{k=1}^D \omega_{0,k}\, q_{0,k}^{2}+ \sum_{i=1}^r \big(\sum_{k=1}^{D} \omega_{i,k}q_{i,k}^2 \big)g_i\ \mod I$$
with  $\deg(q_{i,k})\le \delta$.

Finally, when $\bm h$ is a graded basis of $I$ we deduce the degree bounds as in the proof of \Cref{thm:epss:finite}.
\end{proof}

In \cite{KMS2023} the authors showed that for the univariate case, the condition $h/\gcd(f,h)$ and $f$  being relatively prime is a sufficient condition in order to guarantee that if $f$ is a polynomial such that $f(\xi)\ge 0$ for all $\xi\in V_\RR(I)$, then $f\in Q_\QQ$ (here $I=(h)$). 
The condition $(I:f)+(f)=(1)$ of \Cref{thm:nonneg} is equivalent to that condition in the univariate case.  

\begin{remark}\label{rem:iff_nonnegative_radical}
A direct corollary of this theorem is the characterisation of positive polynomials on $S$, given in \Cref{thm:B}. Indeed if $I$ is radical, we have $(I:f)+(f)=(1)$ since $V_{\CC}(I:f) \cap  V_{\CC}(f)= \emptyset$, and by \Cref{thm:nonneg} any polynomial nonnegative on $S$ has the desired SoS representation.
\end{remark}

If $I$ is not radical the result may not hold, as previously shown in \Cref{exp:x3}.

\section{Height bounds  in the quotient algebra $\QQrng/I$ for $I$ radical}\label{sec:heightradical}

Sections \ref{sec:6} and \ref{sec:7} below deal with the proof of \eqref{thm:C}, where we compute height bounds for all the expressions appearing in the SoS representation \Cref{eq:main_result} presented in \Cref{thm:A} in the case when $K=\QQ$ and $I$ is a zero-dimensional {\em radical} ideal in $\QQrng$. Without loss of generality, we assume in the sequel that $f$, $\bm g=\{g_1,\dots,g_r\}$ and $\bm h=\{h_1,\dots,h_s\}$ have all {\em integer} coefficients.

As a preparation for the proof of \Cref{thm:C}, in this section we derive some general bounds for the heights of polynomials in the quotient algebra $\QQrng/I$ when $I$ is a zero-dimensional radical ideal. We add the following notations to the ones already introduced in \Cref{sec:preliminaries}.
\begin{itemize}
\item For $x\in \CC$, we denote $\h(x):= \log (|x|)$ with $\h(0)= -\infty$, where $\log$ denotes logarithm in base 2 (which coincides with the usual (logarithmic) {\em height} of $x$,  essentially its number of binary digits, when  $x\in \ZZ$).
\item Given $p=\sum_\alpha a_\alpha\bm x^\alpha \in \CCrng$ or $\bm A=(A_{i,j})_{i,j}\in \CC^{N\times M}$, we denote
$$
\h(p):=\max_\alpha\{\h(a_\alpha)\} \quad \mbox{and}  \quad \h(\bm A):=\max_{i,j}\{\h(A_{i,j})\}.
$$
Note that if $p\in \ZZrng$,
$\h(p)$ coincides with 
the  {(logarithmic) height} $\hh(p)$ of the polynomial $p$, and the same holds for a matrix $\bm A\in \ZZrng$. 
However, when $p\in \QQrng$ (or $\bm A\in \QQ^{N\times M}$), these two quantities differ, as $\hh(p)=\max\{\hh(\widehat p),\hh(\nu)\}$, where $p=\widehat p/\nu$ with  $\widehat p\in \ZZrng$ a primitive polynomial associated to $p$ and $\nu\in \NN$ (analogously for $\bm A$). To avoid misunderstandings, we will not use in the sequel the  notation $\h$ for a rational non-integer polynomial or matrix, but instead refer to $\h$ of integer numerators and a common denominator. 
\item Finally, given $n,d\in \NN$ we denote $$\Cn := c\,n\log(n+1)\quad \mbox{and}\quad \Cnd := c\,n\log(d(n+1))$$ for some  suitable positive computable constant $c\in \RR_{>0}$ (in particular, $\C1=c$).
\end{itemize}

\subsection{Bounds on the values of a polynomial at the roots}
The following lemma applies an arithmetic B\'ezout theorem for the height of a variety in terms of the heights of its defining polynomials that is developed  in \cite{KPS01}.

\begin{lemma}\label{lem:htcoord} Let  $J=(h_1, \dots, h_{s+1})$ be a zero-dimensional ideal defined by polynomials $h_1,\dots, h_{s+1}\in \ZZrng$  with  $\deg(h_j)\le d$ for $1\le j\le s$, $e=\max\{d,\deg(h_{s+1})\}$ and $\h(h_j)\le \tau$ for $1\le j\le s+1$. Then
\begin{equation*}\sum_{\zeta\in V_\CC(J)}\h(\|(1,\zeta)\|_2)\le   \Cn d^{n-2}e(d+\tau). \end{equation*}
\end{lemma}
\begin{proof} The arithmetic B\'ezout theorem in \cite[Cor.2.11]{KPS01}, more precisely the comment after it applied to the polynomials $h_1,\dots,h_s\in \QQrng$ which define the variety $V$, implies that the (logarithmic) global 
height $\hh(V)$ of $V$ satisfies
$$
\hh(V_\CC(J))\le 2n\log(n+1)d^{n-1}e+nd^{n-2}e\tau \le \Cn d^{n-2}e(d+\tau).
$$
 Now, the global height $\hh(V)$ is defined as the sum of its local heights, see  
\cite[Sec.1.2.4]{KPS01}, which implies that in particular the local height $\hh_\infty$ corresponding to the ordinary complex absolute value, satisfies 
$$
\sum_{\zeta\in V_\CC(J)}\hh_\infty (\zeta)=\hh_\infty(V_\CC(J))\le\hh(V_\CC(J)),
$$
while by  \cite[Sec.1.2.3]{KPS01}, $\hh_\infty(\zeta)=\h(\|(1,\zeta)\|_2)$.
\end{proof}

We can now bound the value of a polynomial at the roots $\zeta\in V_{\CC}$.
\begin{lemma}\label{lem:htpzeta} Let $I=(\bm h)$  be a zero-dimensional ideal defined by polynomials $h_1,\dots, h_s\in \ZZrng$ with  $\deg(h_j)\le d$ and $\h(h_j)\le \tau$. Let $p\in \ZZ[\bm x]$ be such that $\h(p)\le \tau$ and let $d_p= \max\{d, \deg(p)\}$.
 Then, for any $Z\subset V_\CC$ such that $p(\zeta)\ne 0$, $\forall\, \zeta\in Z$, we have
$$
- \Cn d^{n-1}d_p (d+\tau) \le \sum_{\zeta\in Z}\h(p(\zeta)) \le \Cn d^{n-1}d_p (d+\tau)
$$
\end{lemma}
\begin{proof}
We apply \Cref{lem:htcoord}  to the roots of the zero-dimensional ideals of $\ZZ[\bm x, x_{n+1}]$:
\begin{align*}&I_1=I+(x_{n+1}-p(\bm x))\ \subset \ZZ[\bm x,x_{n+1}], \mbox{ with } \ V_\CC(I_1)=\{(\zeta,p(\zeta)):\, \zeta\in V_\CC\}, \\
&I_2=I+(1-x_{n+1}p(\bm x))\  \subset \ZZ[\bm x,x_{n+1}] , \mbox{ with } \ V_\CC(I_1)=\{(\zeta,1/p(\zeta)):\, \zeta\in V_\CC, p(\zeta)\ne 0\}.\end{align*}
For $I_1$ this gives by \eqref{lem:htcoord} that 
\begin{align*} \sum_{\zeta\in Z}\h(p(\zeta)) & \le \sum_{\zeta\in V_\CC}\h(\|(1,\zeta,p(\zeta))\|_2) \\ 
&\le \Cn d^{n-1}d_p(d+\tau).
\end{align*}
Similarly, for $I_2$,  $$ \sum_{\zeta\in Z}\h\big(\frac{1}{p(\zeta)}\big)
\le \Cn d^{n-1} (d_p +1)(d+\tau)
\le \Cn d^{n-1}d_p(d+\tau)$$
since $d_p\ge d\ge 1$.
\end{proof}

\subsection{Reduction in the quotient algebra $\QQ[\bm x]/I$}\label{subs:52}
The goal of this section is to compute bounds for the height of an integer polynomial $p\in \ZZrng$ when reducing it modulo a zero-dimensional {\em radical} ideal $I$. 
We will achieve this by passing  through the description of the quotient algebra $\QQrng/I$ via its rational univariate representation (also known as geometric resolution), in order to obtain better bounds than if we were working directly with multivariate reduction by the graded basis $\bm h$ of the ideal $I$.

\begin{lemma}\label{lem:rur-height}
Let  $J=(h_1, \dots, h_s, h_{s+1})$ be a zero-dimensional {\em radical}  ideal  defined by polynomials $h_1,\dots, h_s, h_{s+1}\in \ZZrng$  with $\deg(h_j)\le d$ for $1\le j\le s$, $e=\max\{d,\deg(h_{s+1})\}$ and $\h(h_j)\le \tau$ for $1\le j\le s+1$. Then there exists $w_0, w_1, \ldots, w_n \in \ZZ[t]$ with $\deg(w_0)= D := \dim \QQrng/J$, $\deg(w_i)< D$ for $1\le i\le  n$ such that 
\begin{eqnarray*}
    \varphi_J: \QQrng/J & \longrightarrow & \QQ[t]/(w_0)\\
          x_i & \longmapsto & w_i (w_0')^{-1}
\end{eqnarray*}
is an isomorphism of algebras and
$$ 
\h(w_i) \le \Cnd d^{n-2}e(d+\tau),
 \quad 0\le i\le n.
$$
\end{lemma}
\begin{proof}
Let $L(U,\bm x)=U_1x_1+\cdots +U_nx_n\in \QQ[U,\bm x]$ be a generic linear form, and consider the polynomial 
$$
W_0(t,U)=\Chow_V(t, -U_1, \ldots, -U_n) = a\prod_{\zeta\in V}(t-L(U,\zeta))\in  \ZZ[t][U],
$$ 
where $V:=V_\CC(J)$ and $\Chow_V$ is the Chow form of $V$, with $a\in \NN$ such that $W_0$ is a primitive polynomial (this is possible since $V$ is defined over $\QQ$).

Given  $\zeta\in V$,  $W_0(L(U,\zeta), U)=0$ implies that for all $i$  we have
$$\frac{\partial W_0}{\partial t}(L(U,\zeta), U)\zeta_i+ \frac{\partial W_0}{\partial U_i}(L(U,\zeta),U)=0.$$
Therefore, by choosing $u=(u_1,\dots,u_n)\in \ZZ^{n}$ such that $\ell(\bm x):=L(u,\bm x)$ satisfies $\ell(\zeta)\ne \ell(\xi)$ for all $\zeta\ne \xi \in V$,
the polynomials  $$w_0(t)=W_0(t,u),\ w_i(t)=-\frac{\partial W_0}{\partial U_i}(t,u)\, \in \ZZ[t]$$ give  the stated isomorphism of algebras, where we observe that the map $\varphi_J$ is well-defined since $J$ is radical so that $w_0$ has simple roots in $\CC$ and $w_0'$ is invertible modulo $w_0$.

In order to choose $\ell(\bm x)=u_1x_1+\cdots + u_nx_n\in \ZZ[\bm x]$ that separates the points in $V$, we can observe that the non-zero polynomial 
$$
P(U):=\prod_{(\zeta,\xi)\in V\times V, \zeta\ne \xi}(L(U,\zeta)-L(U,\xi))\in \CC[U]
$$
has degree $D':=D(D-1)/2$ and therefore, there exists an element $u$ of the grid $$\{(k_1,\dots,k_n)\in \ZZ^{n}\,:\, 0\le k_i\le  D'\}$$
where $P$ does not vanish.

By \cite[Section 1.2.4, Inequalities (1.1) and (1,2)]{KPS01} and \Cref{lem:htcoord} we have 
$$
\hh(W_0)=\hh(\Chow_V)\le \hh(V)+\log(n+1)\deg(V) \le \Cn d^{n-2} e (d + \tau ).
$$
We conclude by observing that the polynomial
$$
W_0(t,U)=\sum_{\alpha,i:|\alpha|+i=D}a_{\alpha,i}t^iU^\alpha =\sum_{i=0}^D t^i \big(\sum_{|\alpha|=D-i} a_{\alpha,i}U^\alpha\big)\in \ZZ[t,U]$$ is homogeneous of degree $D$ and  $-\frac{\partial W_0(t,U)}{\partial U_i}\in \ZZ[t,U]$  is homogeneous of degree $D-1$, and therefore 
each of the  (integer) coefficients of $w_0(t)$ has absolute value bounded by  
$2^{n+D} 2^{\h(W_0)}D'^{D}$, while  each of the integer coefficients of $w_i(t)$, $1\le i\le n$, has absolute value bounded by $2^{n+D-1} 2^{\h(\frac{\partial W_0(U,t)}{\partial U_i})}D'^{D-1}$. This adds a factor $2\,n\,d^n\log(d)$ that we can summarize in 
$$ \ \hspace{4.25cm}  \h(w_i)\le \Cnd d^{n-2}e(d+\tau), \quad 0\le i\le n. \ \hspace{4.25cm} \qedhere$$
\end{proof}

Consider now a zero-dimensional (and radical) ideal $I=(h_1, \dots, h_s)$  defined by $h_1,\dots, h_s\in \ZZrng$  with $\deg(h_j)\le d$, $\h(h_j)\le \tau$.
We define the following crucial map $\cl U $, which essentially computes a univariate normal form for any polynomial $p\in \QQrng$:
\begin{eqnarray}\label{eq:defN2}
    \cl U :\ \QQrng & \longrightarrow & \QQ[t]/(w_0)\simeq \vspan{1, \ldots, t^{D-1}}_{\QQ}\\
          p & \longmapsto & w_0' \, \varphi_I(\pi_I(p)) \nonumber
\end{eqnarray} where  $\pi_I:\QQrng \to \QQrng/I$ is the projection. 
Since $\varphi_I$ is an isomorphism and $w_0'$ is inversible in $\QQ[t]/(w_0)$, the map $\cl U$  is surjective since $\pi_I$ is and satisfies $\ker \cl U= I$.

\begin{lemma}\label{lem:height-N}
Let $I=(\bm h)$  be a zero-dimensional radical ideal defined by polynomials $h_1,\dots, h_s\in \ZZrng$ with  $\deg(h_j)\le d$ and $\h(h_j)\le \tau$. Let $p\in \ZZ[\bm x]$ be such that $\h(p)\le \tau$ and let $d_p = \max\{d, \deg(p)\}$. Then, $\cl U(p) \in \vspan{1, \ldots, t^{D-1}}_{\ZZ}$ and
$$
\h(\cl U(p)) \le  \Cnd d^{n-1}d_p(d+\tau).
$$
\end{lemma}
\begin{proof}
Let $\varphi_I: \QQ[\bm x]/I\to \QQ[t]/(w_0)$, $x_i\mapsto w_i(w'_0)^{-1}$, be the isomorphism of \Cref{lem:rur-height}. We observe that the linear form $\ell(\bm x)=L(u,\bm x)$ that defines the polynomial $w_0(t)=W_0(u,t)$ in the proof of that lemma is also a separating form for the vanishing set $V_\CC(J)=\{(\zeta,p(\zeta)): \zeta\in V_\CC(I)\}$ of the ideal  
$J:=(h_1(\bm x), \ldots, h_s(\bm x), x_{n+1} - p(\bm x)) \subset \QQ[\bm x, x_{n+1}]$.
Moreover, it is straightforward to verify that under this isomorphism, we have $\varphi_J(x_i)=w_i(w'_0)^{-1}$ for $1\le i\le n$. Let $w_{n+1} \in \ZZ[t]$ be such that $\varphi_J(x_{n+1})=w_{n+1}(w'_0)^{-1}$. Then
$$\cl U (p)=w_0'\, \varphi_{I}(\pi_I(p))=w_0' \, \varphi_{J}(\pi_J(p))=w_0' \varphi_J(x_{n+1})= w_{n+1}$$
since $p\in \ZZrng$ and $x_{n+1}-p(\bm x)\in J$.
This implies that $\cl U (p)\in   \vspan{1, \ldots, t^{D-1}}_{\ZZ}$ and the stated bound by application of 
\Cref{lem:rur-height} to $J\subset \QQ[\bm x,x_{n+1}]$.
\end{proof}

\begin{corollary}\label{cor:height-N'}
Let $I=(\bm h)$  be a zero-dimensional radical ideal defined by polynomials $h_1,\dots, h_s\in \ZZrng$ with  $\deg(h_j)\le d$ and $\h(h_j)\le \tau$. Let $p\in \ZZ[\bm x]$ be such that $\h(p)\le \tau_p$ and let $d_p = \max\{d, \deg(p)\}$. Then, $\cl U(p) \in \vspan{1, \ldots, t^{D-1}}_{\ZZ}$ and
$$
\h(\cl U(p)) \le  \Cnd d^{n-1}d_p(d+\tau)  + \tau_p.
$$
\end{corollary}
\begin{proof}
Let $p= \sum_{|\alpha|\le d_p} p_{\alpha} \bm x^{\alpha}$ with $p_{\alpha} \in \ZZ$ and $\h(p_{\alpha})\le \tau_p$. 
Then $\cl U(p)= \sum_{|\alpha|\le d_p} p_{\alpha} \, \cl U(\bm x^{\alpha})$, where 
 by \Cref{lem:height-N}, $\h(\cl U(\bm x^{\alpha})) \le \Cnd d^{n-1} d_p (d+\tau)$. Since  the number of coefficients of $p$ is bounded by $(d_p+1)^n$, we have
$$ 
\h(\cl U(p)) \le \Cnd d^{n-1} d_p (d+\tau) + {\cl C}(n;d_p) + \tau_p \le \Cnd d^{n-1} d_p (d+\tau) + \tau_p,
$$
by adjusting the term $\Cnd$.
\end{proof}

We will now bound the height  of  the {\em normal form} $\cl N(p)\in \vspan{B}_{\QQ}$ of a polynomial $p\in \ZZrng$, which is the {\em unique}  polynomial $\cl N(p)$ in $\vspan{B}_{\QQ}$  such that 
\begin{equation}\label{eq:NF}
p =  p_1 h_1+ \cdots + p_s h_s +\cl N(p)
\end{equation}
for some $p_j\in \QQrng$, $1\le j\le s$ 
(we recall that here that $B\subset \ZZrng$ is a  monomial basis of $\QQrng/I$ with   $\deg(B) = \max\{\deg(b):\, b \in B \}$). To do that, we are going to exploit the properties of $\cl U$, defined in \eqref{eq:defN2}.

\begin{proposition}\label{prop:Creduction'}
Let $I=(\bm h)$  be a zero-dimensional radical ideal defined by polynomials $h_1,\dots, h_s\in \ZZrng$ with  $\deg(h_j)\le d$, $\h(h_j)\le \tau$, let $p\in \ZZ[\bm x]$ with $d_p = \max\{ \deg(p),d\}$ and $\h(p)\le \tau_p$, and set $\delta:= \max\{d, \deg(B)\}$.   Then, $\cl N(p)= \widehat{\cl N}(p)/{\nu}$ where $\nu\in \NN$ and $\widehat{\cl N}(p) \in \langle B\rangle_\ZZ$ satisfy 
$$
\h(\nu)\le \Cnd d^{2n-1}\delta\, (d+\tau)\quad \mbox{and} \quad
\h(\widehat{\cl N}(p))\le \Cnd d^{n-1} (d^n \delta+ d_p )\, (d+\tau)+ \tau_p .$$
\end{proposition}
\begin{proof}
Let $\cl N(p) = \sum_{i=1}^{D} c_ib_i $ where  $B=\{b_1, \ldots, b_D\}$ and   $c_i \in \QQ$, $1\le i\le D$. Then we have $\cl U(p)=\cl U(\cl N(p))$ since $p\equiv \cl N (p)\ \mod I$. Therefore 
$$
\cl U(p)= \sum_i  c_i \,\cl U(b_i)
$$
where $\bm c = (c_1, \ldots, c_D)$ is the solution of the system $\bm A \bm c=\bm b$, with $\bm A\in \ZZ^{D\times D}$ the invertible matrix of coefficients of $\cl U(b_i) \in \vspan{1, \ldots, t^{D-1}}_{\ZZ}$ and $\bm b \in \ZZ^{D}$ the coefficient vector of $\cl U(p)$.

By \Cref{cor:height-N'}, we have
$$
\h(\bm A) \le  \Cnd d^{n-1} \delta (d+\tau), \quad 
\h(\bm b) \le  \Cnd d^{n-1}d_p  (d+\tau)  + \tau_p.
$$
By  \Cref{rem:HadamardD}, we deduce that $\bm c =  \widehat {\cl N}(p)/\nu$ with $\nu \in \NN$, $\widehat{\cl N}(p)\in \vspan{1, \ldots, t^{D-1}}_{\ZZ}$, and 
\begin{align*}
   & \h(\nu)\le D\big(\log(D)+\Cnd  d^{n-1}  \delta (d+\tau)\big) \le \Cnd  d^{2n-1} \delta (d+\tau),\\
   &  \h(\widehat {\cl N}(p))  \le (D-1)\big(\log(D) +\Cnd d^{n-1}\, \delta \,(d+\tau)\big) + \Cnd d^{n-1} d_p (d+\tau) + \tau_p\\ & \qquad \ \quad \,  \le \Cnd d^{n-1} (d^n \delta+ d_p )\, (d+\tau)+ \tau_p
\end{align*}
by adjusting $\Cnd$, which concludes the proof.
\end{proof}

\begin{remark}\label{rem:complexreduction}
  \Cref{prop:Creduction'} remains true for a polynomial $p\in \CCrng$. We get $\cl N (p)\in \langle B\rangle _\CC$ with 
  $$\h({\cl N}(p))\le   \Cnd d^{n-1}(d^n \delta+d_p)(d+\tau)+ \tau_p .$$
  This is because in the proof,  when solving the linear  system $\bm A \bm y=\bm b$ where $\bm A\in \ZZ^{D\times D}$ and $b\in \CC^D$,  $\nu\in \NN $ satisfies $\nu \ge 1$.
\end{remark}
We will also need to  bound the height of the polynomials $p_1,\dots,p_s\in \QQrng$  in a decomposition as in \eqref{eq:NF}, assuming that $\vb{h} = \{\, h_1,\dots,h_s \,\}$ is a \emph{graded basis} of $I$. The bound we obtain for $\h(\cl N(p))$ has a worse dependence on the degree of $p$ than in \Cref{prop:Creduction'}, and in the sequel we will not apply it, but this is the only bound we get for the polynomial coefficients $p_j$, $1\le j\le s$.
\begin{proposition}\label{prop:reduction}
Let  $h_1, \ldots, h_s$ be a graded basis of $I$ with $\deg(h_j)\le d$ and $\h(h_j)\le \tau$, $1\le j\le s$,  and let $B= \{b_1, \ldots, b_D\}$ be a monomial basis of $\QQrng/I$.
Let $p\in \ZZrng$ and set $d_p:= \deg(p)+1$ and $\h(p)\le \tau_p$. Then there exist $\nu\in \NN$, $ \lambda_j\in \ZZ$, $1\le i\le D$,  and $p_j\in \ZZrng$ with $\deg(p_j)\le \deg(p)-\deg(h_j)$, $1\le j \le s$,  such that 
$$
p =\frac{1} {\nu} \big(\sum_{j=1}^s p_j h_j +\sum_{i=1}^{D} \lambda_i b_i   \big)
$$ where $\cl N(p)= \dfrac{1}{\nu}(\sum_{i=1}^D \lambda_i b_i)$, satisfying 
$$
\h(\nu) \le  \cst(n; d_p) d_p^n \tau \quad\mbox{and}\quad 
\h(\lambda_i), \h(p_j) \le  \cst(n; d_p) d_p^n \tau  + \tau_p\ \quad{ for } \ 1\le i\le D, 1\le j\le s.
$$
\end{proposition}
\begin{proof}
To compute the vector $(\lambda_1, \ldots, \lambda_D)$ and all coefficients of all polynomials $p_j$, we solve a linear system $\bm A \bm y =\bm b$ where  $\bm A\in \ZZ^{M\times N}$ which non-zero entries are coefficients of the $b_i$'s and the polynomials $h_j$, so that $\h(\bm A)\le \tau$, and $\bm b\in \ZZ^M$ is the vector of coefficients of $p$, so that $\h(\bm b)\le \tau_p$. 
Since the number of coefficients of $p$ is at most $d_p^n$ and we can assume that $\bm A$ is of full rank $M$, we have $M\le d_p^n$. By applying \Cref{rem:HadamardD} we obtain $\nu\in \ZZ$ and the $\lambda_i\in \ZZ$ and the coefficients of the polynomials $p_j\in \ZZ[\bm x]$ with
\begin{align*}
& \h(\nu)  \le  \frac{M}{2}\log(M) + M\tau\le \cst(n; d_p) d_p^n \tau ,\\
& \h(\lambda_i), \h(p_j)  \le  \frac{M}{2}\log(M)+(M-1)\tau  + \tau_p \le \cst(n; d_p) d_p^n \tau + \tau_p. \qedhere
\end{align*}
\end{proof}

\subsection{Bounds on the singular values of the Vandermonde matrix}
\noindent{}Let  $I$ be a zero-dimensional  radical ideal and 
let $V = V_{B,\bms{\zeta}}\in \CC^{D\times D}$ be the Vandermonde matrix of the roots $V_\CC(I)=\{\zeta_1, \ldots, \zeta_D \}$ in the monomial basis $B$ of $\QQrng/I$, with $|B|=D$, and let  
 $U=(u_{\zeta_1}, ..., u_{\zeta_D})\in \CC^{D\times D}$ be the coefficient matrix of the interpolation polynomials $u_{\zeta_i}\in \CCrng$ in the basis $B$. By definition, $V^T U = \textup{Id}_D$. We denote by $\sigma_{\max}(\, \cdot \,)$ and $\sigma_{\min}(\, \cdot \,)$ respectively the  maximal and minimal singular values of a matrix.  

\begin{lemma}\label{lem:VdM} Let $I=(\bm h)$  be a zero-dimensional radical ideal defined by polynomials $h_1,\dots, h_s\in \ZZrng$ with  $\deg(h_j)\le d$, $\h(h_j)\le \tau$, and set $\delta:= \max\{d, \deg(B)\}$. Then
$$
\h(\sigma_{\max}(V)) \le  
\Cn \,d^{n-1}\delta (d+\tau)\quad \mbox{and}\quad 
\h(\sigma_{\min}(U)) \ge 
- \Cn d^{n-1}\delta (d+\tau).
$$
\end{lemma}
\begin{proof}

We have the well-known inequalities
$$
\sigma_{\max}(V) \le  \|V\|_F\le D \max\{ |\zeta^\alpha| \ : \ \bm x^\alpha\in B,\,  \zeta\in V_\CC\}  $$
where $\|\cdot \|_F$ denotes the Frobenius norm of matrix $V$. 
We conclude by applying  Lemma~\ref{lem:htpzeta} to the polynomials $\bf x^\alpha\in \ZZ[\bf x]$
which satisfy $\deg(\bm x^\alpha)\le \delta $ and $\h(\bm x^\alpha)=0\le \tau$:
$$\h(\sigma_{\max}(V))\le \log(D)+ \Cn d^{n-1}\delta (d+\tau) \le \Cn d^{n-1} \delta (d+\tau) 
$$ 
since $D\le d^n$.

\smallskip
\noindent As $U^T V= \mathrm{Id}$, we deduce the lower bound on $\sigma_{\min}(U)= (\sigma_{\max}(V))^{-1}$.
\end{proof}

We now  give an upper bound for the height of the idempotents $u_\zeta$, $\zeta \in V_\CC(I)$.
For that purpose, we use the following lemma.
\begin{lemma}\label{lem:sep} Let $I=(\bm h)$  be a zero-dimensional ideal defined by polynomials $h_1,\dots, h_s\in \ZZrng$ with  $\deg(h_j)\le d$ and $\h(h_j)\le \tau$. Then, for any  $i$, $1\le i\le n$,  and  $Z\subset V_\CC^2$ with  $\zeta_i\ne \xi_i$ for all $(\zeta,\xi)\in Z$ with $\zeta=(\zeta_1,\dots,\zeta_n)$ and $\xi=(\xi_1,\dots,\xi_n)$, one has
$$
\sum_{(\zeta,\xi)\in Z}\h(\zeta_i-\xi_i)\ge - \Cn d^{2n} (d+\tau).
$$
\end{lemma}

\begin{proof}
 We apply \cref{lem:htpzeta} to the zero-dimensional  ideal $J=(\bm h(\bm x),\bm h(\bm y))\subset \CC[\bm x,\bm y]$ which satisfies $V_\CC(J)=V_\CC^2$, and the polynomials $p_i(\bm x,\bm y)=x_i-y_i$, $1\le i\le n$, of degree 1 and height 0. 
\end{proof}

\begin{definition}[Interpolation polynomials]\label{def:intpol}
For each $\zeta\in V_\CC$, let $\varphi_\zeta\in \CCrng$ be  defined in the following  way:
    For each $\xi\ne \zeta$ pick a coordinate $\xi_{i_\xi}$ with $\xi_{i_\xi}\ne \zeta_{i_\xi}$ and let 
    $$
    \varphi_{\zeta}=\prod_{\xi\ne \zeta}\dfrac{x_{i_\xi}-\xi_{i_\xi}}{\zeta_{i_\xi}-\xi_{i_\xi}} \ \in \CCrng.
    $$ 
\end{definition}
We verify that $\varphi_\zeta(\zeta)=1$ and $\varphi_{\zeta}(\xi)=0$ for $\xi\ne \zeta$.

\begin{lemma}\label{lem:phizeta2} Let $(\varphi_\zeta)_{\zeta \in V_{\CC}(I)}$ be the interpolation polynomials of \Cref{def:intpol}. Then 
$$
\deg(\varphi_\zeta)=D-1 \quad \mbox{and}\quad \h(\varphi_{\zeta} )\le \Cn d^{2n} (d+\tau).
$$
\end{lemma}
\begin{proof} The degree bound is obvious by the definition.

For the height bound,
the denominator $\prod_{\xi\ne \zeta}(\zeta_{i_\xi}-\xi_{i_\xi})$ of $\varphi_\zeta$ is composed by at most $n$ products of terms of the form 
    $$\prod_{(\zeta,\xi)\in Z_i} (\zeta_i-\xi_i),$$ 
    where $Z_i\subset V_\CC^2$ is the subset of $\xi\ne \zeta$ that we chose such that $i_\xi=i$, $1\le i\le n$.
   Each of these terms satisfies the lower bound of \Cref{lem:sep}. This implies the bound
    $$\sum_{\xi\ne \zeta}\h(\zeta_{i_\xi}-\xi_{i_\xi})\ge -\Cn d^{2n} (d+\tau).$$
    The coefficients of the  numerator of $\varphi_{\zeta}$ are the elementary symmetric polynomials $s_\alpha$ on the $\xi_{i_\xi}$, which satisfy 
$$|s_\alpha|\le 2^{D-1}\max\{1,|\zeta_i|\,:\,\zeta\in V_\CC\}^{D-1}$$ and therefore 
    $$\h(s_\alpha)\le (D-1)(1+\max\{\h(\zeta_i)\,:\,\zeta\in V_\CC\}\le \Cn d^{2n-1}(d+\tau)$$ by Lemma~\ref{lem:htcoord}. This implies that each of the interpolant $\varphi_{\zeta}$ of degree $D-1$ satisfies
    \[\h(\varphi_{\zeta})\le \Cn d^{2n} (d+\tau).\qedhere\]
\end{proof}

\begin{corollary} \label{cor:hu} Let $u_\zeta\in \langle B\rangle _\CC$, $\zeta \in V_\CC$, be the idempotents corresponding to $V_\CC$. Then 
$$ \h(u_\zeta)\le \Cnd d^{2n-1}\delta(d+\tau), \ \forall \, \zeta \in V_\CC.$$
    
\end{corollary}
\begin{proof}
    This is because $u_\zeta=\cl N (\varphi_\zeta)$ with $\deg(\varphi_\zeta)=D-1$ and $\h(\varphi_\zeta)\le \Cnd d^{2n}(d+\tau)$. We  apply \Cref{rem:complexreduction} and obtain
 $$\h(u_\zeta)\le \Cnd d^{2n-1}\delta(d+\tau)+ \Cnd d^{n-1}D(d+\tau) +  \Cn
 d^{2n}(d+\tau)\le \Cnd d^{2n-1}\delta(d+\tau)$$ by adjusting the constant $\Cnd$, since $\delta \ge d$. 
\end{proof}

\section{Height bounds for $I$  radical and $f$ strictly positive on $S$}\label{sec:6}

In this section, we study height bounds for the SoS representation \eqref{eq:main_result} of a polynomial $f\in \ZZrng$ strictly positive on $S=S(\bm g, \bm h)$ as in \eqref{eq:s}, for the case when $I=(\bm h)$ is a radical zero-dimensional  ideal, and $\bm g, \bm h\subset \ZZrng$ with $\deg(g_i), \deg(h_j)\le d$ and $\h(g_i), \h(h_j)\le \tau$ for $1\le i\le r$, $1\le j\le s$.
We keep the notation that $B$ is a monomial basis of $\QQrng/I$ with $D:=|B|$ and $\delta:=\max\{d,\deg(B)\}$.


\subsection{The case when  $S=V_\RR(I)$}

Here we  present height bounds for an expression  \eqref{eq:main_result} for  an arbitrary polynomial $p\in \ZZrng$  in the assumptions of  \Cref{thm:radical}, i.e. besides being   $I$  a radical zero-dimensional ideal, $p>0$ on the whole $V_\RR(I)$. We observe that from Lemmas \ref{lem:sos:nempty} and \ref{lem:sos:empty}
one has 
\begin{equation}\label{eq:thetad}
p\equiv \sum_{\zeta\in V_\CC}\theta^2_\zeta \ \ \mod I_\RR \quad \mbox{with} \quad \sum_{\zeta\in V_\CC}\theta^2_\zeta = B (\Theta \, \Theta^t ) B^t,\end{equation}
 where $\Theta\in \RR^{D\times D}$ is the coefficient matrix of the basis $(\theta_\zeta)_{\zeta\in V_\CC}$ in the basis $B$ of $\QQrng/I$.

However, since  we only deal with {\em radical} zero-dimensional ideals, and therefore we will not perform the Hensel lifting described in \Cref{prop:sqrtNewton}, we do not need the assumption that there exists $\zeta_{0}\in V_{\CC}$  such that $\theta_{\zeta_{0}}(\zeta)\neq 0$  for all $\zeta\in V_{\CC}$  described in the output of \Cref{thm:radical}. This simplifies the description of the  polynomials $\theta_\zeta$ described in Lemmas \ref{lem:sos:nempty} and \ref{lem:sos:empty}. Indeed in this case, if $V_\CC=V_\RR \cup Z \cup \overline Z$, where $Z$ is a representative set of non-conjugate non-real roots on $V_\CC$, it is enough to take
\begin{align}\label{eq:simpl}\bullet & \quad  \theta_{\xi}=\sqrt{\omega_\xi}\,u_\xi \in \vspan{B}_{\RR} \ \mbox{ where } \omega_\xi:=p(\xi) \ \mbox{ for all } \xi \in V_\RR,\nonumber \\
\bullet & \quad 
     \theta_{\zeta}= \sqrt{\omega_\zeta}\Big(\Re( u_{\zeta}) - \dfrac{\Im(p(\zeta))}{ \lambda_\zeta +\Re(p(\zeta))}\, \Im( u_{\zeta})\Big)\in \vspan{B}_{\RR} \  \mbox{ where } \omega_\zeta:=2\,\big(\lambda_\zeta + \Re(p(\zeta))\big) \mbox{    for all } \zeta\in Z, \nonumber \\
\bullet & \quad 
 \theta_{\overline \zeta}= \sqrt{\omega_\zeta } \,\Im( u_{\zeta}) \in \vspan{B}_{\RR}  \  \mbox{ where } \omega_\zeta:=2\dfrac{{ \lambda_\zeta^2-|p(\zeta)|^{2}}}{\lambda_\zeta + \Re(p(\zeta))} \
 \ \mbox{ for all } \zeta\in Z, \end{align}
 with $\lambda_\zeta> |p(\zeta)| $ for all $\zeta\in Z$.

\begin{lemma}\label{lem:omegarho2}
Let $ \wf\in \ZZrng$ be such that  $\wf(\xi)>0$ for all $\xi\in V_\RR$ and $|\h( \wf(\zeta))| \le \otau$ for all $\zeta \in V_{\CC}$ such that $\wf(\zeta)\ne 0$. Given $\zeta\in V_\CC$, let  $\omega_\zeta\in \RR_{>0}$ be defined as in Identities \eqref{eq:simpl}  and given  $ \zeta\in Z$, set $\mu_{\zeta}:= - \dfrac{\Im({\wf}(\zeta))}{ \lambda_\zeta +\Re({\wf}(\zeta))-{\wf}(\xi_{0})}$.
 Then $$|\h(\omega_{\zeta})| \le \C1\otau,\ \forall\, \zeta\in V_\CC,\quad \mbox{and} \quad \h(\mu_{\zeta}) \le \C1\otau, \ \forall\, \zeta\in Z.$$
\end{lemma}
\begin{proof} 
By hypothesis, for all $\xi\in V_\RR$,  $|\h(\omega_{\xi})|= |\h( \wf(\xi))| \le\otau$.

For $\zeta\in Z$, choosing $\lambda_{\zeta}$ such that 
\begin{equation}\label{eq:lambda}
|{ \wf}(\zeta)|+1 < \lambda_{\zeta}<|{ \wf}(\zeta)|+2,
\end{equation}
we have 
$$
1 < \lambda_\zeta - |\Re({ \wf}(\zeta))|< \lambda_\zeta + \Re({ \wf}(\zeta))< 2\, |{\wf}(\zeta)| +2
$$
which implies that 
$|\h(\omega_{\zeta})|\le \C1\otau
$.
Similarly, 
$$
2\dfrac{{ \lambda_\zeta^2-|{ \wf}(\zeta)|^{2}}}{\lambda_\zeta + \Re({ \wf}(\zeta))}  
\ge  \frac{ \lambda_\zeta + |{ \wf}(\zeta)|}{ |{ \wf}(\zeta)| + 1}
\ge \frac{1} { |{ \wf}(\zeta)| + 1}
$$
and 
$$
2\dfrac{{ \lambda_\zeta^2-|{ \wf}(\zeta)|^{2}}}{\lambda_\zeta + \Re({ \wf}(\zeta))}  \le  4 (\lambda_\zeta + |{ \wf}(\zeta)|)  \le  8 (|{ \wf}(\zeta)|+1).
$$
We deduce that for all $\zeta\in Z$, 
$
|\h(\omega_{\bar{\zeta}})|\le \C1\otau
$.

The bound for $\mu_{\zeta}$ for $\zeta\in Z$   follows from the fact that  by \Cref{eq:lambda} 
$
|\mu_{\zeta}| \le { |{\wf}(\zeta)| }$, which implies that $\h(\mu_{\zeta})\le \C1\otau$.
\end{proof}

\begin{lemma}\label{cor:precround2}
Let $ \wf\in \ZZrng$ be such that  $\wf(\xi)>0$ for all $\xi\in V_\RR$ and $|\h( \wf(\zeta))| \le \otau$ for all $\zeta \in V_{\CC}$ such that $\wf(\zeta)\ne 0$.  Let $\widetilde Q:=\Theta  \Theta^t \in S^{D}(\RR) $ with $\Theta$ as in \Cref{eq:thetad} be such that $$\wf\equiv \ B \,\widetilde Q \,B^t\ \ \mod I_\RR.$$
Then
$$
\h(\widetilde Q)\le \Cnd d^{2n-1}\delta(d+\tau)+\C1  \otau \quad \mbox{and}\quad \h(\sigma_{\min}(\widetilde Q))\ge - \big( \Cn d^{n-1}\delta (d+\tau)+\C1\otau\big). 
$$ 
\end{lemma}
\begin{proof}
We have $$\widetilde Q= \Theta\,\Theta^t= UP\,\Delta \,P^t U^t$$
where, if $V_\CC=\{\xi_0,\dots,\xi_{k},\zeta_1,\overline \zeta_1,\dots,\zeta_\ell,\overline \zeta_\ell\}$
with $V_\RR=\{\xi_0,\dots,\xi_k\}$ and $Z=\{\zeta_1,\dots,\zeta_\ell\}$  a maximal set  of non-real  non-conjugate roots in $V_\CC$, 
$U\in \CC^{D\times D}$ is the coefficient  matrix of the interpolation polynomials $u_\zeta\in \langle B\rangle_\CC$, $\Delta$ is the diagonal matrix of the $\omega_\zeta$ defined in \Cref{eq:simpl}, for  $\zeta\in V_\CC$ ordered as stated above and 
$$
P =\left(
\begin{array}{ccccccccc}
1     &  \\
   & 1 & \\
 & &  \ddots\\
  & & & 1  \\
\begin{array}{c}
\\
\\
\end{array} & & & & &P_{\zeta_1}\\
& & & & & & \ddots \\
\begin{array}{c}
\\
\\
\end{array} 
& & & & & & & \ \ P_{\zeta_\ell}
\end{array}
\right)
$$
where for $\zeta\in Z$,
$$P_\zeta=\frac{1}{2\bm i}\left( \begin{array}{cc}
{\bm i  + \mu_{\zeta}}    &  1  \\
{\bm i  - \mu_{\zeta}}    & -1 
\end{array}\right).$$
By \Cref{cor:hu}, \Cref{lem:htcoord} and \Cref{lem:omegarho2}, we have $$\h(U)\le \Cnd d^{2n-1}\delta(d+\tau) \quad \mbox{and} \quad \h(P),\h(\Delta)\le \C1  \otau. $$
This implies, since $D\le d^n$,  that $$\h(\widetilde Q)\le \Cnd d^{2n-1}\delta(d+\tau) + \C1  \otau.$$
Moreover, since all matrices are invertible,
\begin{equation*}
\sigma_{\min}(\widetilde Q) \ge 
\sigma_{\min}(U)^2\, \sigma_{\min}(P)^2\sigma_{\min} (\Delta).
\end{equation*}
To compute $\sigma_{\min} (P)$, we verify that for each $\zeta\in Z$, $P_{\zeta}^* P_{\zeta} = 
\dfrac 1 2 \left( \begin{array}{cc}
1  + \mu_{\zeta}^2   & \mu_{\zeta} \\
\mu_{\zeta} & 1
\end{array}\right)$ so that its eigenvalues $0< \lambda_1\le  \lambda_2$ are such that 
$\lambda_1 + \lambda_2 = 1 + \dfrac{\mu_{\zeta}^2}{2}$ and $\lambda_1\, \lambda_2=\dfrac14$. Thus, 
$\frac12 > \lambda_1 \ge \frac 1 {4 + 2\mu_{\zeta}^2}$, which implies that  
  $\frac12>\sigma_{\min}(P)\ge \min\{1,  \sigma_{\min}(P_{\zeta_1}), \ldots, \sigma_{\min}(P_{\zeta_\ell}) \}
\ge \sqrt{2}\min \{ (2 +\mu_{\zeta}^2)^{-\frac 1 2}, \zeta \in Z\}$.\\
Applying \Cref{lem:omegarho2}, we deduce that 
\begin{equation*}
    0 > \h(\sigma_{\min}(P)) \ge - \C1 \otau.
\end{equation*}
We conclude  from 
\Cref{lem:VdM} and  \Cref{lem:omegarho2}  that 
\[
\h(\sigma_{\min}(\widetilde Q)) \ge - \big( \Cn d^{n-1}\delta (d+\tau)+\C1 \otau \big).\qedhere
\]
\end{proof}

\begin{proposition}\label{prop:Lf}
    Let $ \wf\in \ZZrng$ with $ d_p=\max\{d,\deg(\wf)\}$ and $\h( \wf)\le \wtau$. Define  the linear variety 
    $$
\cl L_{ \wf} = \{  Y\in S^D(\RR)\,:\, B Y B^t\equiv p \ \mod I_{\RR} 
\}.
$$ 
Let $  Q/\nu$, with $Q\in S^D(\ZZ)$ and $\nu\in \NN$, be a {\em positive definite matrix}  with $  \h(  Q),\h(\nu)\le \tau_Q$, which satisfies 
$$
\dist\big(\frac Q \nu , \cl L_{\wf}\big)<\dist\big(\frac Q \nu ,\Sigma\big),
$$ 
where    $\Sigma $ is the set of singular matrices in $\RR^{D\times D}$.
Then the orthogonal projection $ {Q_0}/{\nu_0}$ of $ Q/{\nu}$ on $\cl L_{\wf}$, with $Q_0 \in S^D(\ZZ)$ and $\nu_0\in \NN$,  which is a positive definite matrix,  satisfies that 
$$
  \h(Q_0),\h(\nu_0)\le \Cnd d^{n-1}(d^n\delta+d_p)(d+\tau) 
  + \tau_p + \tau_Q. 
$$  
\end{proposition}

\begin{proof}
Set $Y=(y_{i,j})\in S^D(\RR)$.   We have 
$$ Y\in \cl L_{\wf} \ \iff \ {\cl U}(\wf - B Y B^t)=0 \ \iff \
\sum_{i,j} {\cl U}(b_i b_j)\, y_{i,j} ={\cl U}(\wf),
$$
where $B=\{b_1, \ldots, b_D\}$ is the basis of $\QQrng/I$, and $\cl U$ is the map defined in \eqref{eq:defN2}.
Thus, $\cl L_{\wf}$ is defined by a linear system of equations of the form $\bm A \bm y= \bm b $ where $\bm y$ is the vector of $N=  D (D+1)/2$ unknown coefficients of $Y\in S^D(\RR)$, $\bm A\in \ZZ^{ M\times N}$ is given by the coefficients of ${\cl U}(b_i\,b_j)$ in the basis $\{1, t, \ldots, t^{D-1}\}$ for $b_i,b_j\in B$ and  $\bm b\in \ZZ^{ M}$ is given by the coefficients of $\cl U(\wf)$ (here we can assume that $\bm A$  is of maximal rank $M\le D$ by  deleting  its superfluous linear dependent rows). Since by \Cref{lem:height-N} and \Cref{cor:height-N'}  we have  
$$
\h(\bm A) \le  \Cnd d^{n-1} \delta (d+\tau) \quad\mbox{and}\quad 
\h(\bm b) \le \Cnd d^{n-1}d_p (d+\tau) +  \tau_p,
$$
we conclude by applying \Cref{lem:projection} that the projection $ Q_0/\nu_0$ of $Q/\nu$ on 
$\cl L_{p}$ satisfies 
\[
\h( Q_0),\h(\nu_0)\le \Cnd d^{n-1}(d^{n}\delta+ d_p)(d+\tau) +  \tau_p 
+ \tau_Q. \qedhere
\]
\end{proof}

In order to round real polynomials to obtain polynomials in $\QQrng$, we will need the following rounding remark.
\begin{remark}\label{rem:rounding}
   Let $\displaystyle \widetilde \xi=\sum_{k\le M-1}a_{k}2^{k}\in \RR_{>0}$ for some $M\in \NN$ with  $a_k\in \{0,1\},\ \forall \,k$, be the binary expansion of a positive real number with $\h(\xi)\le M$, and  let  $$\frac{\xi}{2^N}:= \sum_{-N\le k\le M-1}a_{k}2^{k} \ \in \QQ \quad \mbox{with} \quad  \xi:=\sum_{0\le k\le N+M-1}a_{k-N}2^k \ \in \ZZ$$ be the truncation of  $\widetilde\xi$ to  $N$ digits after the comma.
   Then $$2^N\widetilde \xi-\xi < 1 \quad \mbox{ and } \quad \h( \xi)\le N+M.$$
   In this case, we say that we round the real number  $\widetilde \xi$ to a rational number $ \xi /2^N$, with $\widehat \xi\in \NN$, at precision $2^{-N}$.
\end{remark}

All this preparation allows us to prove our main theorem for the case when $I\subset \QQrng$ is a radical zero-dimensional ideal and $\wf\in \ZZrng$ satisfies $\wf>0$ on $V_\RR(I)$. We recall that $I$ is generated by polynomials of degree $\le d$ and height $\le \tau_p$, and that $B$ is the monomial basis of $\QQrng/I$ and $\delta=\max\{d,\deg(B)\}$.

\begin{theorem} \label{thm:Thethm} Let  $p \in \ZZrng $  be  such that $p >0$ on $V_\RR(I)$. 
Let $d_p:=\max\{d,\deg(p)\} $, $\h(p)\le \wtau$  and 
$|\h(p(\zeta))| \le  \otau$ for all $\zeta \in V_{\CC}$ such that $ p(\zeta)\ne 0$.  Then there exists a positive definite matrix $Q_0\in S^D(\ZZ)$ and $\nu_0\in \NN$ such that 
\begin{equation*} 
p \equiv  \dfrac{1}{\nu_0}B Q_0 B^t \ \mod \ I 
\end{equation*}
where    
$$ 
\h(  Q_0) , \h(\nu_0)\le  \Cnd d^{n-1}(d^n\delta+ d_p)(d+\tau)+ \C1 \otau  + \wtau.
$$ 
\end{theorem}

\begin{proof}
Let  $\widetilde Q=\Theta\,\Theta^t\in S^{D}(\RR)$ be the positive definite matrix constructed in \Cref{cor:precround2}, which satisfies that 
$$
p \equiv B \widetilde Q B^t \  \mod I_{\RR},
$$
with 
$$ 
\h(\widetilde Q)\le \Cnd d^{2n-1}\delta(d+\tau)+\C1 \otau \quad \mbox{and} \quad \h(\sigma_{\min}(\widetilde Q))\ge - \big( \Cn d^{n-1}\delta (d+\tau)+\C1\otau\big).$$
By \Cref{rem:rounding} we now  round $\widetilde Q\in S^D(\RR)$ to $ 2^{-N} Q$ with $ Q\in  S^D(\ZZ)$ and $N\in \NN$ with  $\h(N)= \Cn d^{n-1}\delta (d+\tau)+\C1\otau$ so that $\dist(\widetilde Q,2^{-N} Q)< \sigma_{\min}(\widetilde Q)$
 and $$\h({Q})\le \Cnd d^{2n-1}\delta(d+\tau)+\C1 \otau.$$
By \Cref{prop:Lf}, the projection $ Q_0/\nu_0$ with $ Q_0\in S^D(\ZZ)$  and $\nu_0\in \NN$ of $Q/\nu$ on  the linear variety 
$$\cl L_{ p} = \{ Y\in \cl S^D(\RR)\,:\, f \equiv  B Y B^t \ \mod I_{\RR} 
\}
$$ is  a positive definite matrix which satisfies
\begin{align*}\h(Q_0), \h(\nu_0)&\le \Cnd d^{n-1}(d^n\delta+d_p)(d+\tau) + \wtau+ \Cnd d^{2n-1}\delta(d+\tau)+\C1 \otau\\ &\le \Cnd d^{n-1}(d^n\delta+d_p)(d+\tau) + \C1 \otau + \wtau,
\end{align*}
by adjusting the term $\Cnd$.
\end{proof}

\subsection{The case when $f>0$ on $S$ for $S\subset V$}

\begin{lemma}\label{lem:transf} Let $S$ be as in \eqref{eq:s}  and let $p\in \ZZrng$  with $\h(p)\le \tau_p$ be such that $p(\xi)>0$ for all $\xi\in S$ and $|\h(p(\zeta))|\le \eta_p$ for all $\zeta\in V_\CC$ such that $p(\zeta)\ne 0$. 
Set $d_p:=\max\{d,\deg(p)\}$.  Then 
there exist $\nu_1\in \NN$ and for all $\xi\in V_\RR\setminus S$, there exist  $\rho_{\xi}\in \NN$,  $\widehat u_{\xi}\in \langle B\rangle_\ZZ$ and $0\le i_\xi\le r$ such that 
$$
\widehat p= \nu_1 p -\sum_{\xi\in V_\RR\setminus S}\rho_{\xi}\widehat u_{\xi}^2  g_{i_\xi}  \ \in \,\ZZrng
$$
satisfies 
\begin{itemize}
    \item 
 $\widehat p(\zeta)>0$ for all $\zeta \in V_\RR$,
\item $h(\nu_1) \le \Cn d^{n-1}\delta(d+\tau)+\eta_p $ and 
$\h({\rho}_\xi) \le \Cn d^{n}(d+\tau)+\eta_p$, $\forall\, \xi\in V_\RR\setminus S$,
\item $\h( \widehat u_\xi)\le \Cnd d^{2n-1}\delta(d+\tau)+\eta_p,$ $\forall\, \xi\in V_\RR\setminus S$,
\item
$ |\h(\widehat p(\zeta))|\le \widehat\eta_p:=\Cn d^{n-1}\delta (d+\tau)+\eta_p$   \  for all  $\zeta \in V_\CC$ such that $\ p(\zeta)\ne 0$, 
\item $\deg(\widehat p) \le \widehat d_p:=\max\{d_p,d + 2\delta \}$ and $ \h(\widehat p)\le \widehat\tau_p:=\Cnd d^{2n-1}\delta(d+\tau)+2\eta_p+\tau_p$.
\end{itemize}
\end{lemma}
\begin{proof}
We first note that by \Cref{lem:htpzeta},  we have that for all $\zeta \in V_\CC$ such that  $g_{i}(\zeta)\ne 0$,
\begin{equation}\label{eq:bound} |\h(g_{i}(\zeta))|\le \Cn d^{n}(d+\tau).\end{equation}
We proceed as in the proof of \Cref{thm:epss:finite}, which applies 
 Lemma~\ref{lem:perturb}. There  we defined $\widetilde p$ as 
$$ 
\widetilde p:=p-\sum_{\xi\in V_\RR\setminus S} \rho_{\xi}    u_\xi^{2} \, g_{i_{\xi}}  \ \in \ \RR[\bm x]$$
where for $\xi \in V_\RR\setminus S$, ${i_\xi}$ is such that $g_{i_\xi}(\xi)<0$, $ u_{\xi}\in \vspan{ B}_{\RR}$ is the idempotent associated to $\xi$  and $\rho_\xi$ is chosen such that $p(\xi)- \rho_\xi g_{i_\xi}(\xi)> 0$, in the following way.

If $p(\xi)>0$ we can take $\rho_\xi=0$, and if $p(\xi)=0$ we can take $\rho_\xi=1$. 
Now if 
$p(\xi)<0$ we need to choose $\rho_\xi$ such that 
$\rho_\xi> \dfrac{p(\xi)}{g_{i_\xi}(\xi)}$. 
By \cref{eq:bound}, we can take ${\rho}_\xi\in \NN$ with $\h({\rho}_\xi) =\Cn d^{n}(d+\tau)+\eta_p$ for some adjusted $\Cn$ so that $\rho_\xi> 2\dfrac{p(\xi)}{g_{i_\xi}(\xi)}$ and therefore
$$
p(\xi)-\rho_\xi g_{i_\xi}(\xi)>-p(\zeta) \quad \mbox{and} \quad \h\big(p(\xi)-\rho_\xi g_{i_\xi}(\xi)\big)\ge -\eta_p.
$$
Therefore, for all  $\zeta \in V_\RR$, $\widetilde p(\zeta)>0$ since for $\zeta\in S$, $\widetilde p(\zeta)=p(\zeta)$  while for $\zeta\in V_\RR\setminus S$,
$$ \widetilde p(\zeta) = p(\zeta)-\rho_\zeta g_{i_\zeta}(\zeta).$$ Moreover, for all $\zeta\in V_\RR$ we have $
\h\big(\widetilde p(\zeta)\big)\ge -\eta_p, 
$ and 
$$
\h\big(\widetilde p(\zeta)\big)=\h\big(p(\zeta)- \rho_{\zeta} g_{i_{\zeta}}(\zeta)\big)\le  \log(2) + \max\{\h(p(\zeta)), \h(\rho_\zeta)+\h(g_{i_\zeta}(\zeta))\} \le \Cn d^{n} (d+\tau)+\eta_p.
$$
We deduce more generally that   for all $\zeta\in V_\CC$ such that $p(\zeta)\ne 0$ we have  
$$
|\h(\widetilde p(\zeta))|\le \Cn d^{n} (d+\tau)+\eta_p.
$$
Now, for all $\xi\in V_\RR\setminus S$, we  round the $\delta$ coefficients  $u_{\xi,\alpha}$ of  $u_\xi=\sum_\alpha u_{\xi,\alpha}\bm x^\alpha \in \langle B\rangle_\RR$ to $\widehat u_{\xi,\alpha}/2^{N}$ with $\widehat u_{\xi,\alpha}\in \ZZ$ up to precision $2^{-{N}}$, for some $N\in \NN$, so that $|2^N u_{\xi,\alpha}-\widehat u_{\xi,\alpha}|<1$ and $\h(\widehat u_{\xi,\alpha})\le N+\h(u_{\xi,\alpha})$ as in the \Cref{rem:rounding}, i.e. $\h(\widehat u_{\xi})\le N+\h(u_{\xi})$.

We observe that for all $\zeta\in V_\CC$, as in the proof of \cref{lem:VdM}, we have 
$$
\h((2^Nu_\xi -\widehat u_\xi)(\zeta))\le \h(\sum_\alpha |2^Nu_{\xi,\alpha}-\widehat u_{\xi,\alpha}|\,|\zeta|^\alpha))\le  \log (D) + \Cn d^{n-1}\delta(d+\tau)\le \Cn d^{n-1}\delta(d+\tau)
$$
Furthermore,
$$|(2^Nu_\xi +\widehat u_\xi)(\zeta)|\le 2^{N+1}|u_\xi(\zeta)|+|(\widehat u_\xi-2^Nu_\xi)(\zeta)|\le 2^{N+1}+|(\widehat u_\xi-2^Nu_\xi)(\zeta)|$$ implies that 
\begin{align*}\h\Big(\big((2^Nu_\xi)^2-\widehat u_\xi^2\big)(\zeta)\Big)  &\le  \log\big(2^{N+1}|u_\xi(\zeta)|+|(\widehat u_\xi-2^N u_\xi)(\zeta)|\big)+ \h\big((2^Nu_\xi-\widehat u_\xi)(\zeta)\big) \\
& \le  \log\big(2^{N+1}+|(\widehat u_\xi-2^N u_\xi)(\zeta)|\big)+ \h\big((2^Nu_\xi-\widehat u_\xi)(\zeta)\big) \\ &\le N+2 + \Cn
d^{n-1}\delta(d+\tau)\end{align*} 
if we take  $N+1\ge \Cn d^{n-1}\delta(d+\tau)\ge \h((2^Nu_\xi -\widehat u_\xi)(\zeta))$.
We then define 
$$ \nu_1 := 2^{2N}, \quad \widehat p:=\nu_1 p-\sum_{\xi\in V_\RR-S} \rho_\xi \widehat u_\xi^2 g_{i_\xi}  \ \in \ZZrng.$$ We have   that for all $\zeta\in V_\RR$,

\begin{align*}
\widehat p(\zeta)&\ge \big(\nu_1 p-\sum_{\xi\in V_\RR-S} \rho_\xi (2^Nu_\xi)^2 g_{i_\xi}\Big)(\zeta)  - \Big|\Big(\sum_{\xi\in {V_\RR\setminus S}} \rho_\xi\big((2^Nu_\xi)^2-\widehat u_\xi^2 \big)g_{i_\xi}\Big)(\zeta)\Big|\\
&\ge 2^{2N}\widetilde p(\zeta)  - \sum_{\xi\in {V_\RR\setminus S}} \rho_\xi|\big((2^Nu_\xi)^2-\widehat u_\xi^2 \big)(\zeta)||g_{i_\xi}(\zeta)|\\
&\ge 2^{2N-\Cn d^{n}(d+\tau)-\eta_p}-  D 2^{N+2+3\Cn d^{n-1}\delta(d+\tau)}.
\end{align*}
Therefore, by taking $
\h(\nu_1) = 2N =\Cn d^{n-1}(d+\delta)(d+\tau)+\eta_p
$
for some adjusted $\Cn$, we conclude that for all $\zeta\in V_\RR$, $\widehat p(\zeta)>0$,
and
moreover, for all $\zeta \in V_\CC$ such that $p(\zeta)\ne 0$ we can also show  that
$$|\h(\widehat p(\zeta))|\le \Cn d^{n-1}\delta (d+\tau)+\eta_p $$
proceeding in the same way than before for the lower bound and using that
$$|\widehat p(\zeta)|\le 2^{2N}|\widetilde p(\zeta)|+ \sum_{\xi\in V_\RR\setminus S}\rho_\xi |\big((2^Nu_\xi)^2-\widehat u_\xi^2 \big)(\zeta)||g_{i_\xi}(\zeta)|$$ for the upper bound.

We finally observe that $\deg(\widehat p)\le \max\{d_p, d+2\delta\}$ and, by application of \Cref{cor:hu}, which implies that 
$$\h( \widehat u_\xi)\le \Cnd d^{2n-1}\delta(d+\tau)+\eta_p,$$
and therefore
$$\h(\widehat p)\le \Cnd d^{2n-1}\delta(d+\tau)+2\eta_p+\tau_p$$ (by adjusting $\Cnd$ again) since $\widehat u_\xi$ has at most $D\le d^n$ monomials. 
\end{proof}
\color{black}

\begin{theorem} \label{thm:f strict pos} Let $I=(\bm h)\subset \QQrng$ be a zero-dimensional radical ideal and let $B$ be a monomial basis of $\QQrng/I$ with $\delta:=\max\{d,\deg(B)\} $.  Let $S=S(\bm g,\bm h)$ be as in \Cref{eq:s} where $g_1,\dots,g_r,
h_1,\dots, h_s\in \ZZrng$ with  $\deg(g_i),\deg(h_j)\le d$ and $\h(g_i), \h(h_j)\le \tau$, $1\le i\le s$, $1\le j\le s$. 
Let  $p \in \ZZrng $  be  such that $p>0$ on $S$ with $d_p:=\max\{d,\deg(p)\} $, $\h(p)\le \tau_p$ and $|\h(p(\zeta))|\le \eta_p$ for all $\zeta \in V_{\CC}(I)$ such that $p(\zeta)\ne 0$. Then there exist $\nu_0,\nu_1\in \NN$,  a positive definite matrix $Q_0\in S^D(\ZZ)$, and $\omega_{i,k}\in\NN$ and  $q_{i,k}\in \langle B\rangle_\ZZ$ for $1\le i\le r$, $1\le k\le D$, such that 
\begin{equation*} 
p \equiv \frac 1 {\nu_0} B^t  Q_0 B + \frac 1 {\nu_{1}} \sum_{i=1}^r \Big( \sum_{k=1}^D \omega_{i,k}q_{i,k}^2 \Big) g_i   \ \  \mod \ I 
\end{equation*}
where   \begin{itemize} \item $\h(Q_0) , \h(\nu_0)\le   \Cnd d^{n-1}(d^n\delta+d_p)(d+\tau)+\C1\eta_p+\tau_p$;
    \item $ \h(\nu_1) \le  \Cn d^{n-1}\delta(d+\tau)+\eta_p $;
    \item $\h(\omega_{i,k})\le \Cn d^{n}(d+\tau)+\eta_p$;
    \item $ \h(q_{i,k})\le \Cnd d^{2n-1}\delta(d+\tau)+\eta_p$.
\end{itemize}
\end{theorem}

\begin{proof}
We apply \Cref{thm:Thethm}  to  $\widehat p/\nu_1$, where $\widehat p$ is the polynomial obtained in \Cref{lem:transf}.
The bounds for $\nu_1$, $\omega_{i,k}$ and $q_{i,k}$ are obtained directly from this lemma.  Since for $\widehat p$ we have
$$\widehat d_p=\max\{d_p,d + 2\delta \},\  \widehat \eta_p=\Cn d^{n-1}\delta(d+\tau)+\eta_p \ \mbox{ and } \widehat \tau_p=\Cnd d^{2n-1}\delta(d+\tau)+2\eta_p+\tau_p,$$ \Cref{thm:Thethm} implies that
\begin{align*}
 \h(Q_0) , \h(\nu_0)& \le  \Cnd d^{n-1}(d^n\delta+\widehat d_p)(d+\tau) + \C1 \widehat\eta_p+ \widehat\tau_p\\
 &\le \Cnd d^{n-1}(d^n\delta+d_p)(d+\tau)+\C1\eta_p+\tau_p,
 \end{align*}
 by adjusting again the constant $\Cnd$.
\end{proof}

As a consequence of the previous result, we obtain the following  complete SoS representation \Cref{eq:main_result} which includes  degrees and height bounds for a polynomial $f$ strictly positive on $S$ when $I$ is a radical zero-dimensional ideal. 
\begin{theorem} \label{thm:f2 strict pos} Let $I=(\bm h)\subset \QQrng$ be a zero-dimensional radical ideal and let $B$ be a monomial basis of $\QQrng/I$ with $D:=|B|$ and  $\delta:=\max\{d,\deg(B)\} $.  Let $S=S(\bm g,\bm h)$ be as in \Cref{eq:s} where $g_1,\dots,g_r,
h_1,\dots, h_s\in \ZZrng$ with  $\deg(g_i),\deg(h_j)\le d$ and $\h(g_i), \h(h_j)\le \tau$, $1\le i\le s$, $1\le j\le s$. 
Let  $f \in \ZZrng $  be  such that $f>0$ on $S$ with $d_f:=\max\{d,\deg(f)\}$ and $\h(f)\le \tau$. Then there exist $\nu_{0,k},\nu_1,\nu_2\in \NN$,  $\omega_{i,k}\in \NN$, $q_{i,k}\in \langle B\rangle_\ZZ$ and $p_j\in \ZZ[\bm x]$ for $1\le i\le r$, $1\le k\le D$ and $1\le j\le s$,  such that 
\begin{equation*} 
f= \sum_{k=1}^D \frac 1 {\nu_{0,k}} q_{0,k}^2 +  \frac 1 {\nu_1} \sum_{i=1}^r \Big( \sum_{k=1}^D \omega_{i,k}q_{i,k}^2 \Big) g_i + \frac 1 {\nu_2} \sum_{j=1} ^sp_j\, h_j
\end{equation*}
where  
\begin{itemize}
    \item $\h(\nu_{0,k}) \le    \Cnd d^{2n-1}(d^n\delta+d_f)(d+\tau) $ for $1\le k\le D$; 
    \item $\h(\nu_1) \le  \Cn d^{n-1}(\delta+d_f)(d+\tau) $;
    \item $\h(\omega_{i,k})\le \Cn d^{n-1}d_f(d+\tau)$ for $1\le i\le r, \ 1\le k\le D$;
    \item $\h(q_{i,k})\le \Cnd d^{n-1}(d^n\delta+d_f)(d+\tau)$ for $0\le i\le r, \ 1\le k\le D$.
    \end{itemize}
 Furthermore, if $\bf h$ is a graded basis of $I$, then for all $j$     
     $$\deg(p_j)< \widehat d-\deg(h_j)\quad \mbox{and}\quad  \h(\nu_2),\h(p_j)\le \Cnd d^{n-1}(d^n\delta+d_f)(d+\tau)+ \cst(n; \widehat{d}) {\widehat d}^n \tau,$$
for 
$\widehat d:= \max\{d_f, d+2\deg(B)\}+1$.
\end{theorem}

\begin{proof} When we apply \Cref{thm:f strict pos} to the polynomial $f$, taking into account that by \Cref{lem:htpzeta}, $$|\h(f(\zeta))|\le \Cn d^{n-1}d_f(d+\tau) $$
for all $\zeta\in V_\CC$ such that $f(\zeta)\ne 0$,
we obtain
\begin{equation*}
f\equiv \frac 1 {\nu_0} B Q_0 B^t +  \frac 1 {\nu_1} \sum_{i=1}^r \Big( \sum_{k=1}^D \omega_{i,k}q_{i,k}^2 \Big) g_i   \ \  \mod \ I 
\end{equation*}
where 
    $Q_0\in S^{D}(\ZZ), \nu_0\in \NN$ with   \begin{equation}\label{eq:Q0bound} \h(Q_0) , \h(\nu_0)\le   \Cnd d^{n-1}(d^n\delta+d_f)(d+\tau),\end{equation} 
   and  $\nu_1 \in \NN$, $\omega_{i,k}\in \NN$, $q_{i,k}\in \langle B\rangle_\ZZ$ with 
   $$
   \h(\nu_1) \le  \Cn d^{n-1}(d_f+\delta)(d+\tau) ,  \h(\omega_{i,k})\le \Cn d^{n-1}d_f(d+\tau),
    \h( q_{i,k})\le \Cnd d^{n-1}(d^n\delta+d_f)(d+\tau).
    $$
   Then, we apply \Cref{prop:cholesky} to the positive definite matrix $Q_0\in S^D(\ZZ)$ and obtain
    that $$B^tQ_0B=\sum \frac{1}{\widehat\nu_{0,k}} q_{0,k}^2$$
   where by \Cref{eq:Q0bound}  
   $$h(\widehat\nu_{0,k}), \h(q_{0,k})\le  2D\big(\log(D)+\tau_Q)\le \Cnd d^{2n-1}(d^n\delta+d_f)(d+\tau)$$
   for $1\le k\le D$.
   We define $\nu_{0,k}:=\nu_0\widehat\nu_{0,k}$, which satisfies the same height bound as $\widehat \nu_{0,k}$ to get $$\frac{1}{\nu_0}B Q_0 B^t= \sum_{k=1}^D \frac{1}{\nu_{0,k}}q_{0,k}^2.$$  
This shows the existence of an SoS  for $f$  with the stated height bounds for $\nu_{0,k}$, $\nu_1$, $\omega_{i,k}$ and $q_{i,k}$.

In the case when $\bm h$ is a graded basis of $I$, in order to obtain the height bounds for $\nu_2\in \ZZ$ and the polynomial coefficients $p_j\in \ZZ[\bm x]$ for $1\le j\le s$, we 
observe that the polynomial 
$$ \widehat f:=\nu_0\nu_1\,f-  \nu_1 B  Q_0 B^t -  \nu_0 \sum_{i=1}^r \Big( \sum_{k=1}^D \omega_{i,k}q_{i,k}^2 \Big) g_i \  \in I\cap \ZZ[\bm x]$$ 
satisfies 
$$
\deg(\widehat f)< \widehat d:= \max\{d_f,d+2\deg(B)\}+1\quad \mbox{and}\quad \h(\widehat f) \le \widehat\tau:=\Cnd d^{n-1}(d^n\delta+d_f)(d+\tau) .
$$
By \Cref{prop:reduction}, we have $\widehat f = \dfrac 1 {\widehat \nu} \displaystyle{\sum_{j=1}^s }p_j h_j$ with 
$$
\h(\widehat \nu) \le \cst(n; \widehat{d}) {\widehat d}^n \tau , 
\quad 
\h(p_j) \le  \Cnd d^{n-1}(d^n\delta+d_f)(d+\tau) + \cst(n; \widehat{d}) {\widehat d}^n \tau 
$$
Finally, we define $\nu_2:=\nu_0\nu_1\widehat \nu$ which also satisfies
\[
\h(\nu_2) \le \Cnd d^{n-1}(d^n\delta+d_f)(d+\tau)+ \cst(n; \widehat{d}) {\widehat d}^n \tau.\qedhere
\]
\end{proof}
\begin{remark}
    When $\deg(f), \deg(g_i), \deg(h_i) \le d$ we get \Cref{thm:C}, since $d+\tau \le d\,\tau$ for $d, \tau, \delta \ge 1$.
\end{remark}
\color{black}
\section{Height bounds for $I$ radical and $f$ nonnegative on $S$}
\label{sec:7}
We consider now the case where $f\in \ZZrng$ is nonnegative on $S=S(\bm g, \bm h)$ as in \Cref{eq:s}, for the case when $I=(\bm h)$ is a radical zero-dimensional  ideal, and $\bm g, \bm h\subset \ZZrng$ with $\deg(g_i), \deg(h_j)\le d$ and $\h(g_i), \h(h_j)\le \tau$ for $1\le i\le r$, $1\le j\le s$.
Again, $B$ is a monomial basis of $\QQrng/I$ with $D:=|B|$ and $\delta:=\max\{d,\deg(B)\}$.

Since $I$ is radical, we have $(I:f) + (f)= (1)$. Hereafter, we use this identity to reduce to the cases treated in \Cref{sec:6}.

\begin{lemma}\label{lem:a b gamma} Let $I\subset \QQrng$ be a zero-dimensional and radical ideal. Let $f \in \ZZrng $  be  such that $f\ge 0$ on $S$ and set $d_f:=\max\{d,\deg(f)\}$ and $\h(f)\le \tau$. Then, there exists $a, b \in \vspan{B}_{\ZZ}$ and $\gamma\in \NN$ such that 
\begin{itemize}
\item $a\, f + b = \gamma\ \mod I$,
\item $a>0$ on $S$,
\item $\h(a), \h(b), \h(\gamma) \le \Cnd d^{2n-1} (\delta + d_f) (d+\tau),$
\item $|\h(a(\zeta))| \le \Cnd d^{2n-1} (\delta + d_f)  (d+\tau)$ for any $\zeta \in V_{\CC}(I)$ such that $a(\zeta)\ne 0$.
\end{itemize}
\end{lemma}
\begin{proof}
 
Let $B = \{ \, b_1, \dots , b_D\,\}$ be the basis of $\QQrng\big/ I$.  We first analyze  polynomials $\widetilde a\in \langle B\rangle _\QQ$ and $\widetilde b\in (I:f)\cap \langle B\rangle_\QQ$ such that $\widetilde a \, f + \widetilde b\equiv 1\ \mod I$ (we note that such $\widetilde a$ and $\widetilde b$ exist since $(I:f)+(f)=(1)$).

Let  $\widetilde b = \sum_i \lambda_i b_i$ for some $\lambda_i\in \QQ$ and $\widetilde a=\sum_i \mu_i b_i$ for some $\mu_i\in \QQ$, $1\le i\le D$. Since $I$ is radical, we can apply the results of \Cref{subs:52}: we have $$\widetilde b\in (I:f) \iff \widetilde b\, f\in I \iff \cl U(\widetilde b\,f)=0 \iff \sum_{1\le i\le D}\cl U (fb_i) \,\lambda_i =0,$$
where $\cl U$ is the normal form map defined in \eqref{eq:defN2}.
Analogously, 
$$\widetilde a f+\widetilde b \equiv 1 \ \mod I \iff  \sum_{1\le i\le D}\big( \cl U(fb_i) \,\mu_i + \cl U (b_i)\, \lambda_i\big)=\cl U (1)=1.$$
Therefore $(\lambda_1, \ldots, \lambda_D, \mu_1, \ldots, \mu_D) $ is a solution of the linear system
 $\bm A \bm y = \bm b$ where $\bm A\in \ZZ^{2D\times 2D}$ is the  matrix of coefficients of $(\cl U(fb_i),\cl U(b_i))\in \langle 1,\dots, t^{D-1}\rangle\times  \langle 1,\dots, t^{D-1}\rangle$ and $\bm b= (0,\dots, 0;1,\dots,0)$.
 
 By \Cref{lem:height-N}, $\h(\bm A)\le \Cnd d^{n-1} (\delta+d_f) (d+\tau)$, and we deduce from 
 \Cref{rem:HadamardD}, that there is a solution $\bm {c}/\gamma$ with $\bm c\in \ZZ^{2D}$ and $\gamma \in \NN$ with 
$$
\h(\gamma), \h(\bm c) \le \Cnd  d^{2n-1} (\delta+d_f) (d+\tau)
$$
since $D\le d^n$. This implies that  $\sum_{i=1}^D \mu_i b_i=\widehat a/\gamma $ with $\widehat a\in \langle B\rangle_\ZZ$ and $ \sum_{i=1}^D \lambda_i b_i=   b/\gamma$ with $b\in \langle B\rangle_\ZZ$ and $\gamma\in \NN$, which  satisfy \begin{equation}\label{eq:tilde a b}
 \widehat a \, f + b = \gamma \ \mod I \quad \textup{ and } \quad  \h(\gamma),  \h(\widehat a), \h( b), \le \Cnd d^{2n-1} (\delta+d_f) (d+\tau).
\end{equation}
Moreover, by \Cref{lem:htpzeta},  for $\zeta \in V_{\CC}(I)$ we have that 
$|\h(b_i(\zeta))| \le \Cnd d^{n-1} \delta (d+ \tau)$ which implies that 
\begin{align}
    \h(\widehat a(\zeta)),\h( b(\zeta))& \le \log(D)+ \Cnd d^{2n-1} (\delta+d_f)(d+ \tau) +\Cnd d^{n-1} \delta  (d+\tau)   \nonumber \\
    & \le \Cnd d^{2n-1} (\delta+d_f) (d+\tau). \label{eq:h a zeta}
\end{align}
We now modify $\widehat a$ to ensure that it will be strictly positive on all $\xi\in S$ knowing that $f(\xi)\ge 0$ for all $\xi\in S$. Note that for all $\zeta\in V_\CC$, we have $\widehat a(\zeta) f(\zeta) +  b(\zeta) = \gamma$. 
\begin{itemize}
    \item  Let $\xi\in V_\RR$. If $f(\xi)=0$ then $ b(\xi)= \gamma\ge 1$ and for $\rho \ge |\widehat a(\xi)|+1$ we have $$\widehat a(\xi) + \rho \,  b(\xi) \ge\widehat a(\xi) + |\widehat a(\xi)|\gamma + \gamma  \ge \gamma\ge 1,$$
   
   and if $f(\xi)>0$ then $ b(\xi)=0$ since $ b\, f\in I$. Thus, $\widehat a(\xi)= \dfrac{\gamma} {f(\xi)} >0 $ and $\widehat a(\xi) + \rho\,  b(\xi) >0 $ for any  
    $\rho \in \RR$.
    \item More generally, if $\zeta\in V_\CC$ is such that $f(\zeta)=0$, then for $\rho\ge |\widehat a(\xi)|+1$, $|\widehat a(\zeta) + \rho \,  b(\zeta)|\ge 1$, and if 
    $f(\zeta)\ne 0$, then $|\widehat a(\zeta)+\rho\, b(\zeta)|=\dfrac{\gamma}{|f(\zeta)|}$.
\end{itemize}
Now, let $\rho\in \NN$ be such that $\rho \ge  \max\{ |\widehat a(\zeta)|+1:\, \zeta \in V_\CC \mbox{ s.t. } f(\zeta)=0\}$ and define
$a:= \widehat a + \rho \, b$. Then $a\in \langle B\rangle_\ZZ$ and $a f + b \equiv \gamma \ \mod I$ since $bf\in I$. Moreover, since 
by \Cref{eq:h a zeta}, we can take 
$$
\rho \le  \Cnd d^{2n-1} (\delta+d_f) (d+\tau),
$$
we deduce that 
\begin{align*}
    \h(a)=\h(\widehat a +\rho b)  \le \Cnd d^{2n-1} (\delta+d_f) (d+\tau). 
\end{align*}
Finally,  \Cref{eq:h a zeta} also implies that $$
\h(a(\zeta)) \le \Cnd d^{2n-1}(\delta+d_f) (d+\tau)
$$
and for any $\zeta\in V_\CC$ such that $a(\zeta)\ne 0$,
we have $\h(a(\zeta))\ge  -\Cnd d^{n-1}(\delta+d_f)(d+\tau)$ by \Cref{lem:htpzeta}.
\end{proof}

We can now deduce the bounds on the SoS representation of $f$.

\begin{theorem} \label{thm:Thethm2}
Let $I=(\bm h)\subset \QQrng$ be a zero-dimensional radical ideal and let $B$ be a monomial basis of $\QQrng/I$ with $D:=|D|$ and $\delta:=\max\{d,\deg(B)\} $.  Let $S=S(\bm g,\bm h)$ be as in \Cref{eq:s} where $g_1,\dots,g_r,
h_1,\dots, h_s\in \ZZrng$ with  $\deg(g_i),\deg(h_j)\le d$ and $\h(g_i), \h(h_j)\le \tau$, $1\le i\le s$, $1\le j\le s$. 
Let  $f \in \ZZrng $  be  such that $f\ge 0$ on $S$ with   $d_f:=\max\{d,\deg(f)\}$ and $\h(f)\le \tau$. Then there exist $\nu_{0,k},\nu_1,\nu_2\in \NN$,  $\omega_{i,k}\in \NN$, $q_{i,k}\in f\cdot \langle B\rangle_\ZZ$ and $p_j\in \ZZ[\bm x]$ for $1\le i\le r$, $1\le k\le D$ and $1\le j\le s$,  such that 
\begin{equation*} 
f= \sum_{k=1}^D \frac 1 {\nu_{0,k}} q_{0,k}^2 +  \frac 1 {\nu_1} \sum_{i=1}^r \Big( \sum_{k=1}^D \omega_{i,k}q_{i,k}^2 \Big) g_i + \frac 1 {\nu_2} \sum_{j=1} ^sp_j\, h_j
\end{equation*}
where  
\begin{itemize}
    \item  $\h(\nu_{0,k}), \h(q_{0,k})\le \Cnd d^{3n-1}(\delta + d_f) (d+\tau)$ for $1\le k\le D$;
\item $\h(\nu_1),\h(\omega_{i,k}), \h(q_{i,k})\le  \Cnd d^{2n-1}(\delta + d_f) (d+\tau)$  for   $1\le i\le r$, $1\le k\le D$.
\end{itemize}
Furthermore, if $\bm h$ is a graded basis of $I$, then for all $j$,
    $$\deg(p_j)< \widehat d-\deg(h_j)\quad \mbox{and}\quad  \h(\nu_2),\h(p_j)\le \Cnd d^{2n-1}(\delta+d_f)(d+\tau)+ \cst(n; \widehat{d}) {\widehat d}^n \tau$$
for 
$\widehat d:= 2(d_f+\deg(B))+1$.
\end{theorem}
\begin{proof}
We apply \Cref{thm:f strict pos} to $a>0$ on $S$ defined in \Cref{lem:a b gamma} such that $d_a = \max\{d, \delta\}$ and  $\tau_a=\eta_a= \Cnd d^{2n-1} (\delta+d_f) (d+\tau)$. We get
$$
a \equiv \frac 1 {\nu_0} B\, Q_0\, B^t + \frac 1 {\widehat\nu_1} \sum_{i=1}^r \big( \sum_{k=1}^{D} \omega_{i,k} \widehat q_{i,k}^2 \big) g_i \ \mod I
$$
where $ \nu_0,\widehat{} \nu_1 \in \NN$, $Q_0\in S^{D}(\ZZ)$ is a positive definite matrix,  and $\omega_{i,k}\in \NN$, $\widehat q_{i,k}\in \langle B\rangle_\ZZ$ satisfy
\begin{align*}
      \h(\nu_0), \h(Q_0)& \le   \Cnd d^{n-1}(d^n\delta+d_a)(d+\tau)+\C1\eta_a+\tau_a\\
      &\le \Cnd d^{2n-1}\delta(d+\tau)+\Cnd d^{2n-1} (\delta+d_f) (d+\tau)\\
     & \le \Cnd d^{2n-1}(\delta+d_f) (d+\tau).
\end{align*}
Analogously, 
     $$\h(\widehat\nu_1),\h(\omega_{i,k}),\h(\widehat q_{i,k}) \le   \Cnd d^{2n-1} (\delta+d_f) (d +\tau) \ \mbox{ for } \ 1\le i\le r, 1\le k\le D.$$
 Since from \Cref{lem:a b gamma} we have $af+b\equiv \gamma \ \mod I$ and  $b f\equiv 0 \ \mod I$ we deduce that 
\begin{align*}
f & \equiv \frac {1} {\gamma} a f^2 \ \mod I\\ & \equiv  \frac 1 {\gamma \nu_0} f^2 B Q_0 B^t +
\frac 1 {\gamma \widehat \nu_1}
\sum_{i=1}^r \big( \sum_{k=1}^{D} \omega_{i,k} (f \widehat q_{i,k})^2 \big)g_i
\ \mod I,
\end{align*}
where $\gamma \in \NN$ is s.t. $\h(\gamma) \le \Cnd d^{2n-1} (\delta+d_f)(d+\tau)$. 

Then, we apply \Cref{prop:cholesky bound} to the positive definite matrix $Q_0\in S^D(\ZZ)$ and obtain
    that $$B Q_0 B^t=\sum \frac{1}{\widehat\nu_{0,k}} \widehat q_{0,k}^2$$
   where for $1\le k\le D$, $\widehat\nu_{0,k}\in \NN$ and
   $\widehat q_{0,k}\in \langle B\rangle_\ZZ$ satisfy
   $$h(\widehat\nu_{0,k}), \h(\widehat q_{0,k})\le   \Cnd d^{3n-1}(\delta+d_f)(d+\tau).$$
Therefore, defining $\nu_{0,k}:=\gamma \nu_0\widehat \nu_{0,k}$, $\nu_1:=\gamma \widehat \nu_1$ and 
   $q_{i,k}:=f\,\widehat q_{i,k}$
   we obtain $\nu_{0,k}, \nu_1\in \NN$ and $q_{i,k}\in f\cdot \langle B\rangle_\ZZ$ which satisfy
   \begin{align*}
    f &\equiv   
\sum_{k=1}^{D} \frac 1 {\nu_{0,k}} q_{0,k}^2 
+ \frac 1 {\nu_1}
\sum_{i=1}^r \big( \sum_{k=1}^{D} \omega_{i,k} q_{i,k}^2 \big)
\  \mod I
\end{align*}
with
\begin{align*}& \h(\nu_{0,k}),\h (q_{0,k})\le \Cnd d^{3n-1}(\delta + d_f) (d+\tau) \quad \mbox{for} \quad 1\le k\le D,\\ & \h(\nu_1), \h(q_{i,k}) \le \Cnd d^{2n-1}(\delta + d_f) (d+\tau) \quad \mbox{for} \quad 1\le i\le r, 1\le k\le D. \end{align*}
This shows the first part of the statement.

When $\bm h$ is a graded basis of $I$, we conclude as  in the proof of \Cref{thm:f2 strict pos} by observing that the polynomial 
$$ \widehat f:=\gamma\nu_0\nu_1\,f-  \nu_1f^2 B Q_0 B^t -  \nu_0 \sum_{i=1}^r \Big( \sum_{k=1}^D \omega_{i,k}q_{i,k}^2 \Big) g_i \  \in I\cap \ZZ[\bm x]$$ 
satisfies 
$$
\deg(\widehat f)< \widehat d:= 2(d_f+\deg(B))+1\quad \mbox{and}\quad \h(\widehat f) \le \widehat\tau:=\Cnd d^{2n-1}(\delta + d_f) (d+\tau).
$$
By \Cref{prop:reduction}, we have $\widehat f = \dfrac 1 {\widehat \nu} \displaystyle\sum_{j=1}^s p_j h_j$ with 
$$
\h(\widehat \nu) \le \cst(n; \widehat{d}) {\widehat d}^n \tau , 
\quad 
\h(p_j) \le  \Cnd d^{2n-1}(\delta+d_f)(d+\tau) + \cst(n; \widehat{d}) {\widehat d}^n \tau 
$$
Finally, we obtain the required SoS representation   by  defining  $\nu_2:=\gamma\nu_0\nu_1\widehat \nu$, which also satisfies
\[
\h(\nu_2) \le \Cnd d^{2n-1}(\delta+d_f)(d+\tau)+ \cst(n; \widehat{d}) {\widehat d}^n \tau. \qedhere
\]
\end{proof}

\section{Algorithm and examples} \label{sec:8}
The approach developed in the manuscript leads naturally to an algorithm to compute the SoS decomposition of a strictly positive polynomial, using Semi-Definite Programming (SDP), which we draft below. 

For $\ell \in \NN$, let $\bm m_\ell$ be the row vector of all monomials of degree $\le \ell$. 
It is a well-known fact that $p \in \rng_{2\ell}$ is a sum of squares iff there exist  $Q\succcurlyeq 0$ (a  positive semi-definite symmetric matrix) such that $p = \bm m_\ell  \, Q\, \bm m_\ell^t$. Therefore, computing SoS representations of a certain degree boils down to computing positive semi-definite  matrices that represent the sums of squares. This task can be performed efficiently, using existing SDP solvers, provided the degree of the polynomials and the size of the SDP matrices is not too big, or equivalently that $\ell$ is small enough.

The SoS representation of a strictly positive polynomial $$f\equiv \sum_{k=1}^D \omega_{0,k}q_{0,k}^2 + \sum_{i=1}^r \left( \sum_{k=1}^D \omega_{i,k}q_{i,k}^2 \right) \, g_i \ \mod I$$ given in \Cref{thm:A} can be rewritten as 
\begin{equation*}\label{eq:SDP mod I}
f\equiv   B \widetilde Q_{0} B^t  + \sum_{i}  (B \widetilde Q_{i} B^t)g_i\ \mod I 
\end{equation*}
with $\widetilde Q_{i}\succcurlyeq 0$ and $\widetilde Q_0\succ 0$ (i.e. respectively $\widetilde Q_{i}$ positive semidefinite and $\widetilde Q_0$ positive definite) when $f>0$ on $S$, see also  \Cref{thm:epss:finite}. 

This can be reformulated into a classical SDP decomposition using monomials of bounded degree, as shown in the next lemma:
\begin{lemma} \label{lem:positive_definite}
If $f$ admits a representation of the form \eqref{eq:SDP mod I} with 
$\widetilde Q_{i}\succcurlyeq 0$ and $\widetilde Q_0\succ 0$ and $\bm h$ is a graded basis, then for $\ell_i \ge  \deg(B)$,
\begin{equation} \label{eq:decomp sdp}
  f =  \bm m_{\ell_0}  Q_{0} \bm m_{\ell_0}^t + \sum_{i} ( \bm m_{\ell_i} Q_{i} \bm m_{\ell_i}^t)g_i + \sum_{j=1}^{s} p_j\,h_j
\end{equation}
with $Q_0\succ 0$,  $Q_i \succcurlyeq 0$ for $i=1,..,r$, $\bm m_\ell$ is the vector of all monomials of degree $\le \ell$, $\ell_i \ge \deg(B)$ and $p_j \in \rng$ satisfies $\deg(p_j)\le  \max\{\deg(f), 2 \ell_0, \deg(g_i) +2 \ell_i\}-\deg(h_j)$. 
\end{lemma}
\begin{proof}
Since $B$ is a monomial basis of $\rng/I$ and $\bm h$ is a graded basis of $I$, we can decompose $\rng_\ell =B \oplus I_{\ell}$ with $I_{\ell}= I \cap \rng_{\ell} = \{\sum_{j=1}^{s} p_j\,h_j \mid \deg(p_j) \le \ell -\deg(h_j) , j=1, \ldots , s\}$.
Choosing a basis $C$ of $I_{\ell}$ so that $B\cup C$ is basis of $\rng_{\ell}$, we take $Q_i$ to be a block diagonal matrix in the basis $B \cup C$ with the upper block equal to $\widetilde Q_i$ and the identity for the lower block. Then we have $Q_0\succ 0$, $Q_i \succcurlyeq 0$ and $\bm m_{\ell_i} Q_{i} \bm m_{\ell_i}^t \equiv B \widetilde Q_{i} B ^t \mod I$. We deduce that 
$$
f - \left(  \bm m_{\ell_0}  Q_{0} \bm m_{\ell_0}^t + \sum_{i} ( \bm m_{\ell_i} Q_{i} \bm m_{\ell_i}^t)g_i \right) \in I_{\max\{\deg(f), 2 \ell_0, \deg(g_i) +2 \ell_i\}},
$$
which yields the decomposition \eqref{eq:decomp sdp}.
\end{proof}

To find such a decomposition, we solve 
a Semi-Definite Program that maximizes the smallest eigenvalue of $Q_0$ in \eqref{eq:decomp sdp}. \Cref{lem:positive_definite} implies that, when considering monomials of sufficiently high degree, this smallest eigenvalue will be strictly positive. Then we round, to a correct precision and taking into account the smallest eigenvalue of $Q^*_0$, the approximate optimal matrices $Q_i^*$ and polynomials $p_j^*$, to obtain an exact SoS representation for $f$.

\begin{algorithm}[H]
\label{algo:1}
\caption{Computing a rational decomposition of $f>0$ on $S$}
\textbf{Input:} $f, g_1, \ldots, g_r, h_1, \ldots, h_s \in \QQrng$ and $\ell_0, \ldots, \ell_r \in \NN$ such that $\ell_i\ge \delta$, where $\delta$ is the degree of a finite basis $B$ of $\QQrng/I$, and $(h_1, \ldots, h_s)$ is a graded basis of $I$.

\begin{enumerate}
\item Solve the Semi-Definite Program 
\begin{equation}\label{eq:algo sdp}
\begin{array}{rl}
\sup & \lambda\\
\textup{s.t.} & Q_{0}-\lambda\, \textup{I}\succcurlyeq 0\\
& f = \bm m_{\ell_0}  Q_{0} \bm m_{\ell_0}^t + \sum_{i=1}^{r}( \bm m_{\ell_i} Q_{i} \bm m_{\ell_i}^t) g_i+ \sum_{j=1}^{s} p_j\,h_j  \textup{ with }\\ 
& Q_{i}\succcurlyeq 0, p_j \in \rng_{ \max\{\deg(f), \deg(g_i) +2 \ell_i\}-\deg(h_k)}
\end{array}
\end{equation}
to get approximate optimal $Q^*_i, p^*_j$ and $\lambda^*>0$ lower bounding the optimal smallest eigenvalue of $Q^*_0$.
\item Repeat for 
$\kappa =\max\{ \lceil - \log(\lambda^*)\rceil, 0 \} \ldots $ 
\begin{itemize}
\item Compute a weighted Cholesky factorization of $Q_i^*$ and round it at precision $k$ to get $\widehat \omega_{i,k}\in \QQ_{\ge 0}$ and $\widehat q_{i,k}\in \vspan{B}_{\QQ}$ such that 
$$
B Q_i^* B^t \approx \sum_k \widehat \omega_{i,k} \widehat q_{i,k}^{\, 2},
$$
\item Round $p^*_j$ at precision $\kappa$ to get $\widehat p_j$,

\item Compute a weighted Cholesky decomposition of 
$$
\widehat f= f - \sum_{i} \big(\sum_{k} \widehat \omega_{i,k} \widehat q_{i,k}^{\, 2}\big)g_{i}  - \sum_{j} \widehat p_j\,h_j 
$$
as $\widehat f = \sum_{k} \widehat \omega_{0,k} (\widehat q_{0,k})^2$,
\end{itemize}
until $\widehat \omega_{0,k}>0$.
\end{enumerate}
\textbf{Output:} $f = \sum_{k} \widehat \omega_{0,k} \widehat q_{0,k}^{\,2} +  \sum_{i} \big(\sum_{k} \widehat \omega_{i,k} \widehat q_{i,k}^{\, 2}\big)g_{i}  + \sum_{j} \widehat p_j\,h_j$ with 
$\widehat \omega_{0,k} \in \QQ_{>0}$, $\widehat \omega_{i,k} \in \QQ_{\ge 0}$, \ $\widehat q_{i,k} \in \QQrng_{ \ell_i}$, 
$\widehat p_j \in \rng_{\max\{2 \ell, \deg(g_i) +2\ell_i\}-\deg(h_k)} $.
\end{algorithm}
\smallskip
The output of the algorithm is a certificate of positivity for $f$ on $S$. It may not be of smallest degree as in \Cref{thm:A}, and with small bitsize coefficients for the SoS representation, as in \Cref{thm:C}.  
To decrease the degrees in the SoS representation, one can simply reduce $\widehat q_{i,k}$ by the graded basis $\bm h$ to get $q_{i,k}\in \vspan{B}_{\QQ}$.
It should be noticed that the precision needed in the rounding in this algorithm is less than in \Cref{sec:6}, since we maximize the smallest eigenvalue of $Q_0$. In other words, the algorithm is self-adaptive, and we can expect good bitsize bounds on the computed SoS representation.

The complexity of Algorithm \ref{algo:1} is dominated by the complexity of solving the SDP program \eqref{eq:algo sdp} approximately. By \cite{Nesterov1994}, the number of arithmetic operations needed to find an approximate solution of 
\Cref{eq:algo sdp} is in \[\cl O\big(\sqrt{M}\, N^{2}(\sum_{i=0}^{r} {M_i}^{2} + 2\sum_{j=1}^{s} {M'_{j}}^{2})+\sqrt{M} N\, (\sum_{i=1}^{r} M_{i}^{3} + 2\,\sum_{j=1}^{s} {M'_{j}}^{3}) \big)\] where
\begin{itemize}
    \item $\widehat{d}=\max\{\deg(f), \deg(g_i), \deg(h_j)\}+ 2 \delta+1$ \item $M_i\le \widehat{d}^n$ (resp. $M'_j \le \widehat{d}^n$) (resp. $N\le \widehat{d}^n$) is the number of monomials of degree $<\widehat{d}-\deg(g_i)$ (resp. $<\widehat{d}-\deg(h_j)$ with $g_0=1$) (resp. $\widehat{d}$),
    \item $M = \sum_{i=0}^{r} M_0 + 2 \sum_{j=1}^{s} M'_{j}\le (r + 2s+ 1) \widehat{d}^n$.
\end{itemize}
The precision needed to perform this computation is in $\cl O(\kappa)$ where $\kappa =\max\{ \lceil - \log(\lambda^*)\rceil, 0 \}$. 
By \Cref{thm:f strict pos}, when $f>0$ on $S$, we have $\kappa = \widetilde{\cl O}(d^{n}) \delta \tau = \widetilde{\cl O}(\widehat{d}^{n+1})\tau$.
The weighted Cholesky decompositions involved in step 3 requires $\cl O(M_i^3)= \cl O(\widehat{d}^{3n})$ arithmetic operations. Thus the total bit complexity of Algorithm \ref{algo:1} when $I$ is radical and $f\ge 0$ on $S$ is in 
$$
\widetilde{\cl O}(\kappa) (r+2\,s+1) \widehat{d}^{4.5 n} = \widetilde{\cl O}(\widehat{d}^{5.5n+1}) (r+2\,s+1)  \tau.
$$

This algorithm is implemented in the \textsc{Julia} package \texttt{MomentPolynomialOpt.jl}\footnote{\url{https://github.com/AlgebraicGeometricModeling/MomentPolynomialOpt.jl}}. We present hereafter experimentation with this tool.

\begin{example} Set
$g_1:=y$, $h_1:=x^2-1$, $h_2:=
y^2-x-2$  $\in \QQ[x,y]$ and $S=S(g;h_1,h_2)\subset \RR^2$, the basic closed semialgebraic set defined by $g_1\ge 0$, $h_1=h_2=0$. 

We have $S=\{(1,\sqrt 3), (-1,1)\} $. Moreover,  $(h_1, h_2)$ is a graded basis of $I=(h_1,h_2)$ and $B=\{1, x, y, x\,y\}$ is a reduced basis of $\QQrng/I$, of degree $\delta=2$.

Let $f:= x+y+3\in \QQ[x,y]$, which is strictly positive on $S$. 
We apply Algorithm \ref{algo:1} in degree $4$ 
to compute a certificate of (strict) positivity of $f$ on $S$, using a rounding precision of $1$ (decimal) digit:
{\small
\begin{verbatim}
       WS, P, v, M = sos_decompose(f, G, H, X, 2; exact = true, round = 1)
\end{verbatim}
}%
\noindent{}with \texttt{G} $=[g_1]$, \texttt{H} $ = [ h_1, h_2]$, \texttt{X} $=[x,y]$ and a relaxation order $r=2$. We obtain the exact decomposition 
$$
f = q_0 + q_1 \, g_1 + p_1\,  h_1 + p_2\, h_2
$$
where {\small
\begin{align*}
q_0  = &\
\frac{1}{2} x^4 + \frac{3}{10} \, x^{2}y^2 - \frac{1}{10} \, xy^3 + \frac{2}{5} \, y^4 + \frac{1}{10} \, x^3 - \frac{3}{50} \, xy^2 - \frac{9}{125} \, y^3 \\ 
& + \frac{3}{10} \, x^2 + \frac{7}{25} \, y^2 - \frac{39}{500} \, y + \frac{7}{10},
\\
 = &\ 
\frac{1}{2} \, \left( x^2 + \frac{1}{5}\,{ y^2} + \frac{1}{10}\, {x} + \frac{1}{5} \right)^{2} 
+ \frac{1}{10}\, \left( xy - \frac{1}{2}\, {y^2} - \frac{3}{20}\, {y} \right)^{2} \\
& + \frac{71}{200}\, \left( y^2 - \frac{5}{71}\, {x} - \frac{87}{710}\, {y} + \frac{44}{213} \right)^{2} 
+ \frac{331}{3550}\, \left( x - \frac{87}{2648}\, {y} - \frac{103}{1986} \right)^{2} \\ 
& + \frac{6804227}{79440000}\, \left( y - \frac{2396940}{6804227} \right)^{2} + \frac{1001198282}{1530951075},\\
q_1  = &\ \frac{1}{5} \, {x^2} - \frac{6}{25 }\, { xy} + \frac{34}{125}\, {y^2} - \frac{17}{25}\, {y} + \frac{289}{500},
\\
= &\  \frac{1}{5} \,\left( x - \frac{3}{5}\, {y} \right)^{2} + \frac{1}{5}\, \left( y - \frac{17}{10} \right)^{2},\\
p_1  = &\ -\frac{1}{2}\, {x^2} - \frac{2}{5}\, {y^2} - \frac{1}{10}\, {y} - \frac{7}{10},
\\
p_2 = &\  \frac{1}{10}\, {x^2} + \frac{1}{10}\, {xy} - \frac{2}{5}\, {y^2} - \frac{1}{10}\, {x} - \frac{1}{5}\, {y} - \frac{4}{5}.
\end{align*}}
This example shows that SoS representations can exist in striclty smaller degree  than the upper bound $5$ of \Cref{thm:f strict pos}. 
\end{example}

\begin{example}\label{ex:singular} Set $h_1:=x^{3} - y^2$, $h_2:=x^{2} - 2x + y^{2}$ $\in \QQ[x,y]$ and 
$S = V_\RR(h_1, h_2)\subset \RR^2$.

We have $S=\{(0,0), (1,\pm 1), (2,\pm 2\sqrt 2)\}$ where 
$(0,0)$ has multiplicity $2$. Moreover, 
$(h_1, h_2)$ is a graded basis of $I=(h_1,h_2)$, and  $B=\{1,x,x^2, y, xy, xy^2\}$ is a reduced basis of $\QQ[x,y]/I$ of degree  $\delta=3$.

Let $f=x\in \QQ[x,y]$, which is nonnegative on $S$. We can verify that $(I:f)+(f)=(1)$ and that $a\, x + b =2$, for  $a= 1+x$ and $b=2 -x -x^2$, where $b\, x = -h_1 -h_2\in I$. 

Applying Algorithm 1 to $a=1+x$ in degree $4$, using a rounding precision of $1$ (decimal) digit, we get 
$$
a = q_0 + p_1 h_1+ p_2 h_2
$$
where
{\small
\begin{align*}
q_0 = &\left( 1 - \frac{3}{10} \, x - \frac{1}{15} \, y^{2} - \frac{1}{3} \, x^{2} \right)^{2} +
\frac{5}{6} \, \left( y - \frac{57}{100} \, x \, y \right)^{2} 
+ \frac{233}{300} \, \left( x - \frac{57}{466} \, y^{2} - \frac{60}{233} \, x^{2} \right)^{2} \\
&+\frac{16237}{33552} \, \left( y^{2} - \frac{7832}{81185} \, x^{2} \right)^{2} + \frac{2117}{4000} \, \left( x \, y \right)^{2} + \frac{432621}{811850} \, \left( x^{2} \right)^{2} \\
p_1= & -\frac{1}{10} - \frac{2}{5} \, x \\
p_2 = & - \frac{4}{5} - \frac{3}{10} \, x - \frac{1}{2} \, y^{2} - \frac{3}{10} \, x^{2} \\
\end{align*}
}
Since $2 x = a x^2+ b\,x= a\, x^2 -h_1-h_2$, we deduce {an SoS} representation of $f=x$:
$$
x = \frac{1}{2}  q_0 \, x^2 + \frac{1}{2} (p_1\,x^2 -1) h_1 + \frac{1}{2} (p_2 \, x^2-1) h_2.
$$
\end{example}

\subsection*{Acknowledgments}
The authors acknowledge the support of the Institut Henri Poincar\'e (UAR 839 CNRS-Sorbonne Universit\'e) and LabEx CARMIN (ANR-10-LABX-59-01).
%
They would like to thank Luis Felipe Vargas and Markus Schweighofer for bibliographical suggestions, Martin Sombra for discussions on the  arithmetic B\'ezout theorem, and Dávid Papp for discussions about rational SoS representations.

Lorenzo Baldi was partially funded by the Paris \^{I}le-de-France Region, under the grant agreement 2021-02--C21/1131, and by the Humboldt Research Fellowship for postdoctoral researchers.

{Finally, the authors want to thank the referee for the helpful comments and for suggesting the reference \cite{Hua2025}.}
\bibliographystyle{myalpha}
\bibliography{references}

\begin{thebibliography}{MSEDS19}

\bibitem[AH19]{atseriasSizeDegreeTradeOffsSumsofSquares2019}
A.~Atserias and T.~Hakoniemi.
\newblock {Size-Degree Trade-Offs for Sums-of-Squares and Positivstellensatz Proofs}.
\newblock In {\em 34th Computational Complexity Conference (CCC 2019)}, volume 137 of {\em Leibniz International Proceedings in Informatics (LIPIcs)}, pages 24:1--24:20. Schloss Dagstuhl -- Leibniz-Zentrum f{\"u}r Informatik, 2019.

\bibitem[Art27]{artinUberZerlegungDefiniter1927}
E.~Artin.
\newblock Uber Die {{Zerlegung}} Definiter {{Funktionen}} in {{Quadrate}}.
\newblock {\em Abhandlungen aus dem Mathematischen Seminar der Universitat Hamburg}, 5(1):100--115, 1927.

\bibitem[Bar68]{Bareiss1968}
E.~H. Bareiss.
\newblock Sylvester’s identity and multistep integer-preserving Gaussian elimination.
\newblock {\em Mathematics of Computation}, 22(103):565–578, 1968.

\bibitem[BM23]{Baldi2022}
L.~Baldi and B.~Mourrain.
\newblock {On the effective Putinar’s Positivstellensatz and moment approximation}.
\newblock {\em Mathematical Programming}, 200(1):71–103, 2023.

\bibitem[BMP25]{Baldi2025}
L.~Baldi, B.~Mourrain, and A.~Parusiński.
\newblock On Łojasiewicz inequalities and the effective Putinar’s Positivstellensatz.
\newblock {\em Journal of Algebra}, 662:741–767, 2025.

\bibitem[BS87]{BayerStillman1987}
D.~Bayer and M.~Stillman.
\newblock A criterion for detecting m-regularity.
\newblock {\em Inventiones Mathematicae}, 87(1):1–11, 1987.

\bibitem[BS24]{baldiDegreeBoundsPutinar2024}
L.~Baldi and L.~Slot.
\newblock Degree Bounds for Putinar’s Positivstellensatz on the Hypercube.
\newblock {\em SIAM Journal on Applied Algebra and Geometry}, 8(1):1--25, 2024.

\bibitem[Cha07]{Chardin2007}
M.~Chardin.
\newblock Some results and questions on {C}astelnuovo-{M}umford regularity.
\newblock In P.~Irena, editor, {\em Syzygies and {H}ilbert Functions}, volume 254 of {\em Lecture Notes in Pure and Applied Mathematics}, pages 1--40. {Springer}, 2007.

\bibitem[CLR95]{choiSumsSquaresReal1995}
M.~D. Choi, T.~Y. Lam, and B.~Reznick.
\newblock Sums of Squares of Real Polynomials.
\newblock In B.~Jacob and A.~Rosenberg, editors, {\em K-{{Theory}} and {{Algebraic Geometry}}: {{Connections}} with {{Quadratic Forms}} and {{Division Algebras}}, {{Part}} 2}, Proceedings of {{Symposia}} in {{Pure Mathematics}}, pages 103--126. {American Mathematical Society}, 1995.

\bibitem[DNP07]{demmelRepresentationsPositivePolynomials2007}
J.~Demmel, J.~Nie, and V.~Powers.
\newblock Representations of Positive Polynomials on Noncompact Semialgebraic Sets via {{KKT}} Ideals.
\newblock {\em Journal of Pure and Applied Algebra}, 209(1):189--200, 2007.

\bibitem[DP22]{Davis2022}
M.~M. Davis and D.~Papp.
\newblock Dual Certificates and Efficient Rational Sum-of-Squares Decompositions for Polynomial Optimization over Compact Sets.
\newblock {\em SIAM Journal on Optimization}, 32(4):2461–2492, 2022.

\bibitem[DP24]{Davis2024}
M.~M. Davis and D.~Papp.
\newblock Rational dual certificates for weighted sums-of-squares polynomials with boundable bit size.
\newblock {\em Journal of Symbolic Computation}, 121:102254, 2024.

\bibitem[EMT20]{Emiris2020}
I.~Emiris, B.~Mourrain, and E.~Tsigaridas.
\newblock Separation bounds for polynomial systems.
\newblock {\em Journal of Symbolic Computation}, 101:128–151, 2020.

\bibitem[FSP15]{Fawzi2015}
H.~Fawzi, J.~Saunderson, and P.~A. Parrilo.
\newblock Sparse sum-of-squares certificates on finite abelian groups.
\newblock In {\em 2015 54th IEEE Conference on Decision and Control (CDC)}, page 5909–5914. IEEE, December 2015.

\bibitem[GPT10]{gouveiaThetaBodiesPolynomial2010}
J.~Gouveia, P.~A. Parrilo, and R.~R. Thomas.
\newblock Theta {{Bodies}} for {{Polynomial Ideals}}.
\newblock {\em SIAM Journal on Optimization}, 20(4):2097--2118, 2010.

\bibitem[Gri01]{grigorievLinearLowerBound2001}
D.~Grigoriev.
\newblock Linear lower bound on degrees of Positivstellensatz calculus proofs for the parity.
\newblock {\em Theoretical Computer Science}, 259(1):613--622, 2001.

\bibitem[Hak21]{hakoniemiMonomialsizeVsBitcomplexity2021}
T.~Hakoniemi.
\newblock Monomial size vs. bit-complexity in sums-of-squares and polynomial calculus.
\newblock In {\em Proceedings of the 36th Annual ACM/IEEE Symposium on Logic in Computer Science}. Association for Computing Machinery, 2021.

\bibitem[Har07]{Harrison2007}
J.~Harrison.
\newblock Verifying Nonlinear Real Formulas Via Sums of Squares.
\newblock In K.~Schneider and J.~Brandt, editors, {\em Theorem Proving in Higher Order Logics}, pages 102--118, Berlin, Heidelberg, 2007. Springer Berlin Heidelberg.

\bibitem[HQ25]{Hua2025}
Z.~Hua and Z.~Qu.
\newblock Exactness and Effective Degree Bound of Lasserre’s Relaxation for Polynomial Optimization over Finite Variety.
\newblock {\em Mathematics of Operations Research}, July 2025.

\bibitem[Hua25]{Huang2025}
L.~Huang.
\newblock On the Complexity of Matrix Putinar’s Positivstellens\"{a}tz.
\newblock {\em SIAM Journal on Optimization}, 35(1):567–591, March 2025.

\bibitem[KLM15]{kurpiszHardestProblemFormulations2015}
A.~Kurpisz, S.~Lepp{\"a}nen, and M.~Mastrolilli.
\newblock On the {{Hardest Problem Formulations}} for the {0/1} {{Lasserre Hierarchy}}.
\newblock In M.~M. Halld{\'o}rsson, K.~Iwama, N.~Kobayashi, and B.~Speckmann, editors, {\em Automata, {{Languages}}, and {{Programming}}}, Lecture {{Notes}} in {{Computer Science}}, pages 872--885, {Berlin, Heidelberg}, 2015. {Springer}.

\bibitem[KMS23]{KMS2023}
T.~Krick, B.~Mourrain, and A.~Szanto.
\newblock Univariate Rational Sums of Squares.
\newblock {\em Revista de la Uni\'on Matem\'atica Argentina}, 64(2):215--237, 2023.

\bibitem[KPS01]{KPS01}
T.~Krick, L.~M. Pardo, and M.~Sombra.
\newblock Sharp estimates for the arithmetic Nullstellensatz.
\newblock {\em Duke Mathematical Journal}, 109(3), 2001.

\bibitem[Kri64]{krivineAnneauxPreordonnes1964}
J.~L. Krivine.
\newblock Anneaux Pr\'eordonn\'es.
\newblock {\em Journal d'Analyse Math\'ematique}, 12(1):307--326, 1964.

\bibitem[KS22]{knebuschRealAlgebraFirst2022}
M.~Knebusch and C.~Scheiderer.
\newblock {\em Real {{Algebra}}: {{A First Course}}}.
\newblock Universitext. Springer International Publishing, Cham, 2022.

\bibitem[Lan06]{landauUberDarstellungDefiniter1906}
E.~Landau.
\newblock {Uber die Darstellung definiter Funktionen durch Quadrate}.
\newblock {\em Mathematische Annalen}, 62(2):272--285, 1906.

\bibitem[Las01]{lasserreGlobalOptimizationPolynomials2001a}
J.~B. Lasserre.
\newblock Global {{Optimization}} with {{Polynomials}} and the {{Problem}} of {{Moments}}.
\newblock {\em SIAM Journal on Optimization}, 11(3):796--817, 2001.

\bibitem[Las02]{lasserrePolynomialsNonnegativeGrid2002}
J.~B. Lasserre.
\newblock Polynomials {{Nonnegative}} on a {{Grid}} and {{Discrete Optimization}}.
\newblock {\em Transactions of the American Mathematical Society}, 354(2):631--649, 2002.

\bibitem[Las09]{Lasserre2009}
J.~B. Lasserre.
\newblock {\em Moments, Positive Polynomials and Their Applications}.
\newblock IMPERIAL COLLEGE PRESS, 2009.

\bibitem[Lau03]{Laurent2003}
M.~Laurent.
\newblock {A Comparison of the {S}herali-{A}dams, {L}ovasz-{S}chrijver, and Lasserre Relaxations for 0–1 Programming}.
\newblock {\em Mathematics of Operations Research}, 28(3):470–496, 2003.

\bibitem[Lau06]{Laurent2006}
M.~Laurent.
\newblock Semidefinite representations for finite varieties.
\newblock {\em Mathematical Programming}, 109(1):1–26, 2006.

\bibitem[LPR20]{Lombardi2020}
H.~Lombardi, D.~Perrucci, and M.-F. Roy.
\newblock An Elementary Recursive Bound for Effective Positivstellensatz and Hilbert’s 17th problem.
\newblock {\em Memoirs of the American Mathematical Society}, 263(1277), 2020.

\bibitem[LQ15]{LombardiQuitte2015}
H.~Lombardi and C.~Quitté.
\newblock {\em Commutative Algebra: Constructive Methods: Finite Projective Modules}.
\newblock Springer Netherlands, 2015.

\bibitem[LS91]{Lovsz1991}
L.~Lovász and A.~Schrijver.
\newblock Cones of Matrices and Set-Functions and 0–1 Optimization.
\newblock {\em SIAM Journal on Optimization}, 1(2):166–190, 1991.

\bibitem[Mag14]{Magron2014}
V.~Magron.
\newblock {NLCertify: A Tool for Formal Nonlinear Optimization}.
\newblock In H.~Hong and C.~Yap, editors, {\em Mathematical Software -- ICMS 2014}, pages 315--320, Berlin, Heidelberg, 2014. Springer Berlin Heidelberg.

\bibitem[Mar08]{Marshall2008}
M.~Marshall.
\newblock {\em Positive Polynomials and Sums of Squares}.
\newblock {American Mathematical Society}, 2008.

\bibitem[MD21]{https://doi.org/10.48550/arxiv.1811.10062}
V.~Magron and M.~S.~E. Din.
\newblock {On Exact Reznick, Hilbert-Artin and Putinar's Representations}, arXiv preprint, 2021.

\bibitem[MDR17]{DBLP:conf/cpp/Martin-DorelR17}
{\'E}.~Martin-Dorel and P.~Roux.
\newblock A reflexive tactic for polynomial positivity using numerical solvers and floating-point computations.
\newblock In {\em Proceedings of the 6th ACM SIGPLAN Conference on Certified Programs and Proofs - CPP 2017}, CPP 2017, page 90–99. ACM Press, 2017.

\bibitem[MDV23]{Magron2023}
V.~Magron, M.~S.~E. Din, and T.-H. Vu.
\newblock Sum of Squares Decompositions of Polynomials over Their Gradient Ideals with Rational Coefficients.
\newblock {\em SIAM Journal on Optimization}, 33(1):63--88, 2023.

\bibitem[MS18]{Martinez2018}
C.~Martínez and M.~Sombra.
\newblock An arithmetic Bernštein–Kušnirenko inequality.
\newblock {\em Mathematische Zeitschrift}, 291(3–4):1211–1244, September 2018.

\bibitem[MSED21]{Magron2021}
V.~Magron and M.~Safey El~Din.
\newblock On exact {R}eznick, {H}ilbert-{A}rtin and {P}utinar’s representations.
\newblock {\em Journal of Symbolic Computation}, 107:221–250, 2021.

\bibitem[MSEDS19]{Magron2019}
V.~Magron, M.~Safey El~Din, and M.~Schweighofer.
\newblock Algorithms for weighted sum of squares decomposition of non-negative univariate polynomials.
\newblock {\em Journal of Symbolic Computation}, 93:200–220, 2019.

\bibitem[NDS06]{nieMinimizingPolynomialsSum2006}
J.~Nie, J.~Demmel, and B.~Sturmfels.
\newblock Minimizing {{Polynomials}} via Sum of {{Squares}} over the {{Gradient Ideal}}.
\newblock {\em Mathematical Programming}, 106(3):587--606, 2006.

\bibitem[NN94]{Nesterov1994}
Y.~Nesterov and A.~Nemirovskii.
\newblock {\em Interior-Point Polynomial Algorithms in Convex Programming}.
\newblock Society for Industrial and Applied Mathematics, January 1994.

\bibitem[NP23]{netzerGeometryLinearMatrix2023}
T.~Netzer and D.~Plaumann.
\newblock {\em Geometry of {{Linear Matrix Inequalities}}: {{A Course}} in {{Convexity}} and {{Real Algebraic Geometry}} with a {{View Towards Optimization}}}.
\newblock Compact {{Textbooks}} in {{Mathematics}}. {Springer International Publishing}, {Cham}, 2023.

\bibitem[NS07]{Nie2007}
J.~Nie and M.~Schweighofer.
\newblock On the complexity of {Putinar’s Positivstellensatz}.
\newblock {\em Journal of Complexity}, 23(1):135–150, 2007.

\bibitem[O'D17]{odonnellSOSNotObviously2016}
R.~O'Donnell.
\newblock {SOS Is Not Obviously Automatizable, Even Approximately}.
\newblock In {\em 8th Innovations in Theoretical Computer Science Conference (ITCS 2017)}, volume~67 of {\em Leibniz International Proceedings in Informatics (LIPIcs)}, pages 59:1--59:10. Schloss Dagstuhl -- Leibniz-Zentrum f{\"u}r Informatik, 2017.

\bibitem[Par00]{parriloStructuredSemidefinitePrograms2000}
P.~A. Parrilo.
\newblock {\em Structured Semidefinite Programs and Semialgebraic Geometry Methods in Robustness and Optimization}.
\newblock PhD thesis, California Institute of Technology, 2000.

\bibitem[Par02]{Parrilo2002}
P.~A. Parrilo.
\newblock An Explicit Construction of Distinguished Representations of Polynomials Nonnegative over Finite Sets.
\newblock Technical report, {IfA Technical Report AUT02-02}, 2002.

\bibitem[PD01]{Prestel2001}
A.~Prestel and C.~N. Delzell.
\newblock {\em Positive Polynomials}.
\newblock Springer Berlin Heidelberg, 2001.

\bibitem[Pou71]{pourchet}
Y.~Pourchet.
\newblock Sur la repr\'{e}sentation en somme de carr\'{e}s des polyn\^{o}mes \`a une ind\'{e}termin\'{e}e sur un corps de nombres alg\'{e}briques.
\newblock {\em Acta Arithmetica}, 19:89--104, 1971.

\bibitem[Pow11]{Powers2011}
V.~Powers.
\newblock Rational certificates of positivity on compact semialgebraic sets.
\newblock {\em Pacific Journal of Mathematics}, 251(2):385–391, 2011.

\bibitem[PP08]{peyrlComputingSumSquares2008}
H.~Peyrl and P.~A. Parrilo.
\newblock Computing Sum of Squares Decompositions with Rational Coefficients.
\newblock {\em Theoretical Computer Science}, 409(2):269--281, 2008.

\bibitem[Put93]{putinarPositivePolynomialsCompact1993}
M.~Putinar.
\newblock Positive {{Polynomials}} on {{Compact Semi-algebraic Sets}}.
\newblock {\em Indiana University Mathematics Journal}, 42(3):969--984, 1993.

\bibitem[RW17]{raghavendraBitComplexitySumofSquares2017}
P.~Raghavendra and B.~Weitz.
\newblock {On the Bit Complexity of Sum-of-Squares Proofs}.
\newblock In {\em 44th International Colloquium on Automata, Languages, and Programming (ICALP 2017)}, volume~80 of {\em Leibniz International Proceedings in Informatics (LIPIcs)}, pages 1--13. Schloss Dagstuhl -- Leibniz-Zentrum f{\"u}r Informatik, 2017.

\bibitem[SA90]{Sherali1990}
H.~D. Sherali and W.~P. Adams.
\newblock A Hierarchy of Relaxations between the Continuous and Convex Hull Representations for Zero-One Programming Problems.
\newblock {\em SIAM Journal on Discrete Mathematics}, 3(3):411–430, 1990.

\bibitem[Sch91]{schmudgenTheKmomentProblemCompact1991}
K.~Schm{\"u}dgen.
\newblock {{The K-moment}} Problem for Compact Semi-Algebraic Sets.
\newblock {\em Mathematische Annalen}, 289(1):203--206, 1991.

\bibitem[Sch02]{Schweighofer2002}
M.~Schweighofer.
\newblock An algorithmic approach to Schm\"{u}dgen’s Positivstellensatz.
\newblock {\em Journal of Pure and Applied Algebra}, 166(3):307–319, 2002.

\bibitem[Sch16]{scheidererSumsSquaresPolynomials2016}
C.~Scheiderer.
\newblock Sums of Squares of Polynomials with Rational Coefficients.
\newblock {\em Journal of the European Mathematical Society}, 18(7):1495--1513, 2016.

\bibitem[Sho87]{shorClassGlobalMinimum1987}
N.~Z. Shor.
\newblock Class of Global Minimum Bounds of Polynomial Functions.
\newblock {\em Cybernetics}, 23(6):731--734, 1987.

\bibitem[Ste74]{stengleNullstellensatzPositivstellensatzSemialgebraic1974}
G.~Stengle.
\newblock A Nullstellensatz and a Positivstellensatz in Semialgebraic Geometry.
\newblock {\em Mathematische Annalen}, 207(2):87--97, 1974.

\bibitem[STKI17]{Sakaue2017}
S.~Sakaue, A.~Takeda, S.~Kim, and N.~Ito.
\newblock Exact Semidefinite Programming Relaxations with Truncated Moment Matrix for Binary Polynomial Optimization Problems.
\newblock {\em SIAM Journal on Optimization}, 27(1):565–582, 2017.

\bibitem[Sza08]{Szanto2008}
A.~Szanto.
\newblock Solving over-determined systems by the subresultant method (with an appendix by Marc Chardin).
\newblock {\em Journal of Symbolic Computation}, 43(1):46–74, 2008.

\bibitem[vG13]{vonzurGathen2013}
J.~{von zur Gathen} and J.~Gerhard.
\newblock {\em Modern Computer Algebra}.
\newblock {Cambridge University Press}, 2013.

\end{thebibliography}

\appendix
\section{Appendix: Linear algebra bounds}
In this appendix, we recall height bounds for the solutions of basic linear algebra operations on matrices with integer coefficients.

\subsection{Solving linear systems}

An important operation in our approach is to compute a solution of a linear system of equations.  We can bound the size of a rational solution of a linear system with integer coefficients as follows:

\begin{lemma}\label{rem:HadamardD}
Let $\bm A\in \ZZ^{M \times N}$ be such that $M\le N$ and $\bm b\in \ZZ^M$ with $\h(\bm A)\le \tau_A$ and $\h(\bm b)\le \tau_b$. If the linear system $\bm A\, \bm y=\bm b$ admits a solution, then one of these solutions can be written by Cramer's rule as $\bm c=  \widehat {\bm c}/\nu$ with $\widehat {\bm c}\in \ZZ^N$ and $\nu\in \NN$ satisfying by Hadamard's inequality 
$$
\h(\nu)\le  \frac{M}{2}\log(M) + M\tau_A  \quad \mbox{and} 
\quad \h(\widehat{\bm c})\le  \frac{M}{2}\log(M)+(M-1)\tau_A +\tau_b.
$$
\end{lemma}

\subsection{Projection on a linear space}

\begin{lemma}\label{lem:projection} 
    Let $M\le N$,   $\bm A\in \ZZ^{M\times N}$ of rank $M$ and $\bm b\in \ZZ^{M}$, with $1\le \h(\bm A)\le \tau_A, 1\le \h(\bm b)\le \tau_b$. Let  $\cl L$ be the linear variety
    $$\cl L=\{ \bm y\in \RR^N\,: \, \bm A \bm y=\bm b\}.$$ 
    Let $\bm c_0=\widehat{\bm c}_0/\nu_0\in \QQ^N$ with $\widehat{\bm c}_0\in \ZZ^N$, $\nu_0\in \NN$, $\h(\widehat{\bm c}_0),\h(\nu_0)\le \tau_0$. Then, the orthogonal projection $\bm c=\widehat{\bm c}/\nu\in \QQ^N$ with $\widehat{\bm c}\in \ZZ^N$, $\nu\in \NN$,  of $\bm c_0$ on $\cl L$ satisfies
    $$
     \h(\widehat{\bm c}),\h(\nu)\le (2M+1)(\log(N)+ \tau_A)+ \tau_b + \tau_0.
    $$
\end{lemma}
\begin{proof}
Since the columns of $\bm A^t$ span $(\ker \bm A)^{\perp}$, the orthogonal projection $\bm c$ of $\bm c_0$ on $\cl L$ is $\bm c= \bm c_0-\bm A^t \bm z$ for $\bm z\in \QQ^{M}$ such that
$$
\bm A \bm A^t \bm z = \bm A \bm c_0 - \bm b=\frac{\bm A \widehat{\bm c}_0 -\nu_0 \bm b}{\nu_0}.
$$
Let $\bm z= {\widetilde{\bm z}}/ \nu_0$ with $\widetilde{\bm z}\in \QQ^M$, so that 
$$
\bm A \bm A^t \widetilde{\bm z} = \bm A \widehat{\bm c}_0 - \nu_0\bm b.
$$
As   $\h(\bm A \bm A^T)\le   \log(N)+ 2 \tau_A$ and
$\h(\bm A \widehat{\bm c}_0 -\nu_0\bm b)\le  \log(N)+  \tau_A+ \tau_b+ \tau_0$, we deduce
by  \Cref{rem:HadamardD} that $\widetilde{\bm z}= {\widehat{\bm z}/{\nu_1} } $ with $\nu_1 \in \NN, \widehat {\bm z} \in \ZZ^{N}$ such that
$$
\h(\nu_1)\le  2M(\log(N)+ \tau_A ) \quad \mbox{and} \quad 
\quad
\h(\widehat{\bm z})\le 2M\log(N)+ (2M-1)\tau_A+ \tau_b + \tau_0.
$$
Thus, $\bm c = \bm c_0-\bm A^t \bm z = \dfrac{  \nu_1  \widehat{\bm  c}_0  - \bm A^t \widehat{\bm z}}{\nu_0\nu_1}= \dfrac{\widehat{\bm c}}{\nu}$ satisfies
\[
\h(\nu),\h(\widehat{\bm c})\le (2M+1)\log(N)+2M\tau_A+ \tau_b+ \tau_0.\qquad \qedhere
\]
\end{proof}

\subsection{Weighted square-root-free Cholesky decomposition}
Another important ingredient is the decomposition of a definite positive quadratic form as a weighted sum of squares, also known as a weighted Cholevsky factorisation.
The decomposition proceeds by induction, removing one variable at a time, using a completion-of-square technique.

Let $Q(X_1, \ldots, X_M)= \sum_{1 \le i,j\le M} q_{i,j} X_i\, X_j$ be a form with $q_{i,j}= q_{j,i}\in \RR$.
We denote by $\Delta^{k}_{i,j}$  the $(k+1)\times (k+1)$ minor formed by the rows $1, \ldots, k, i$ and columns $1, \ldots, k, j$ of $\bm Q = (q_{i,j})$. 

Notice that $\Delta^0_{i,j}= q_i,j$ and $\Delta^{k}_{i,j}=0$ if $i\le k$ or $j\le k$. 
Let $\Delta_k= \Delta^{k-1}_{k,k}$ be the $k^{\mathrm th}$ first principal minor.
We assume hereafter that $\Delta_{k}\ne 0$ for $k=1, \ldots, M$.

Completing the square with respect to the variable $X_1$, we obtain 
$$
Q(X_1,\ldots, X_M) = \frac 1 {q_{1,1}} (\sum_{i} q_{i,1} X_i)^2 + Q'(X_2,\ldots, X_n),
$$
where $Q' = (q'_{i,j}) $ with
$$
q'_{i,j} = q_{i,j} - \frac 1 {q_{1,1}}  q_{i,1} q_{1,j} 
= \frac 1 {q_{1,1}} ( q_{1,1} q_{i,j} - q_{i,1} \, q_{1,j}).
$$
Then $q'_{i,j} = \frac 1 {\Delta_{1}} \Delta^1_{i,j}$.
By the same computation applied to $Q'$, we get
\begin{align}
\label{eq:sos 2}
Q(X_1,\ldots, X_M) 
 &= \frac 1 {\Delta_{1}} (\sum_{i} \Delta^0_{i,1} X_i)^2 
+ \frac 1 {q'_{2,2}} (\sum_{i=2}^M q'_{i,2} X_i)^2 
+ Q''(X_3, \ldots, X_M)\nonumber\\
& = \frac 1 {\Delta_{1}} (\sum_{i} \Delta^0_{i,1} X_i)^2 
+ \frac 1 {\Delta_{2} \Delta_1} (\sum_{i} \Delta^1_{i,2} X_i)^2 
+ Q''(X_3, \ldots, X_M)
\end{align}
where $Q'' = (q''_{i,j}) $ with
$$
q''_{i,j} = \frac 1 {q'_{2,2}} ( q'_{2,2} q'_{i,j} - q'_{i,1} \, q'_{1,j})
= \frac {1} {\Delta_{2} \Delta_{1}}( \Delta^{1}_{2,2} \Delta^{1}_{i,j} - \Delta^{1}_{i,1} \, \Delta^{1}_{1,j}).
$$
Using Sylvester identity (see e.g. \cite{Bareiss1968}), we have $ \Delta^{1}_{2,2} \Delta^{1}_{i,j} - \Delta^{1}_{i,1} \, \Delta^{1}_{1,j} = \Delta_{1}\, \Delta^2_{i,j}$ and thus
$q''_{i,j} = \frac {1} {\Delta_{2}} \Delta^2_{i,j}$.

Repeating this computation, we obtain an explicit decomposition of $Q$ in terms of squares:
\begin{proposition}\label{prop:wsos}
Let $Q(X_1, \ldots, X_M)= \sum_{1 \le i,j\le M} q_{i,j} X_i\, X_j$ be a quadratic form such that $\Delta_{k}\ne 0$ for $k=1,\ldots, M$ and let $\Delta_0=1$. Then 
$$
Q = \sum_{k=1}^{M} \frac 1 {\Delta_{k} \Delta_{k-1}} (\sum_{i} \Delta^{k-1}_{i,k} X_i)^2.
$$
\end{proposition}
\begin{proof}
We prove by induction that 
\begin{equation}\label{eq:sos k}
Q(X_1,\ldots, X_M) 
= \frac 1 {\Delta_{1}} (\sum_{i} \Delta^0_{i,1} X_i)^2 
+ \,s  
+ \frac 1 {\Delta_{k-1} \Delta_{k-2}} (\sum_{i} \Delta^{k-2}_{i,k-1} X_i)^2 
+ Q^{(k-1)}(X_{k}, \ldots, X_M)
\end{equation}
where $Q^{(k-1)}(X_{k}, \ldots, X_{M}) = (q_{i,j}^{(k-1)})$ with 
$q_{i,j}^{(k-1)} = \frac 1 {\Delta_{k-1}} \Delta^{(k-1)}_{i,j}$. By the same construction as above applied to $Q^{(k-1)}$ we obtain a decomposition at order $k$ from the decomposition \eqref{eq:sos k} or order $k-1$, 
with $Q^{(k)}= (q^{(k)}_{i,j})$ such that
$$
q^{(k)}_{i,j} = \frac 1 {q^{(k-1)}_{k,k}} ( q^{(k-1)}_{k,k} q^{(k-1)}_{i,j} - q^{(k-1)}_{i,k} \, q^{(k-1)}_{k,j})
= \frac {1} {\Delta_{k}} \Delta^{k}_{i,j}.
$$
Here we use again Sylvester identity: $\Delta_{k} \Delta^{k-1}_{i,j} - \Delta^{k-1}_{i,k} \, \Delta^{k-1}_{k,j} =  \Delta_{k-1}\, \Delta^{k}_{i,j}$. This proves the induction and the proposition since we have \eqref{eq:sos 2}.
\end{proof}
\begin{remark}
This shows in particular that if $Q$ is a quadratic form such as $\Delta_k \ne 0$, then its signature is $(M-s,s)$ where $s$ is the number of sign changes of the sequence $[\Delta_0=1, \Delta_1, \ldots, \Delta_M]$. This is a well-known result, probably due to Sylvester.  
\end{remark}
We deduce the following proposition:
\begin{proposition}\label{prop:cholesky} Let $K$ be a number field and $Q\in S^M(K)$ such that $\Delta_k \neq 0$ for $k=1, \ldots, M$ and $\Delta_0=1$. Then 
$$
\bm Q = \bm L\, \bm D\, \bm L^t
$$
where 
\begin{itemize}
    \item $\bm L=(\bm L_{i,j})_{1 \le i,j\le M}$ is lower triangular with $\bm L_{i,j} =  \Delta_{i,j}^{j-1}$.
    \item $\bm D$ is diagonal with $\bm D_{i,i}= \frac 1 {\Delta_{i}\Delta_{i-1}}$.
\end{itemize}
\end{proposition}
\begin{proposition}\label{prop:cholesky bound}
Let $\bm Q\in S^M(\ZZ)$ be  a positive definite matrix  with $\h(\bm Q)\le \tau_Q$. Then 
$$
\bm Q = \bm L\, \bm D \, \bm L^t
$$
where $\bm D= {\diag}(\frac 1 {\nu_1}, \ldots, \frac 1 {\nu_M})$, $\nu_i \in \NN$, $\bm L\in S^M(\ZZ)$ is lower triangular with 
$$
\h(\nu_i), \h(\bm L) \le 2 M \, (\log(M) + \tau_Q)
$$
\end{proposition}
\begin{proof}
If $\bm Q\succ 0$, then its principal minors $\Delta_{j}>0$ are strictly positive.
By \Cref{prop:cholesky}, $\bm Q = \bm L\, \bm D \bm L^t$ with 
$L_{i,j} = \Delta^{j-1}_{i,j}\in \ZZ$ and $L_{i,j}=0$ if $i<j$, and $\nu_{j}= {\Delta_{j}\Delta_{j-1}} \in \QQ_{+}$. 
By Hadamard identity on $j\times j$ minors, we have $\h(L_{i,j}) \le j\, (\log(j)+ \tau_Q)$ and $\h(\omega_j)\le 2 j\,( \log(j) + \tau_Q)$, which proves the proposition.
\end{proof}

\end{document}